\documentclass[10pt]{amsart}
\usepackage{t1enc}
\usepackage[latin2]{inputenc}
\usepackage{amsmath}
\usepackage{amsthm}
\usepackage{amssymb}
\usepackage{stmaryrd}
\usepackage[T1]{fontenc}
\usepackage{enumerate}
\usepackage{verbatim}
\usepackage{amscd}
\usepackage{color}
\usepackage{wasysym}
\usepackage{url}
\usepackage[small,nohug,
]{diagrams}
\diagramstyle[labelstyle=\scriptstyle]

\newarrow{Corresponds} <--->

\newarrow{Dashto} {}{dash}{}{dash}>
\newarrow{Dotsto} ....>
\newarrow{Eqv} =====

\usepackage{pdflscape}

\usepackage[bitstream-charter]{mathdesign}
\usepackage{eucal}

\newif\ifcommented
\commentedtrue

\newcommand{\comm}[1]{}

\ifcommented
\renewcommand{\comm}[1]{\fbox{\fbox{\begin{minipage}{300pt}#1\end{minipage}}}}

\fi

\newtheorem*{claim}{Claim}

\newtheorem{Thm}{Theorem}[section]
\newtheorem{thm}[Thm]{Theorem}
\newtheorem{Prop}[Thm]{Proposition}
\newtheorem{Lem}[Thm]{Lemma}
\newtheorem{lem}[Thm]{Lemma}

\newtheorem{cor}[Thm]{Corollary}
\newtheorem{prop}[Thm]{Proposition}

\newtheorem{Fact}[Thm]{Fact}
\newtheorem{fact}[Thm]{Fact}

\theoremstyle{definition}

\newtheorem{rem}[Thm]{Remark}

\newtheorem{que}[Thm]{Question}

\newtheorem{definition}[Thm]{Definition}

\newtheorem{exa}[Thm]{Example}

\newtheorem{model}[Thm]{Model}

\newcommand{\mc}{\mathcal}
\newcommand{\mf}{\protect\mathfrak}
\newcommand{\mbb}{\mathbb}

\newcommand{\mrm}{\mathrm}
\newcommand{\PP}{\mathbb{P}}

\newcommand{\0}{\emptyset}
\newcommand{\bs}{\backslash}
\newcommand{\vd}{\Vdash}

\newcommand{\fv}{\rightarrow}

\newcommand{\al}{\alpha}
\newcommand{\be}{\beta}
\newcommand{\ga}{\gamma}
\newcommand{\de}{\delta}
\newcommand{\eps}{\varepsilon}
\newcommand{\lam}{\lambda}
\newcommand{\ka}{\kappa}

\newcommand{\om}{\omega}
\newcommand{\bc}{\begin{center}}
\newcommand{\ec}{\end{center}}
\newcommand{\fel}{\ge}
\newcommand{\QQ}{\dot{\mbb{Q}}}

\newcommand{\I}{\mathcal{I}}
\newcommand{\omoms}{[\omega]^\omega}
\newcommand{\cc}{\mathfrak{c}}
\newcommand{\MM}{\mathbb{M}}

\newcommand{\cf}{\mathop{\mathrm{cf}}\nolimits}
\newcommand{\dom}{\mathop{\mathrm{dom}}\nolimits}
\newcommand{\ran}{\mathop{\mathrm{ran}}\nolimits}

\newcommand{\cof}{\mathop{\mathrm{cof}}\nolimits}
\newcommand{\add}{\mathop{\mathrm{add}}\nolimits}
\newcommand{\cov}{\mathop{\mathrm{cov}}\nolimits}
\newcommand{\non}{\mathop{\mathrm{non}}\nolimits}

\newcommand{\ical}{{\mathcal{I}}}

\makeindex

\newif\ifdeveloping

\developingtrue

\def\myheads#1;#2;{
\pagestyle{myheadings}
\markboth{{\sc\hfill #1\hfill\protect\makebox[0cm][r]{\rm\today}}}
{{\sc\protect\makebox[0cm][l]{\rm\today}\hfill #2\hfill}}
}

\def\<{\left\langle}

\def\>{\right\rangle}
\def\cf{\operatorname{cf}}

\def\br#1;#2;{\bigl[ {#1} \bigr]^ {#2} }

\def\Ubf#1{{\baselineskip=0pt\vtop{\hbox{$#1$}\hbox{$\sim$}}}{}}

\begin{document}

\title{Towers in filters, cardinal invariants, and Luzin type families}
\thanks{The first author was supported by Grants-in-Aid for Scientific Research (C) 21540128, (C) 24540126 and (C) 15K04977, Japan Society for the Promotion of Science, by JSPS and FWF under the Japan-Austria Research Cooperative Program {\em New developments regarding forcing in set theory}, by the IMS program {\em Sets and Computations}, National University of Singapore, and by Michael Hru\v{s}\'ak's grants, CONACyT grant no. 177758  and PAPIIT grant IN-108014, during his stay at UNAM
in spring 2015. The second author was supported by the Austrian
Science Fund (FWF) grant no. P25671, by the IMS program Sets and Computations, National University of Singapore, and by the Hungarian National Foundation for Scientific Research grant nos.
68262, 83726 and 77476. The third author was supported by joint FWF-GA\v{C}R grant no. I 1921-N25, The continuum, forcing, and large cardinals and by the IMS program Sets and Computations, National University of Singapore.}

\author[J\"org Brendle]{J\"org Brendle}
\address{Graduate School of System Informatics, Kobe University, Japan}
\email{brendle@kurt.scitec.kobe-u.ac.jp}

\author[Barnab\'as Farkas]{Barnab\'as Farkas}
\address{Kurt G\"odel Research Center, University of Vienna, Austria}
\email{barnabasfarkas@gmail.com}

\author[Jonathan Verner]{Jonathan Verner}
\address{Charles University, Prague, Czech Republic}
\email{jonathan.verner@matfyz.cz}

\subjclass[2010]{03E05, 03E17, 03E35} \keywords{Borel ideal, analytic ideal, summable ideal, density zero ideal,
tower, Mathias-Prikry forcing, Laver-Prikry forcing, cardinal invariants of the continuum, Kat\v{e}tov order, Kat\v{e}tov-Blass order, Luzin type family, diamond, Near Coherence of Filters}

\begin{abstract}
We investigate which filters on $\om$ can contain towers, that is, a modulo finite descending sequence without any pseudointersection (in $[\om]^\om$). We prove the following results:
\begin{itemize}
\item[(1)] Many classical examples of nice tall filters contain no towers (in $\mrm{ZFC}$).
\item[(2)] It is consistent that tall analytic P-filters contain towers of arbitrary regular height (simultaneously for many regular cardinals as well).
\item[(3)] It is consistent that all towers generate non-meager filters, in particular (consistently) Borel filters do not contain towers.
\item[(4)] The statement ``Every ultrafilter contains towers.'' is independent of $\mrm{ZFC}$.
\end{itemize}
Furthermore, we study many possible logical (non)implications between the existence of towers in filters, inequalities between cardinal invariants of filters ($\mrm{add}^*(\mc{F})$, $\mrm{cof}^*(\mc{F})$, $\mrm{non}^*(\mc{F})$, and $\mrm{cov}^*(\mc{F})$), and the existence of Luzin type families (of size $\geq \om_2$), that is, if $\mc{F}$ is a filter then $\mc{X}\subseteq [\om]^\om$ is an $\mc{F}$-Luzin family if $\{X\in\mc{X}:|X\setminus F|=\om\}$ is countable for every $F\in\mc{F}$.
\end{abstract}

\maketitle

\section{Introduction}
Let us denote $\mathrm{Fin}$ the {\em Fr\'echet ideal} on ${\omega}$, i.e. $\mathrm{Fin}=[\om]^{<\om}$. We always assume that if $\ical$ is
an ideal on $\om$ (or on any countable infinite set), then it is {\em proper}, i.e. $\om\notin
\ical$, and $\mathrm{Fin}\subseteq\ical$. Write $\mc{I}^+=\mc{P}(\om)\bs \mc{I}$ for the family of {\em $\mc{I}$-positive} sets and
$\mc{I}^*=\{\om\bs X:X\in\mc{I}\}$ for the {\em dual filter} of $\mc{I}$. If $X\in\mc{I}^+$ then we can talk about the {\em restriction} of $\mc{I}$ to $X$, that is $\mc{I}\upharpoonright X=\{A\in\mc{I}:A\subseteq X\}$, an ideal on $X$. We will use the same notations for filters as well, e.g. $\mc{F}^+=(\mc{F}^*)^+=\{X\subseteq\om:\forall$ $F\in\mc{F}$ $|X\cap F|=\om\}$.

An ideal $\ical$ on $\om$ is {\em $F_\sigma$, Borel, analytic, meager} etc if
$\ical\subseteq\mc{P}(\om)\simeq 2^\om$ is an $F_\sigma$, Borel, analytic, meager etc set in the
usual (compact Polish) topology of the Cantor-set. $\ical$ is a {\em
P-ideal} if for each countable $\mc{C}\subseteq\mc{I}$ there is an
$A\in\mc{I}$ such that $C\subseteq^* A$ for each $C\in\mc{C}$,
where $C\subseteq^* A$ iff $C\bs A$ is finite.  $\ical$ is {\em tall} (or {\em dense}) if each infinite subset of $\om$ contains
an infinite element of $\ical$.

\smallskip
A sequence $\mc{T}=(T_\al)_{\al<\ka}$ in $[\om]^\om$ is a {\em tower}
if $\ka$ is a cardinal, $\mc{T}$ is $\subseteq^*$-descending, i.e. $T_\be\subseteq^* T_\al$
if $\al\le\be<\ka$, and it is not {\em diagonalizable}, i.e. it
has no {\em pseudointersection}, that is, there is no $X\in [\om]^\om$ such
that $X\subseteq^* T_\al$ for each $\al<\ka$. The {\em tower number} $\mf{t}$ is the minimum of cardinalities of towers.

Let $\mc{F}$ be a
filter on $\om$. A tower $(T_\al)_{\al<\ka}$ is a {\em tower in
$\mc{F}$} if $T_\al\in\mc{F}$ for each $\al<\ka$. Similarly, we can talk about {\em cotowers} in ideals, that is, a sequence $(A_\al)_{\al<\ka}$ in the ideal such that $(\om\setminus A_\al)_{\al<\ka}$ is a tower in the dual filter of the ideal.

\smallskip
In this paper, we investigate which filters (ideals) contain (co)towers. Recall that the {\em pseudointersection number}, $\mf{p}$ is the least cardinality
of a family $\mc{H}\subseteq [\om]^\om$ such that (a) $\mc{H}$ has the {\em strong finite intersection property} (sfip), that is, $\bigcap\mc{H}'$ is infinite for every nonempty finite $\mc{H}'\subseteq\mc{H}$, and (b) $\mc{H}$ has no pseudointersection. In other words, $\mc{H}$ has the sfip iff the family
\[ \mrm{fr}(\mc{H})=\Big\{X\subseteq\om:\exists\;\mc{H}'\in [\mc{H}]^{<\om}\setminus\{\0\}\;\bigcap\mc{H}'\subseteq^* X\Big\}\]
is a tall filter (i.e. its dual ideal is tall). $\mrm{fr}(\mc{H})$ is called the filter {\em generated by $\mc{H}$}. Clearly, $\mf{p}\leq\mf{t}$. The following theorem of M. Malliaris and S. Shelah resolved a long-standing open problem:

\begin{thm} {\em (see \cite{tp})}
$\mathfrak p=\mathfrak t$.
\end{thm}

In view of this result, one may ask whether every family of infinite subsets of $\om$ witnessing $\mathfrak p$ and closed for finite intersections has a subfamily witnessing $\mathfrak t$; or simply whether every tall filter contains a tower. We know that the answer is no (see Section \ref{examp} for such examples). We will investigate which Borel or maximal filters contain towers.

In Section \ref{examp} we present a plethora of nice (that is, Borel and projective) ideals, and show that many of them cannot contain cotowers.

In Section \ref{analp} we show that consistently there are cotowers of arbitrary regular heights (simultaneously for many regular cardinals) in tall analytic P-ideals.

These results motivate the following:

\begin{que}\label{mainq}
Does there exist a Borel ideal $\mc{I}$ which contains a cotower (in $\mrm{ZFC}$)? (See also \cite[Prob. 63]{BoCh}.) Or at least, does there exist a Borel ideal $\mc{I}$ such that $\mc{I}$ does not contain tall analytic P-ideals (this property is $\Ubf{\Pi}^1_2$ hence absolute between $V$ and $V^\PP$) but we can force a cotower into $\mc{I}$?
\end{que}

In Section \ref{killing}, first we prove that consistently all towers generate non-meager filters, in particular, answering the first part of the question above, Borel ideals contain no cotowers in this model. Then using the filter based Mathias-Prikry and Laver-Prikry forcings, we present a more subtle way of destroying cotowers in $F_\sigma$ ideals and analytic P-ideals.

In Section \ref{tuf} we show that consistently there is an ultrafilter which does not contain towers, moreover if $\mf{c}>\om_1$ then this ultrafilter can be chosen as selective. Furthermore, applying the axiom $\mrm{NCF}$ ({\em Near Coherence of Filters}), we show that consistently every ultrafilter contains towers.

In Section \ref{luzinsec} we introduce the notion of $\mc{I}$-Luzin families and $\mc{I}$-inaccessibility (of $[\om]^\om$), and study connections between existence of (large) $\mc{I}$-Luzin families, strict inequalities between cardinal invariants, $\mc{I}$-inaccessibility, and existence of cotowers in ideals. Also, because of its impact on these notions, we analyze all possible Kat\v{e}tov-Blass reducibilities between our examples.

In Section \ref{countnon} we study six of our examples ($\mc{ED}$, $\mrm{Ran}$, $\mc{S}$, $\mrm{Nwd}$, $\mrm{Conv}$, and $\mrm{Fin}\otimes\mrm{Fin}$) satisfying the property $\mrm{non}^*(\mc{I})=\om<\mrm{cov}^*(\mc{I})$, in particular, we are interested in all possible cuts of the diagram of logical implications from Section \ref{luzinsec} in the case of these examples.

In Section \ref{edfin} we study $\mc{ED}_\mrm{fin}$ and consistent cuts of the diagram from Section \ref{luzinsec}.

In Section \ref{analpindiag} we study consistent cuts of the diagram from Section \ref{luzinsec} in the case of analytic P-ideals, especially in the case of our three main examples ($\mc{I}_{1/n}$, $\mrm{tr}(\mc{N})$, and $\mc{Z}$).

And finally, in Section \ref{questions}, we list some of our related questions (additionally to the ones already stated in the previous sections), and we present a partial result concerning the consistency of (e.g.) ``$\mrm{tr}(\mc{N})$ contains a cotower but $\mc{I}_{1/n}$ does not''.

\subsection*{Acknowledgement}
The authors would like to thank the organizers of the {\em Sets and Computations} conference (March 30 -- April 30, 2015) held at Institute for Mathematical Sciences of the National University of Singapore for their hospitality and for providing a great environment for scientific activities.

\section{Examples of Borel and projective ideals}\label{examp}
Let us present some important examples of nice ideals (for results on their role in combinatorial set theory see e.g. \cite{Meza} or \cite{hrusaksummary}):

\smallskip
\textbf{Some $F_\sigma$ ideals:}

\smallskip
{\em Summable ideals.} Let $h:\om\to [0,\infty)$ be a function such that $\sum_{n\in\om}
 h(n)=\infty$. The {\em summable ideal associated to $h$} is
\[ \ical_h=\bigg\{A\subseteq\om:\sum_{n\in A} h(n)<\infty\bigg\}.\]

It is easy to see that a summable ideal $\ical_h$ is tall iff
$\lim_{n\fv\infty}h(n)=0$, and that summable ideals are $F_\sigma$ P-ideals.
The {\em classical summable ideal} is $\mc{I}_{1/n}=\mc{I}_h$ where $h(n)=1/(n+1)$, or $h(0)=1$ and $h(n)=1/n$ if $n>0$.
We know that there are tall $F_\sigma$ P-ideals which are not summable ideals, Farah's example (see \cite[Example 1.11.1]{Farah}) is the following ideal:
\[ \mc{J}_F=\bigg\{A\subseteq\om:\sum_{n<\om}\frac{\min\big\{n,|A\cap [2^n,2^{n+1})|\big\}}{n^2}<\infty\bigg\}.\]

The {\em eventually different ideals.} \[\mc{ED}=\Big\{A\subseteq\om\times\om:\limsup_{n\fv\infty}|(A)_n|<\infty\Big\}\]
where $(A)_n=\{k\in\om:(n,k)\in A\}$, and $\mc{ED}_\mathrm{fin}=\mc{ED}\upharpoonright\Delta$ where $\Delta=\{(n,m)\in\om\times\om:m\le n\}$. $\mc{ED}$ and $\mc{ED}_\mathrm{fin}$ are tall non P-ideals.

\smallskip
{\em Fragmented ideals.} Let $(P_n)_{n\in\om}$ be a partition of $\om$ into finite sets and $\vec{\varphi}=(\varphi_n)_{n\in\om}$ be a sequence of submeasures (see below), $\varphi_n:\mc{P}(P_n)\to [0,\infty)$ satisfying $\sup\{\varphi_n(P_n):n\in\om\}=\infty$. Then the {\em fragmented ideal generated by $\vec{\varphi}$} is the following ideal (see \cite{hrufrag}):
\[ \mc{I}(\vec{\varphi})=\big\{A\subseteq\om:\sup\big\{\varphi_n(A\cap P_n):n\in\om\big\}<\infty\big\}.\]
Notice that for example $\mc{ED}_\mrm{fin}$ is a fragmented ideal: Let $P_n=\{(n,m):m\leq n\}$ and define $\varphi(A)=|A|$ for $A\subseteq P_n$. It is easy to see that $\mc{I}(\vec{\varphi})$ is tall iff $\sup\{\varphi_n(\{k\}):n\in\om,k\in P_n\}<\infty$; and that tall fragmented ideals are not P-ideals.

\smallskip
The {\em van der Waerden ideal:}
\[ \mc{W}=\big\{A\subseteq\om:A\;\text{does not contain arbitrary long arithmetic progressions}\big\}.\]
Van der Waerden's well-known theorem says that $\mc{W}$ is a proper ideal. $\mc{W}$ is a tall non P-ideal.
Szemer\'edi's famous theorem says that $\mc{W}\subseteq\mc{Z}=\{A\subseteq\om:|A\cap n|/n\to 0\}$ (see \cite{szemeredi}). The stronger statement $\mc{W}\subseteq\mc{I}_{1/n}$ is a still open Erd\H{o}s prize problem ($\$3000$). For some interesting set-theoretic results about this ideal see e.g. \cite{jana1} and \cite{jana2}.

\smallskip
The {\em random graph ideal:}
\[ \mathrm{Ran}=\mathrm{id}\big(\big\{\text{homogeneous subsets of the random graph}\big\}\big).\]
where the {\em random graph} $(\om,E)$, $E\subseteq [\om]^2$ is up to isomorphism uniquely determined by the following property: If $A,B\in[\om]^{<\om}$ are nonempty and disjoint, then there is an $n\in\om\setminus (A\cup B)$ such that $\{\{n,a\}:a\in A\}\subseteq E$ and $\{\{n,b\}:b\in B\}\cap E=\0$. A set $H\subseteq\om$ is ($E$-){\em homogeneous} iff $[H]^2\subseteq E$ or $[H]^2\cap E=\0$. $\mathrm{Ran}$ is a tall non P-ideal.

\smallskip
The {\em ideal of graphs with finite chromatic number}:
\[\mc{G}_\mathrm{fc}=\big\{E\subseteq [\om]^2:\chi(\om,E)<\om\big\}.\]
It is a tall non P-ideal.

\smallskip
{\em Solecki's ideal:} Let $\mathrm{CO}(2^\om)$ be the family of clopen subsets of $2^\om$ (clearly $|\mathrm{CO}(2^\om)|=\om$), and let $\Omega=\{A\in\mathrm{CO}(2^\om):\lam(A)=1/2\}$ where $\lam$ is the usual product measure on $2^\om$. The ideal $\mc{S}$ on $\Omega$ is generated by $\{I_x:x\in 2^\om\}$ where $I_x=\{A\in\Omega:x\in A\}$.
$\mc{S}$ is a tall non P-ideal.

\smallskip
\textbf{Some $F_{\sigma\delta}$ ideals:}

\smallskip
{\em Density ideals.}
Let $(P_n)_{n\in\omega}$ be a sequence of pairwise disjoint finite subsets of $\om$ and let $\vec{\mu}=(\mu_n)_{n\in\om}$ be a
sequences of measures, $\mu_n$ concentrated on $P_n$ such that $\limsup_{n\fv\infty}\mu_n(P_n)>0$. The {\em density ideal generated by $\vec{\mu}$}
is
\[ \mc{Z}_{\vec{\mu}}=\Big\{A\subseteq\om:\lim_{n\fv\infty}\mu_n(A)=0\Big\}.\]
A density ideal $\mc{Z}_{\vec{\mu}}$ is tall iff $\max\{\mu_n(\{i\}):i\in P_n\}\xrightarrow{n\fv\infty}0$, and density ideals are $F_{\sigma\delta}$ P-ideals.
The {\em density zero ideal} $\mc{Z}=\big\{A\subseteq\om:\lim_{n\fv\infty}|A\cap n|/n=0\big\}$ is a tall density ideal because $\mc{Z}=\mc{Z}_{\vec{\mu}}$ where $P_n=[2^n,2^{n+1})$ and $\mu_n(A)=|A\cap P_n|/2^n$. It is easy to see that $\mc{I}_{1/n}\subsetneq\mc{Z}$.

\smallskip
{\em Generalized density ideals.}
Like in the case of density ideals we fix a partition $(P_n)_{n\in\omega}$ of $\om$ into finite sets but we allow $\varphi_n:\mc{P}(P_n)\to [0,\infty)$ to be a submeasure (see below) for every $n$ satisfying $\limsup_{n\fv\infty}\varphi_n(P_n)>0$. Then the {\em generalized density ideal generated by $\vec{\varphi}$} (where $\vec{\varphi}=(\varphi_n)_{n\in\om}$)
is $\mc{Z}_{\vec{\varphi}}=\{A\subseteq\om:\lim_{n\fv\infty}\varphi_n(A)=0\}$. If $\sup\{\varphi_n(P_n):n\in\om\}=\infty$, then  $\mc{Z}_{\vec{\varphi}}\subseteq\mc{I}(\vec{\varphi})$. A generalized density ideal $\mc{Z}_{\vec{\varphi}}$ is tall iff $\max\{\varphi_n(\{i\}):i\in P_n\}\xrightarrow{n\fv\infty}0$; and these ideals are $F_{\sigma\delta}$ P-ideals.

\smallskip
The {\em uniform density zero ideal.} For $A\subseteq\om$ and $n\in\om$ let
$S_n(A)=\max\{|A\cap[k,k+n)|:k\in\om\}$ and let
\[ \mathcal{Z}_u=\Big\{A\subseteq\om:\lim_{n\to\infty}\frac{S_n(A)}{n}=0\Big\}.\]
It is easy to see that $\mathcal{Z}_u$ is a tall $F_{\sigma\delta}$ non P-ideal (for more details, see \cite{BFMS}). Notice that Szemer\'edi's theorem actually claims that $\mc{W}\subseteq \mc{Z}_u$. It is straightforward to see that $\mc{Z}_u\ne \mc{W}$, $\mc{Z}_u\subsetneq\mathcal{Z}$, and that there are no inclusions between   $\mc{Z}_u$ and $\mc{I}_{1/n}$.

\smallskip
The {\em ideal of nowhere dense subsets of the rationals:} \[ \mathrm{Nwd}=\big\{A\subseteq\mbb{Q}:\mathrm{int}(\overline{A})=\0\big\}\]
where $\mathrm{int}(\cdot)$ stands for the interior operation on subsets of the reals, and $\overline{A}$ is the closure of $A$ in $\mbb{R}$. $\mathrm{Nwd}$ is a tall non P-ideal.

\smallskip
The {\em trace ideal of the null ideal}: Let $\mc{N}$ be the $\sigma$-ideal of subsets of $2^\om$ with measure zero (with respect to the usual product measure). The {\em $G_\delta$-closure} of a set $A\subseteq 2^{<\om}$ is $[A]=\big\{x\in 2^\om:\exists^\infty$ $n$ $x\upharpoonright n\in A\big\}$, a $G_\delta$ subset of $2^\om$. The trace of $\mc{N}$ is defined by
\[ \mathrm{tr}(\mc{N})=\big\{A\subseteq 2^{<\om}:[A]\in \mc{N}\big\}.\]
It is a tall $F_{\sigma\delta}$ P-ideal.

\smallskip
\textbf{Some tall $F_{\sigma\delta\sigma}$ (non P-)ideals:}

\smallskip
The ideal $\mathrm{Conv}$ is generated by those infinite subsets of $\mbb{Q}\cap [0,1]$ which are convergent in $[0,1]$, in other words
\[ \mathrm{Conv}=\big\{A\subseteq \mbb{Q}\cap [0,1]:|\text{accumulation points of $A$ (in}\;\mbb{R})|<\om\big\}.\]

The Fubini product of $\mathrm{Fin}$ by itself:
\[ \mathrm{Fin}\otimes\mathrm{Fin}=\big\{A\subseteq\om\times\om:\forall^\infty\;n\in\om\;|(A)_n|<\om\big\}.\]

\textbf{Some non tall ideals:}

\smallskip
An important $F_\sigma$ ideal:
\[ \mathrm{Fin}\otimes\{\0\}=\big\{A\subseteq\om\times\om:\forall^\infty\;n\in\om\;(A)_n=\0\big\},\]
\indent and its $F_{\sigma\delta}$ brother (a density ideal):
\[ \{\0\}\otimes\mathrm{Fin}=\big\{A\subseteq\om\times\om:\forall\;n\in\om\;|(A)_n|<\om\big\}.\]

It is easy to see that there are no $G_\delta$ (i.e. $\Ubf{\Pi}^0_2$) ideals, and we know that there are many $F_\sigma$ (i.e. $\Ubf{\Sigma}^0_2$) ideals. In general, we know the following:
\begin{thm} {\em (see \cite{Cal85} and \cite{Cal88})} There are $\Ubf{\Sigma}^0_\al$- and $\Ubf{\Pi}^0_\al$-complete ideals for every $\al\geq 3$.
\end{thm}

About ideals on the ambiguous levels of the Borel hierarchy see \cite{Eng}.

\smallskip
We also present some (co)analytic examples.

\begin{thm} {\em (see \cite[page 321]{zaf})}
For every $x\in\om^\om$ let $I_x=\{s\in \om^{<\om}:x\upharpoonright |s|\nleq s\}$ where $\leq$ is the coordinatewise ordering on every $\om^n$. Then the ideal on $\om^{<\om}$ generated by $\{I_x:x\in\om^\om\}$ is $\Ubf{\Sigma}^1_1$-complete.
\end{thm}

\begin{thm} {\em (see \cite{adrmix})}
The ideal of graphs without infinite complete subgraphs,
\[ \mc{G}_\mathrm{c}=\big\{E\subseteq [\om]^2:\forall\;X\in [\om]^\om\;[X]^2\nsubseteq E\big\}\]
is $\Ubf{\Pi}^1_1$-complete (in $\mc{P}([\om]^2)$), tall, and a non P-ideal.
\end{thm}

\begin{thm} {\em (see \cite{adrmix})} The ideal
\[ \big\{A\subseteq\om\times\om:\forall\;X,Y\in [\om]^\om\;\exists\;X'\in [X]^\om\;\exists\;Y'\in [Y]^\om\;A\cap (X'\times Y')=\0\big\}\]
is $\Ubf{\Pi}^1_1$-complete (in $\mc{P}(\om\times\om)$), tall, and a non P-ideal.
\end{thm}

\begin{thm} {\em (see \cite{adrmix})} For every natural number $n>0$, there are $\Ubf{\Sigma}^1_n$- and $\Ubf{\Pi}^1_n$-complete tall ideals on $\om$.
\end{thm}

There is a natural way of defining nice ideals on $\om$.
A function $\varphi:\mc{P}({\omega})\to [0,\infty]$
is a {\em submeasure on $\om$} if $\varphi(\0)=0$; ${\varphi}(X)\le \varphi(X\cup Y)\le \varphi(X)+\varphi(Y)$ for every $X,Y\subseteq\om$; and $\varphi(\{n\})<\infty$ for every $n\in {\omega}$. $\varphi$ is {\em lower semicontinuous} (lsc, in short) if $\varphi(X)=\lim_{n\fv\infty}\varphi(X\cap n)$ for each $X\subseteq\om$. $\varphi$ is {\em finite} if $\varphi(\om)<\infty$.

If $\varphi$ is an lsc submeasure on $\om$ then for $X\subseteq\om$ let $\|X\|_\varphi=\lim_{n\fv\infty}\varphi(X\bs n)$.
We assign two ideals to a submeasure $\varphi$ as follows
\begin{align*}
\mathrm{Fin}(\varphi) & =  \big\{X\subseteq\om:\varphi(X)<\infty\big\},\\
\mathrm{Exh}(\varphi) & =  \big\{X\subseteq {\omega}:\|X\|_\varphi=0\big\}.
\end{align*}
It is easy to see that if $\mathrm{Fin}(\varphi)\ne\mc{P}(\om)$, then it is an $F_\sigma$ ideal; and similarly if $\mathrm{Exh}(\varphi)\ne\mc{P}(\om)$, then it is an $F_{\sigma\delta}$ P-ideal. Clearly, $\mathrm{Exh}(\varphi)\subseteq\mathrm{Fin}(\varphi)$ always holds.
From now on, if we are working with $\mathrm{Fin}(\varphi)$ or $\mathrm{Exh}(\varphi)$, then we always assume that they are proper ideals.
It is straightforward to see that if $\varphi$ is an lsc submeasure on $\om$ then $\mathrm{Exh}(\varphi)$ is tall iff $\lim_{n\fv\infty}\varphi(\{n\})=0$.

The following characterization theorem gives us the most important tool for working on combinatorics of $F_\sigma$ ideals and analytic P-ideals.

\begin{Thm}{\em (\cite{Mazur} and \cite{So})}\label{char} Let $\mc{I}$ be an ideal on $\om$.
\begin{itemize}
\item $\mc{I}$ is an $F_\sigma$ ideal iff $\mc{I}=\mathrm{Fin}(\varphi)$ for some lsc submeasure $\varphi$.
\item $\mc{I}$ is an analytic $P$-ideal iff $\mc{I}=\mathrm{Exh}(\varphi)$ for some (finite) lsc submeasure $\varphi$.
\item $\mc{I}$ is an $F_\sigma$ P-ideal iff $\mc{I}=\mathrm{Fin}(\varphi)=\mathrm{Exh}(\varphi)$ for some lsc submeasure $\varphi$.
\end{itemize}
\end{Thm}

In particular, each analytic P-ideal is $F_{\sigma\delta}$.

\smallskip
We recall the classical cardinal invariants of an arbitrary ideal $I$ on a set $X$:
\begin{align*}
\mathrm{add}(I) & = \min\big\{|J|:J\subseteq I\;\text{and}\;\bigcup J\notin I\big\}\\
\mathrm{non}(I) & = \min\big\{|Y|:Y\subseteq X\;\text{and}\;Y\notin I\big\}\\
\mathrm{cov}(I) & = \min\big\{|C|:C\subseteq I\;\text{and}\;\bigcup C=X\big\}\\
\mathrm{cof}(I) & = \min\big\{|D|:D\;\text{is cofinal in}\;(I,\subseteq)\big\}
\end{align*}

We know that $\mathrm{add}(I)\le\mathrm{non}(I),\mathrm{cov}(I)\le\mathrm{cof}(I)$, for more details about classical cardinal invariants see \cite{Bl}.

The set $[\om]^\om$ can be seen as a Polish space because it is a $G_\delta$ subset of $\mc{P}(\om)\simeq 2^\om$, moreover, it is easy to see that $[\om]^\om$ is homeomorphic to $\om^{\uparrow\om}=\{x\in\om^\om:x$ is strictly increasing$\}$, a closed subset of $\om^\om$ which is homeomorphic to $\om^\om$.

There is a natural way of constructing ideals on $[\om]^\om$ from tall ideals on $\om$ (see \cite{cardinvanalp}):
If $\mc{I}$ is a tall ideal on $\om$, let $\widehat{\mc{I}}$ be generated by all sets of the form $\widehat{A}=\{X\in [\om]^\om:|X\cap A|=\om\}$ (a $G_\delta$ subset of $[\om]^\om$) where $A\in\mc{I}$. The {\em star-invariants} of a tall ideal $\mc{I}$ on $\om$ are defined as follows:
\begin{align*}
\mrm{add}^*(\mc{I})=\mrm{add}(\widehat{\mc{I}}) & = \min\big\{|\mc{A}|:\mc{A}\subseteq\mc{I}\;\text{and}\;\mc{A}\;\text{is}\;\subseteq^*\text{-unbounded in}\;\mc{I}\big\}\\
\mrm{non}^*(\mc{I})=\mrm{non}(\widehat{\mc{I}}) & = \min\big\{|\mc{X}|:\mc{X}\subseteq [\om]^\om\;\;\text{and}\;\;\forall\;A\in\mc{I}\;\exists\;X\in\mc{X}\;|A\cap X|<\om\big\}\\
\mrm{cov}^*(\mc{I})=\mrm{cov}(\widehat{\mc{I}}) & =\min\big\{|\mc{D}|:\mc{D}\subseteq \mc{I}\;\;\text{and}\;\;\forall\;X\in [\om]^\om\;\exists\;D\in\mc{D}\;|X\cap D|=\om\big\}\\
\mrm{cof}^*(\mc{I})=\mrm{cof}(\widehat{\mc{I}}) & =\min\big\{|\mc{C}|:\mc{C}\subseteq\mc{I}\;\text{and}\;\mc{C}\;\text{is}\;\subseteq^{(*)}\text{-cofinal in}\;\mc{I}\big\}
\end{align*}
Clearly, $\mc{I}$ is a P-ideal iff $\add^*(\mc{I})>\om$. In the forthcoming sections, we will recall some known values of these coefficients in the case of our examples. At this point, we focus on their effects on the existence and possible lengths of cotowers in the ideal. It is trivial to see that if there is a cotower in an ideal $\mc{I}$ then it must be tall.

\begin{fact}\label{fact1} Let $\mc{I}$ be a tall ideal.
\begin{itemize}
\item[(a)] If $\add^*(\mc{I})=\cov^*(\mc{I})=\ka$ then there is a cotower in $\mc{I}$ of height $\ka$.
\item[(b)] If there is a cotower in $\mc{I}$ of length $\lam$, then $\cov^*(\mc{I})\leq\mrm{cf}(\lam)\leq \mathrm{non}^*(\mc{I})$ (in particular, $\mf{t}\leq\mrm{non}^*(\mc{I})$).
\item[(c)] If $\mc{I}$ is a P-ideal then $\mrm{cov}^*(\mc{I})=\om_1$ iff there is a cotower in $\mc{I}$ of height $\om_1$.
\end{itemize}
\end{fact}
\begin{proof}
(a): Fix a family $\{D_\al:\al<\ka\}\subseteq\mc{I}$ witnessing $\cov^*(\mc{I})=\ka$. Applying $\add^*(\mc{I})=\cov^*(\mc{I})$, by recursion on $\al<\ka$ we can pick $A_\al\in\mc{I}$ such that $A_\be,D_\be\subseteq^* A_\al$ for every $\be\leq\al<\ka$. Then $(A_\al)_{\al<\ka}$ is a cotower in $\mc{I}$.

(b): $\cov^*(\mc{I})\leq \mrm{cf}(\lam)$ is trivial because if a cofinal subsequence of a cotower is of length $<\cov^*(\mc{I})$ then there is an infinite $X\subseteq\om$ which has finite intersection with all elements of the cotower, i.e. $\om\setminus X$ is a pseudounion of the cotower, a contradiction.

Now assume on the contrary that $(T_\xi)_{\xi<\lam}$ is a cotower in $\mc{I}$ and the family $\{X_\al:\al<\ka\}\subseteq [\om]^\om$ witnesses $\mathrm{non}^*(\mc{I})=\ka<\mrm{cf}(\lam)$. We can assume that $|\om\setminus X_\al|=\om$ for every $\al$.  For each $\xi$ we can pick an $\al_\xi<\ka$ such that $|T_\xi\cap X_{\al_\xi}|<\om$, then $\al_\xi=\al$ for $\lam$ many $\xi$, and hence $T_\xi\subseteq^* \om\setminus X_\al$ for each $\xi$, a contradiction.

(c) follows from (a) and (b).
\end{proof}

The star-uniformity of numerous classical Borel ideals is equal to $\om$ so there are no cotowers in these ideals. Among our examples, the following have countable star-uniformity:

$\mc{ED}$ because of the columns of $\om\times\om$. In general, all ideals of the form $\mc{I}\otimes\mathrm{Fin}$ where $\mc{I}\otimes\mc{J}$ is the {\em Fubini-product} of $\mc{I}$ and $\mc{J}$, that is
\[ \mc{I}\otimes\mc{J}=\big\{A\subseteq\om\times\om:\{n\in\om:(A)_n\in\mc{J}\}\in\mc{I}^*\big\}.\]

The random graph ideal because of the following reason: Consider the graph $E=\{\{(n,m),(n,k)\}:n\in\om, m\ne k\}$ on $\om\times\om$. We know that the random graph contains an isomorphic copy of $E$, and clearly if a set has infinite intersection with the (images of the) columns, then it cannot be covered by finitely many homogeneous sets.

Solecki's ideal because of \cite[Thm. 1.6.2]{Meza}.

$\mrm{Nwd}$ (hence $\mrm{Conv}$ as well) because of any countable base of $\mbb{Q}$.

\smallskip
Clearly, if $\mc{I}$ and $\mc{J}$ are ideals, $\mc{I}$ contains cotowers, and $\mc{J}$ contains a copy of $\mc{I}$ (via a permutation), then $\mc{J}$ also contains cotowers. 
In general, we know a bit more. Let $\leq_\mrm{KB}$ be the {\em Kat\v{e}tov-Blass} and $\leq_\mrm{K}$ be the {\em Kat\v{e}tov preorder} on the family of ideals on $\om$, that is,
\begin{align*} \mc{I}\leq_{\mathrm{KB}}\mc{J}\;\;& \text{iff}\;\;\exists\;f:\om\xrightarrow{\text{fin-to-one}}\om\;\forall\;A\in\mc{I}\,\;f^{-1}[A]\in\mc{J},\\
\mc{I}\leq_{\mathrm{K}}\mc{J}\;\;&\text{iff}\;\;\exists\;f:\om\xrightarrow{\,\text{arbitrary}\,}\om\;\forall\;A\in\mc{I}\,\;f^{-1}[A]\in\mc{J}.
\end{align*}
Clearly, if $\mc{J}$ contains a copy of $\mc{I}$ then $\mc{I}\leq_\mrm{KB} \mc{J}$.

It is straightforward to see that if $\mc{I}$ and $\mc{J}$ are Borel then the statement ``$\mc{I}\leq_{\mrm{K(B)}}\mc{J}$'' is $\Ubf{\Sigma}^1_2$ hence absolute for transitive models $V\subseteq W$ satisfying $\om_1^W\subseteq V$.

\begin{fact}\label{kbtower}
Assume that $\mc{I}\leq_\mrm{KB}\mc{J}$ and that $\mc{I}$ contains a cotower. Then $\mc{J}$ also contains a cotower.
\end{fact}
\begin{proof}
We show that if $f:\om\to\om$ witnesses  $\mc{I}\leq_\mrm{KB}\mc{J}$ and
$(A_\al)_{\al<\ka}$ is a cotower in $\mc{I}$, then $(f^{-1}[A_\al])_{\al<\ka}$ is a cotower in $\mc{J}$. Clearly, this is a $\subseteq^*$-increasing sequence in $\mc{J}$. If $\al$ is large enough then $f^{-1}[A_\al]$ is infinite (i.e. $A_\al\cap\mrm{ran}(f)$ is infinite) otherwise $\om\setminus\ran(f)$ would be a pseudounion of $(A_\al)_{\al<\ka}$ (notice that $\ran(f)$ must be infinite).

Assume on the contrary that $f^{-1}[A_\al]\subseteq^* Y$ for some coinfinite  $Y\subseteq \om$. For every $n\in\ran(f)$ let $P_n=f^{-1}(\{n\})$ (a partition of $\om$ into nonempty finite sets) and let $Y'=\{n\in\ran(f):P_n\subseteq Y\}$. Then $Y'$ must be coinfinite in $\mrm{ran}(f)$. If $\al<\ka$ then $f^{-1}[A_\al]=\bigcup\{P_n:n\in A_\al\cap\ran(f)\}\subseteq^* Y$ therefore $P_n\subseteq Y$ for every large enough $n\in A_\al\cap\ran(f)$ and hence  $A_\al\subseteq^* Y'\cup (\om\setminus\mrm{ran}(f))$ for every $\al<\ka$, a contradiction.
\end{proof}

\begin{rem}
It is easy to show that if $\mc{I}=\mrm{Exh}(\varphi)$ is a tall analytic P-ideal and $\mc{I}\leq_\mrm{KB}\mc{J}$, then $\mc{J}$ contains a tall analytic P-ideal: If $f:\om\to\om$ witnesses this reduction and $\psi:\mc{P}(\om)\to [0,\infty]$, $\psi(X)=\varphi(f[X])$, then $\psi$ is an lsc submeasure, $\mrm{Exh}(\psi)$ is tall, and $\mrm{Exh}(\psi)\subseteq\mc{J}$.
\end{rem}

It is easy to see that the following are equivalent:
\begin{itemize}
\item[(i)] $\mc{ED}_\mathrm{fin}\le_\mathrm{KB}\mc{I}$;
\item[(ii)] $\mc{I}$ is not a {\em weak Q-ideal}: there is a partition $(P_n)_{n\in\om}$ of $\om$ into finite sets such that if $X\subseteq\om$ and $|X\cap P_n|\le 1$ for each $n$, then $X\in\mc{I}$.
\end{itemize}
In the case of Borel ideals, these properties are equivalent to (see \cite[Thm. 3.2.1]{Meza}):
\begin{itemize}
\item[(iii)] $\mathrm{non}^*(\mc{I})>\om$ (moreover, $\mathrm{non}^*(\mc{I})\geq\mathrm{non}^*(\mc{ED}_\mathrm{fin})\geq\mf{t}$).
\end{itemize}

We show a few more examples of ideals without cotowers in them.

\begin{exa}
There are no cotowers in $\mc{ED}_\mathrm{fin}$. Assume that $(A_\al)_{\al<\ka}$ is a $\subseteq^*$-increasing sequence in $\mc{ED}_\mathrm{fin}$ for some regular and uncountable $\ka$. For each $\al<\ka$ let $n_\al=\limsup_{n\fv\infty}|(A_\al)_n|<\om$.
Then $n_\al\leq n_\be$ for every $\al\leq\be<\ka$, therefore we can assume that $n_\al=N$ for every $\al$. Let
\[ B=\bigcup\big\{\{k\}\times (A_0)_k:|(A_0)_k|=N\big\}\cup\bigcup\big\{\{\ell\}\times (\ell+1):|(A_0)_\ell|<N\big\}.\]
It is easy to see that $|\Delta\setminus B|=\om$ and $A_\al\subseteq^* B$ for each $\al<\ka$.
\end{exa}

It is not difficult to see that $\mc{ED}_\mathrm{fin}\le_{\mrm{KB}}\mc{W}$ and $\mc{ED}_\mathrm{fin}\le_{\mrm{KB}}\mc{Z}_u$ hence the previous example follows from the following stronger ones. Actually, we also know that $\mc{W}\leq_\mrm{KB}\mc{Z}_u$ (because $\mc{W}\subseteq\mc{Z}_u$) but we do not want to use this difficult fact, we give a direct proof for $\mc{W}$ as well.

\begin{exa}\label{cotW}
There are no cotowers in $\mc{W}$.
Assume that $(A_\al)_{\al<\ka}$ is a $\subseteq^*$-increasing sequence in $\mc{W}$ for some regular and uncountable $\ka$. For every $X\subseteq\om$ let
\[ \mathrm{ap}(X)=\mathrm{sup}\big\{|A|:A\subseteq X\;\;\text{is an arithmetic progression}\big\}.\]
This invariant of a set cannot be controlled by $\subseteq^*$, we need a ``hereditary'' version of $\mathrm{ap}$: $\mathrm{ap}^*(X)=\mathrm{min}\{\mathrm{ap}(X\setminus n):n\in\om\}$. Clearly if $X\subseteq^* Y$ then $\mathrm{ap}^*(X)\le\mathrm{ap}^*(Y)$. Hence we can assume that there is an $N$ such that $\mathrm{ap}^*(A_\al)=N$ for every $\al<\ka$.

Pick $N$ long arithmetic progressions $P_k\subseteq A_0$ such that $\max(P_k)<\min(P_{k+1})$ ($k\in\om$). Then the set
\[B=\big\{n\in\om:\exists\;k\;\big(\max(P_k)<n\;\text{and}\;P_k\cup\{n\}\;\text{is an arithmetic progression}\big)\big\}\]
is infinite. It is easy to see that $B\cap A_\al$ is finite hence $A_\al\subseteq^* \om\setminus B$ for each $\al$.
\end{exa}

\begin{exa}
$\mc{Z}_u$ does not contain cotowers. Assume that $(A_\al)_{\al<\ka}$ is a $\subseteq^*$-increasing sequence in $\mc{Z}_u$ where $\ka$ is regular and uncountable. For every $\al$ define the function $s_\al\in\prod_{n\in\om}(n+1)$ as follows \[ s_\al(n)=\limsup_{k\to\infty}|A_\al\cap [k,k+n)|=\min\big\{S_n(A_\al\setminus K):K\in\om\big\}\leq S_n(A_\al).\]
Then $s_\al\leq s_\be$ whenever $\al\leq\be<\ka$. For every $n$ fix an $\al_n<\ka$ such that $s_{\al_n}(n)=\max\{s_\xi(n):\xi<\ka\}$, and let $\al=\sup\{\al_n:n\in\om\}<\ka$ (then $s_\xi\leq s_\al$ for every $\xi<\ka$). Pick an $N\in\om$ such that $s_\al(N)<N$ (almost every $N$ has this property because $A_\al\in\mc{Z}_u$) and define
\[ B=\bigcup\big\{[k,k+N)\setminus A_\al:k\in\om\;\text{and}\;|A_\al\cap [k,k+N)|=s_\al(N)\big\}.\]
It is easy to see that $B$ is infinite and $B\cap A_\xi$ is finite hence $A_\xi\subseteq^*\om\setminus B$ for every $\xi<\ka$.
\end{exa}

\begin{exa} $\mc{G}_\mrm{fc}$ does not contain cotowers. Let $(E_\al)_{\al<\ka}$ be a $\subseteq^*$-increasing sequence in $\mc{G}_\mrm{fc}$ for some $\ka=\mrm{cf}(\ka)>\om$. For every $E\subseteq [\om]^2$ define
\[ \chi^*(E)=\min\big\{\chi\big(\om,E\setminus [n]^2\big):n\in\om\big\}.\]
Clearly, $\chi^*(E_\al)$ is an increasing sequence in $\om$ hence we can assume that $\chi^*(E_\al)=N\in\om$ for every $\al<\ka$. Let $(P_n)_{n\in\om}$ be a partition of $\om$ into consecutive intervals of size $N+1$. Then $\forall$ $\al<\ka$ $\forall^\infty$ $n\in\om$ $[P_n]^2\nsubseteq E_\al$. For every $\al<\ka$ define $f_\al:\om\to \om$, $f_\al(n)=|E_\al\cap [P_n]^2|\leq \binom{N+1}{2}$. Then $f_\al\leq^* f_\be$ (i.e. $|\{n:f_\al(n)>f_\be(n)\}|<\om$) whenever $\al\leq \be<\ka$, and hence the sequence  $K_\al=\limsup_{n\to\infty}f_\al(n)<\binom{N+1}{2}$ is increasing. We can assume that $K_\al=K$ for every $\al$. We know that $X=\{n\in\om:|E_0\cap [P_n]^2|=K\}$ is infinite. For every $n\in X$, pick an $e_n\in [P_n]^2\setminus E_0$ and let $B=\{e_n:n\in X\}$.
It is easy to see that $E_\al\subseteq^* [\om]^2\setminus B$ for every $\al<\ka$.
\end{exa}

Among our examples, we still have to deal with tall fragmented ideals and tall analytic (i.e. $F_{\sigma\delta}$) P-ideals: $\mc{I}_h$, $\mc{J}_F$, $\mc{Z}_{\vec{\varphi}}$, and $\mrm{tr}(\mc{N})$. In Sections \ref{analp} and \ref{killing} we show that the existence of cotowers in tall analytic P-ideals is independent of $\mrm{ZFC}$, and that there may be no meager ideals which contain cotowers.

In particular, if a Borel ideal contains a tall analytic P-ideal, then consistently it contains a cotower. Unfortunately, we do not know (not even for $F_\sigma$ ideals) if it is a characterisation, more precisely, if containing a cotower in a forcing extension implies containing a tall analytic P-ideal. The following result shows that this characterisation holds at least for ``nice'' fragmented ideals.

\begin{prop}\label{fragdich}
Let $\mc{I}(\vec\mu)$ be a fragmented ideal based on a sequence $\vec\mu=(\mu_n)_{n\in\om}$ of measures with $\mu_n$ concentrated on $P_n$. Then one of the following holds:
\begin{itemize}
\item[(1)] There is an $X\in\mc{I}(\vec\mu)^*$ such that $\max\{\mu_n(\{k\}):k\in P_n\cap X\}\xrightarrow{n\to\infty} 0$, i.e. $\mc{Z}_{\vec\mu}\upharpoonright X$ is tall (and of course $\mc{I}(\vec\mu)\upharpoonright X\supseteq\mc{Z}_{\vec\mu}\upharpoonright X$).
\item[(2)] $\mc{I}(\vec\mu)$ does not contain cotowers.
\end{itemize}
\end{prop}
\begin{proof}
Assume that (1) does not hold. First we show that
\[\tag{$\star$} \forall\;r>0\;\exists\;\de>0\;\exists^\infty\;n\in\om\;\mu_n\big(\big\{k\in P_n:\mu_n(\{k\})\geq\de\big\}\big)\geq r.\]
Otherwise for some $r>0$
\[ \forall\;m\in\om\;\exists\;N_m\in\om\;\forall\;n\geq N_m \;\mu_n\big(\big\{k\in P_n:\mu_n(\{k\})\geq 2^{-m}\big\}\big)<r.\]
We can assume that $N_0<N_1<\dots$. Now let
\[ A=\bigcup_{m\in\om}\bigcup_{n\in N_{m+1}\setminus N_m}\big\{k\in P_n:\mu_n(\{k\})\geq 2^{-m}\big\}.\]
Then $A\in\mc{I}(\vec\mu)$ because $\mu_n(A\cap P_n)<r$ for every $n$; and clearly, $\max\{\mu_n(\{k\}):k\in P_n\setminus A\}\xrightarrow{n\to\infty} 0$, i.e. (1) holds with $X=\om\setminus A\in\mc{I}(\vec{\mu})^*$, a contradiction.

Now assume that $(A_\al)_{\al<\ka}$ is a $\subseteq^*$-increasing sequence in $\mc{I}(\vec\mu)$ for some $\ka=\mrm{cf}(\ka)>\om$. By passing to a subsequence, we can assume that there is an $L\in\om$ such that $\sup\{\mu_n(A_\al\cap P_n):n\in\om\}<L$ for every $\al$. Applying $(\star)$ for $r=L$, there is a $\de>0$ and a $Z\in [\om]^\om$ such that
\[ \mu_{n}\big(\big\{k\in P_{n}:\mu_{n}(\{k\})\geq\de\big\}\big)\geq L\;\text{ for every}\;n\in Z.\]
We can assume that $\de<L$. For every $n\in Z$ fix a $Q_n\subseteq  \{k\in P_{n}:\mu_{n}(\{k\})\geq\de\}$ such that $|Q_n|\leq \lceil L/\de\rceil$ and $\mu_n(Q_n)\geq L$. By shrinking $Z$, we can assume that $|Q_n|=K\in\om$ for every $n\in Z$. For every $\al$ let $n_\al=\limsup_{n\in Z}|A_\al\cap Q_n|$. Then $n_\al<K$ for every $\al<\ka$ because no $A_\al$ can contain any of the $Q_n$'s (because $\mu_n(Q_n)\geq L>\sup\{\mu_n(A_\al\cap P_n):n\in\om\}$). We can assume that $n_\al=N$ for every $\al$. Pick $k_n\in Q_n\setminus A_0$ for every $n\in Z'=\{n\in Z:|A_0\cap Q_n|=N\}\in [\om]^\om$ and let $B=\{k_n:n\in Z'\}\in [\om]^\om$. For every $\al$, if $n\in Z'$ is large enough then $A_0\cap Q_n\subseteq A_\al\cap Q_n$ and $|A_\al\cap Q_n|\leq N$, therefore $k_n\notin A_\al\cap Q_n$, in other words $|B\cap A_\al|<\om$, i.e. $\om\setminus B$ is a pseudounion of $(A_\al)_{\al<\ka}$.
\end{proof}

\begin{cor}\label{dichcoro}
Let $\mc{I}(\vec\mu)$ be a fragmented ideal where $\vec\mu$ is a sequence of measures. Then the following are equivalent:
\begin{itemize}
\item[(i)] For some $X\in\mc{I}(\vec\mu)^*$ the ideal $\mc{Z}_{\vec\mu}\upharpoonright X$ is tall.
\item[(ii)] For some $X\in\mc{I}(\vec\mu)^*$, $\mc{I}(\vec\mu)\upharpoonright X$ contains a tall analytic P-ideal.
\item[(iii)] $\mc{I}\leq_\mrm{KB}\mc{I}(\vec\mu)$ for some tall analytic P-ideal $\mc{I}$.
\item[(iv)] After adding $\om_1$ Cohen reals, $\mc{I}(\vec\mu)$ contains a cotower.
\item[(v)] Under $\mrm{CH}$, $\mc{I}(\vec\mu)$ contains a cotower.
\item[(vi)] In a forcing extension $\mc{I}(\vec\mu)$ contains a cotower.
\end{itemize}
\end{cor}
\begin{proof}
(i)$\to$(ii) is trivial. (ii)$\to$(iii): If $\mc{I}$ is a tall analytic P-ideal on an $X\in\mc{I}(\vec\mu)^*$ and $\mc{I}\subseteq\mc{I}(\vec\mu)\upharpoonright X$, then every $f:\om\to X$ satisfying $f\upharpoonright X=\mrm{id}_X$ witnesses $\mc{I}\leq_\mrm{KB}\mc{I}(\vec\mu)$. (iii)$\to$(iv): See the first paragraph in the next section, Fact \ref{kbtower}, and the remark on absoluteness before Fact \ref{kbtower}. (iii)$\to$(v): Under $\mrm{CH}$, every tall analytic P-ideal is generated by a cotower, and hence we can apply Fact \ref{kbtower}. (iv)$\to$(vi) and (v)$\to$(vi) are trivial. (vi)$\to$(i): Assume that $V^\PP\models$``$\mc{I}(\vec\mu)$ contains a cotower''. Applying Proposition \ref{fragdich} in $V^\PP$, and the fact that $(1)$ is a $\Ubf{\Sigma}^1_1$ statement, we know that $(1)$ and hence (i) holds in $V$ as well.
\end{proof}

\begin{rem} In Proposition \ref{fragdich} and Corollary \ref{dichcoro} we assumed that the ideal is generated by a sequence of measures. In the general case, if $\mc{I}(\vec\varphi)$ is an arbitrary fragmented ideal and there is an $X\in\mc{I}(\vec\varphi)^*$ such that $\max\{\varphi_n(\{k\}):k\in P_n\cap X\}\xrightarrow{n\to\infty} 0$, then $\mc{Z}_{\vec\varphi}\upharpoonright X$ is a tall analytic P-ideal (a generalized density ideal) and $\mc{I}(\vec\varphi)\upharpoonright X\supseteq\mc{Z}_{\vec\varphi}\upharpoonright X$. However, we do not know whether in Proposition \ref{fragdich} $\neg(1)\to (2)$ holds for arbitrary fragmented ideals.
\end{rem}

\begin{que}\label{fragmq}
Let $\mc{I}(\vec\varphi)$ be a fragmented ideal and assume that
\[ \limsup_{n\to\infty}\max\{\varphi_n(\{k\}):k\in P_n\cap X\}>0\;\;\text{for every}\;X\in\mc{I}(\vec\varphi)^*.\]
Can we show in $\mrm{ZFC}$ that $\mc{I}(\vec\varphi)$ does not contain cotowers?
\end{que}

Concerning these examples and observations above, the main open question is the following:

\begin{que}\label{fsq}
Assume that an $F_\sigma$ ideal $\mc{I}$ contains cotowers. Does $\mc{I}$ contain a tall analytic P-ideal?
\end{que}

\section{Cotowers in tall analytic P-ideals}\label{analp}

We know (see \cite{FaSo}) that
after adding $\om_1$ Cohen-reals (to any model) there is a cotower
of height $\om_1$ in each tall analytic P-ideal. Actually, we can prove it very easily now: We know (see \cite{cardinvanalp}) that $\cov^*(\mc{I})\leq \non(\mc{M})$ for every tall analytic P-ideal $\mc{I}$ where $\mc{M}$ is the $\sigma$-ideal of meager subsets of $2^\om$ (or of $\om^\om$). This is not difficult either, simply fix a partition $(P_n=\{i^n_k:k<k_n\})_{n\in
{\omega}}$ of ${\omega}$ into finite intervals such that
${\varphi}(\{x\})<2^{-n}$ for $x\in P_{n+1}$ (where of course, $\mc{I}=\mrm{Exh}(\varphi)$), and if $F\subseteq\om^\om$ is non-meager, then the set $\{\{i^n_k:f(n)\equiv k$ mod $k_n\}:f\in F\}\subseteq\mc{I}$ witnesses $\cov^*(\mc{I})\leq |F|$. In particular, after adding $\om_1$ Cohen reals, or in general, forcing with the limit of an arbitrary nontrivial $\om_1$ stage finite support iteration, in the extension $\non(\mc{M})=\om_1$ because of the Cohen reals and hence applying Fact \ref{fact1} (a), every tall analytic P-ideal contains cotowers.

\smallskip
It is also possible to force longer cotowers into tall analytic P-ideals. Let $\mbb{LOC}$ be the {\em localization forcing}:  Let $\mc{T}=\bigcup_{n\in\om}\prod_{k<n}[\om]^{\le k}$ be the tree of {\em initial slaloms}. $p=(s^p,F^p)\in\mbb{LOC}$ iff $s^p\in\mc{T}$, $F^p\subseteq\om^\om$, and $|F^p|< |s^p|$ (if $s^p=\0$ then $F^p=\0$ too). $p\le q$ iff
$s^p\supseteq s^q$, $F^p\supseteq F^q$,
and $\forall$ $n\in |s^p|\bs |s^q|$ $\forall$ $f\in F^q$
$f(n)\in s^p(n)$. It is straightforward to see that $\mbb{LOC}$ is $\sigma$-$n$-linked for each $n$ but it is not $\sigma$-centered.

If $G$ is a $(V,\mbb{LOC})$-generic filter, then $S=\bigcup\{s^p:p\in G\}\in V[G]$ is a {\em slalom over $V$}, that is
\[ S\in\prod_{n\in\om}[\om]^{\leq n}\;\;\text{and}\;\;\forall\;f\in \om^\om\cap V\;\forall^\infty\;n\in\om\;f(n)\in S(n).\] In particular, $d(n)=\max(S(n))$ is a dominating real over $V$.

Although, the next lemma is well-known as a consequence of certain {\em (Borel) Galois-Tukey connections}, the proof is so short, it is worth including here.

\begin{lem}\label{locdom}
If $\mc{I}$ is a tall analytic P-ideal, then $\mbb{LOC}$ is $\mc{I}$-dominating, that is, $\exists$ $A\in \mc{I}\cap V^{\mbb{LOC}}$ $\forall$ $B\in \mc{I}\cap V$ $B\subseteq^* A$.
\end{lem}
\begin{proof}
Let $\mc{I}=\mathrm{Exh}(\varphi)$ and fix a bijection
$\om\fv [\om]^{<\om}$, $k\mapsto F_k$ (i.e. a coding of finite sets with natural numbers). If $S\in V^{\mbb{LOC}}$ is a slalom over $V$, then let
\[ A=A_S=\bigcup_{n\in\om}\bigcup\big\{F_k:k\in S(n)\;\text{and}\; \varphi(F_k)<2^{-n}\big\}.\]
For each $n$ the set $\bigcup\big\{F_k:k\in S(n)$ and $\varphi(F_k)<2^{-n}\big\}$ is finite and has submeasure less then $n/2^n$ hence $A\in\mc{I}$. Assume now that $B\in\mc{I}\cap V$, let $d\in\om^\om\cap V$, $d(n)=\min\{k\in\om:\varphi(B\bs k)<2^{-n}\}$, and define $c\in\om^\om\cap V$ such that
\[ F_{c(n)}=B\cap\big[d(n),d(n+1)\big)\;\;\text{for every}\;n.\]
Notice that $B\subseteq^*\bigcup\big\{F_{c(n)}:n\in\om\big\}$ (we can assume that $\varphi(\{n\})>0$ for every $n$ hence $d(n)\to\infty$). There is an $N$ such that $c(n)\in S(n)$ for every $n\fel N$.
Clearly, $\varphi\big(F_{c(n)}\big)<2^{-n}$ so if $n\fel N$ then $F_{c(n)}\subseteq A$, therefore $B\subseteq^* \bigcup\{F_{c(n)}:n\geq N\}\subseteq A$.
\end{proof}

\begin{rem}
It is easy to see that if $X\subseteq\mc{P}(\om)$ is an analytic or coanalytic set, then ``$X$ is an ideal'' is $\Ubf{\Pi}^1_2$ hence absolute between $V$ and $V^\PP$. If $\mc{I}$ is analytic then the formula ``$\mc{I}$ is a P-ideal'' is $\Ubf{\Pi}^1_2$ hence absolute, but if $\mc{I}$ is coanalytic, then this formula is (formally) $\Ubf{\Pi}^1_3$. It is natural to ask whether there exist a coanalytic P-ideal $\mc{I}$ and a forcing notion $\PP$ (e.g. in $L$) such that $\vd_\PP$``$\mc{I}$ is not P-ideal''?

If an analytic or coanalytic ideal $\mc{I}$ is not a P-ideal in a ground model $V$, then there is no forcing notion which dominates it: The formula ``$(x_n)_{n\in\om}$ is a sequence in $\mc{I}$ without pseudounion in $\mc{I}$'' is also $\Ubf{\Pi}^1_2$ hence absolute. Therefore, if in $V$ a sequence $(A_n)_{n\in\om}$ of elements of $\mc{I}$ does not have a pseudounion in $\mc{I}$, then this sequence cannot have any pseudounion in $\mc{I}$ in any forcing extensions of $V$.
\end{rem}

\begin{thm}\label{locmodel}
Let $\PP$ be the $\ka=\mrm{cf}(\ka)>\om$ stage finite support iteration of $\mbb{LOC}$. Then in $V^\PP$, there are cotowers of length $\ka$ in all tall analytic P-ideals (and there are no cotowers of any other length in them).
\end{thm}
\begin{proof} (Sketch.)
Let $\mc{I}=\mrm{Exh}(\varphi)$. The only thing we have to show is that if $S\in V^{\mbb{LOC}}$ is the generic slalom over $V$, then $|A_S\setminus X|=\om$ for every coinfinite $X\in\mc{P}(\om)\cap V$ where $A_S$ is defined as in the lemma above. For $p\in\mbb{LOC}$ let
\[ A_p=\bigcup_{n<|s^p|}\bigcup\big\{F_k:k\in s^p(n)\;\text{and}\; \varphi(F_k)<2^{-n}\big\}.\]
Applying tallness of $\mc{I}$ (that is, $\varphi(\{n\})\to 0$), it is trivial to show that $D_n=\{p\in\mbb{LOC}:A_p\setminus X\nsubseteq n\}$ is dense in $\mbb{LOC}$ for every $n$.
\end{proof}

\begin{thm}\label{bigspec}
Let $\mc{I}=\mrm{Exh}(\varphi)$ be a tall analytic P-ideal and $R$ be a set of uncountable regular cardinals. Then in a ccc forcing extension $\mc{I}$ contains cotowers of length $\ka$ for every $\ka\in R$.
\end{thm}
\begin{proof}
We will add towers of height $\ka$   simultaneously for every $\ka\in R$ into $\mc{I}^*$. Define $p\in\PP$ iff $p=(n^p,\eps^p,(A_\ka^p:\ka\in F^p),(s^p_\ka(\xi):\ka\in F^p,\xi\in A^p_\ka))$ where
\begin{itemize}
\item[(a)] $n^p\in\om$, $\eps^p>0$ is a rational number,
\item[(b)] $F^p\in [R]^{<\om}$,
\item[(c)] $A^p_\ka\in [\ka]^{<\om}$ for every $\ka\in F^p$,
\item[(d)] $s^p_{\ka}(\xi)\subseteq n^p$ for every $\ka\in F^p$ and $\xi\in A^p_\ka$.
\end{itemize}
Define $p\leq q$ iff
\begin{itemize}
\item[(i)] $n^p\geq n^q$, $\eps^p\leq \eps^q$, $F^p\supseteq F^q$,
\item[(ii)] $A^p_\ka\supseteq A^q_\ka$ for every $\ka\in F^q$,
\item[(iii)] $s^p_{\ka}(\xi)\cap n^q=s^q_\ka(\xi)$ for every $\ka\in F^q$ and $\xi\in A^q_\ka$,
\item[(iv)] $s^p_\ka(\xi_0)\setminus n^q\supseteq s^p_\ka(\xi_1)\setminus n^q$ for every $\ka\in F^q$ and $\xi_0,\xi_1\in A^q_\ka$ with $\xi_0<\xi_1$.
\item[(v)] $\varphi([n^q,n^p)\setminus s^p_\ka(\xi))\leq\eps^q-\eps^p$ for every $\ka\in F^q$ and $\xi\in A^q_\ka$.
\end{itemize}
It is straightforward to check that $\leq$ is a partial order.
\begin{claim}
$\PP$ has the ccc.
\end{claim}
\begin{proof}[Proof of the Claim.]
Let
\[\big\{p_\al=\big(n^\al,\eps^\al,(A_\ka^\al:\ka\in F^\al),(s^\al_\ka(\xi):\ka\in F^\al,\xi\in A^\al_\ka)\big):\al<\om_1\big\}\subseteq\PP.\] We can assume the following: $n^\al=n$, $\eps^\al=\eps$ for every $\al$; the family $\{F^\al:\al<\om_1\}$ is a $\Delta$-system, $F^\al=F\cup E^\al$; for every $\ka\in F$ $\{A^\al_\ka:\al<\om_1\}$ is a $\Delta$-system, $A^\al_\ka=A_\ka\cup B^\al_\ka$; and $s^\al_\ka(\xi)=s_\ka(\xi)$ for every $\ka\in F$ and $\xi\in A_\ka$ (we simply shrink the family of conditions in finitely many steps).

Fix $\al,\be<\om_1$ and let $r=(n,\eps,(A^r_\ka:\ka\in F^r),(s^r_\ka(\xi):\ka\in  F^r,\xi \in A^r_\ka))$ where $F^r=F^\al\cup F^\be$, $A^r_\ka=A^\al_\ka$ if $\ka\in E^\al$, $A^r_\ka=A^\be_\ka$ if $\ka\in E^\be$, $A^r_\ka=A^\al_\ka\cup A^\be_\ka$ if $\ka\in F$, and

\[
 s^r_\ka(\xi) =
  \begin{cases}
   s^\al_\ka(\xi) & \text{if}\;\; \ka\in E^\al\;\text{and}\;\xi\in A^\al_\ka \\
   s^\be_\ka(\xi) & \text{if}\;\; \ka\in E^\be\;\text{and}\;\xi\in A^\be_\ka\\
   s^\al_\ka(\xi) & \text{if}\;\; \ka\in F\;\text{and}\;\xi\in B^\al_\ka\\
   s^\be_\ka(\xi) & \text{if}\;\; \ka\in F\;\text{and}\;\xi\in B^\be_\ka\\
   s_\ka(\xi) & \text{if}\;\; \ka\in F\;\text{and}\;\xi\in A_\ka
  \end{cases}
\]

It is easy to see that $r\in \PP$ and $r\leq p_\al,p_\be$.
\end{proof}

It is important to see that compatibility of two conditions from $\PP$ depends only on their ``intersection'':

\begin{claim}\label{claim2}
Let $p,q\in\PP$ with $n^q\geq n^p$. Then $p$ and $q$ are compatible iff
\begin{itemize}
\item[(1)]  either $n^p=n^q$ and $s^p_\ka(\xi)=s^q_\ka(\xi)$ for every $\ka\in F^p\cap F^q$ and $\xi\in A^p_\ka\cap A^q_\ka$,

\item[(2)] or $n^q>n^p$ and the following hold for every $\ka\in F^p\cap F^q$: If $A^p_\ka\cap A^q_\ka=\{\xi^\ka_0<\xi^\ka_1<\dots<\xi^\ka_{m_\ka-1}\}$ then
    \begin{itemize}
    \item[(2a)] $s^q_\ka(\xi^\ka_i)\cap n^p=s^p_\ka(\xi^\ka_i)$ for every $i<m_\ka$;

    \item[(2b)] $s^q_\ka(\xi^\ka_0)\setminus n^p\supseteq s^q_\ka(\xi^\ka_1)\setminus n^p\supseteq\dots\supseteq s^q_\ka(\xi^\ka_{m_\ka-1})\setminus n^p$;
    \item[(2c)] $\varphi([n^p,n^q)\setminus s^q_\ka(\xi^\ka_{m_\ka-1}))<\eps^p$.
\end{itemize}
\end{itemize}
\end{claim}
\begin{proof}[Proof of the Claim.]
The ``only if'' part is trivial. Similarly, (1) clearly implies that $p$ and $q$ are compatible. Now assume that (2a)-(2b)-(2c) holds for every $\ka\in  F^p\cap F^q$ and define $r\in\PP$ as follows: $n^r=n^q$, $0<\eps^r\leq \eps^q$ and \[\tag{$\star$} \eps^r<\min\big\{\eps^p-\varphi\big([n^p,n^q)\setminus s^q_\ka(\xi^\ka_{m_\ka-1})\big):\ka\in F^p\cap F^q\big\},\]
$F^r=F^p\cup F^q$, $A^r_\ka=A^p_\ka\cup A^q_\ka$ for every $\ka\in F^r$ (where of course, if e.g. $\ka\notin F^q$ then $A^q_\ka=\0$). We still have to define $s^r_\ka(\xi)$ for $\ka\in F^r$ and $\xi\in A^r_\ka$:
\[
 s^r_\ka(\xi) =
  \begin{cases}
    s^p_\ka(\xi)\cup [n^p,n^q) & \text{if}\;\; \ka\notin F^q\;\text{and}\; \xi\in A^p_\ka\\
    s^q_\ka(\xi) & \text{if}\;\; \ka\in F^q\;\text{and}\;\xi\in A^q_\ka\\
    s^p_\ka(\xi)\cup [n^p,n^q) & \text{if}\;\;\ka\in F^q\;\text{and}\; \xi\in A^p_\ka\setminus A^q_\ka\;\text{and}\;\xi<\xi^\ka_0\\
     s^p_\ka(\xi)\cup \big(s^q_\ka(\xi_i)\setminus n^p\big) &  \text{if}\;\;\ka\in F^q\;\text{and}\; \xi\in A^p_\ka\setminus A^q_\ka\;\text{and}\;\xi^\ka_i<\xi<\xi^\ka_{i+1}
  \end{cases}\]
where $\xi^\ka_{m_\ka}=\ka$. It is trivial that $r\leq q$ and straightforward to show that $r\leq p$ also holds.
\end{proof}

One can easily show that the following sets are dense in $\PP$:
\begin{align*}
D_{n,\eps} & =\big\{p\in \PP:n^p\geq n\;\text{and}\; \eps^p<\eps \big\}\;\;\text{for every}\;n\in\om\;\text{and}\;\eps>0\\
E_{\ka,\xi} & =\big\{p\in\PP:\ka\in F^p\;\text{and}\;\xi\in A^p_\ka\big\}\;\;\text{for every}\;\ka\in R\;\text{and}\;\xi<\ka
\end{align*}
For every $\ka\in R$ let $\dot{T}_\ka$ be a $\PP$-name such that $V^\PP\models$``$\dot{T}_\ka:\ka\to\mc{P}(\om)$ and $\forall$ $\xi<\ka$ $\dot{T}_\ka(\xi)=\bigcup\{s^p_\ka(\xi):p\in \dot{G}\}$''. Applying (iv), it is easy to check that $V^\PP\models$``$\dot{T}_\ka$ is a $\subseteq^*$-decreasing sequence''. Similarly, applying that $\mc{I}$ is tall, i.e. that $\varphi(\{n\})\to 0$, one can show that $V^\PP\models|\dot{T}_\ka(\xi_0)\setminus\dot{T}_\ka(\xi_1)|=\om$  for every $\ka\in R$ and $\xi_0<\xi_1<\ka$. Applying (v) and the density of $D_{n,\eps}$ and $E_{\ka,\xi}$, it is straightforward to show that $V^\PP\models \dot{T}_\ka(\xi)\in\mc{I}^*$ for every $\ka\in R$ and $\xi<\ka$. To finish the proof, we show that $V^\PP\models$``$\dot{T}_\ka$ is a tower''.

Let $\dot{X}=\bigcup\{\{n\}\times A_n:n\in\om\}$ be a nice $\PP$-name for an infinite subset of $\om$ where of course $A_n\subseteq\PP$ is an antichain for every $n$. We can fix $\xi_\ka<\ka$ for every $\ka\in R$ such that
\[ \sup\Big(\bigcup\big\{A^p_\ka:p\in \bigcup\big\{A_n:n\in\om\big\}\;\text{and}\;\ka\in F^p\big\}\Big)<\xi_\ka.\]
In other words, the sequence $(\xi_\ka)_{\ka\in R}$ can be seen as a ``support'' of $\dot{X}$. Fix a $\ka\in R$, a $p\in E_{\ka,\xi_\ka}$, and an $m\in\om$. If we can find an $r\leq p$ such that $r\vd \dot{X}\setminus \dot{T}_\ka(\xi_\ka)\nsubseteq m$, then we are done.
We can assume that $m\geq n^p$ and that $\varphi(\{n\})<\eps^p/2$ for every $n\geq m$. Now let $p\upharpoonright_\ka \xi_\ka$ be the {\em restriction of $p$ below $\xi_\ka$ at $\ka$}, that is,
\[ p\leq p\upharpoonright_\ka \xi_\ka=\big(n^p,\eps^p,(A^p_\lam:\lam\in F^p),(s^p_\lam(\xi):\lam\in F^p,\xi\in A^{p\upharpoonright_\ka \xi_\ka}_\lam)\big)\] where $A^{p\upharpoonright_\ka \xi_\ka}_\lam=A^p_\lam$ if $\lam\in F^p\setminus \{\ka\}$ and $A^{p\upharpoonright_\ka \xi_\ka}_\ka=A^p_\ka\cap\xi_\ka$. Define $p_0\leq p\upharpoonright_\ka \xi_\ka$ by decreasing $\eps^{p\upharpoonright_\ka \xi_\ka}=\eps^p$ to its half:
\[ p_0=\big(n^p,\frac{\eps^p}{2},(A^p_\lam:\lam\in F^p),(s^p_\lam(\xi):\lam\in F^p,\xi\in A^{p\upharpoonright_\ka \xi_\ka}_\lam)\big)\leq p\upharpoonright_\ka \xi_\ka.\]

We can find a $q_0\leq p_0$ and an $n\geq m$ such that $n^{q_0}>n$ and $q_0\vd n\in \dot{X}$, i.e. $A_n$ is predense below $q_0$. We show that $q_0\upharpoonright_\ka \xi_\ka\vd n\in\dot{X}$ as well. Applying the second Claim and the definition of $\xi_\ka$, an $r\in\PP$ is compatible with an element $a$ of $A_n$, $r\|a$ iff $r\upharpoonright_\ka \xi_\ka\|a$. If $r\leq q_0\upharpoonright_\ka \xi_\ka$, then there is an $r'\leq q_0$ such that $r'\upharpoonright_\ka \xi_\ka=r\upharpoonright_\ka \xi_\ka$ because we can easily modify $r$ at $\ka$ above $\xi_\ka$ to construct such an $r'$. And we are done since this shows that $A_n$ is predense below $q_0\upharpoonright_\ka\xi_\ka$: If $r\leq q_0\upharpoonright_\ka\xi_\ka$ then fix $r'$ as above, we know that $r'\| a$ for some $a\in A_n$, then $r'\upharpoonright_\ka \xi_\ka=r\upharpoonright_\ka\xi_\ka\| a$, and hence $r\| a$.

We show that $q_0\upharpoonright_\ka\xi_\ka$ and $p$ are also compatible, moreover they have a common extension $r$ such that $n\notin s^r_\ka(\xi_\ka)$, in a diagram (where of course, $p\to q$ stands for $p\leq q$):
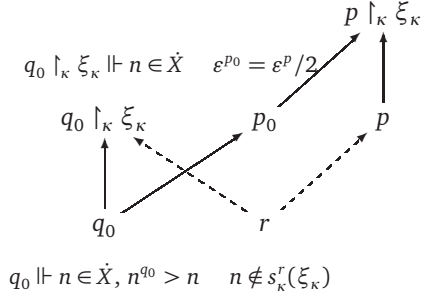
\begin{diagram}
&&&& p\upharpoonright_\ka\xi_\ka\\
\text{\small{$q_0\upharpoonright_\ka\xi_\ka\vd n\in\dot{X}$}}&& \text{\small{$\eps^{p_0}=\eps^p/2$}}& \ruTo & \uTo\\
q_0\upharpoonright_\ka\xi_\ka & & p_0 & & p\\
\uTo & \luDashto \ruTo & & \ruDashto & \\
q_0 & & r &  &\\
\text{\small{$q_0\vd n\in\dot{X}$, $n^{q_0}>n$}} &  & \;\;\;\;\;\text{\small{$n\notin s^r_\ka(\xi_\ka)$}} & &
\end{diagram}
To define $r$, we will apply the second Claim: We know that $n^{q_0\upharpoonright_\ka\xi_\ka}=n^{q_0}>n^p$. We show that (2c) holds, the rest is trivial. For $\lam\in F^p\cap F^{q_0\upharpoonright_\ka\xi_\ka}=F^p$ let $A^p_\lam\cap A^{q_0\upharpoonright_\ka\xi_\ka}_\lam=\{\xi^\lam_0<\xi^\lam_1<\dots<\xi^\lam_{m_\lam-1}\}$. Notice that if $\lam\ne\ka$ then $A^p_\lam\cap A^{q_0\upharpoonright_\ka\xi_\ka}_\lam=A^p_\lam$ and $A^p_\ka\cap A^{q_0\upharpoonright_\ka\xi_\ka}_\ka=A^p_\ka\cap\xi_\ka$ (in particular $\xi^\ka_{m_\ka-1}<\xi_\ka$). We know that $q_0\leq p_0$ and hence
\begin{align*}\tag{$\sharp$}
\varphi\big([n^p,n^{q_0\upharpoonright_\ka\xi_\ka})\setminus s^{q_0\upharpoonright_\ka\xi_\ka}_\lam(\xi^\lam_{m_\lam-1})\big)& =\varphi\big([n^{p_0},n^{q_0})\setminus s^{q_0}_\lam(\xi^\lam_{m_\lam-1})\big)\\
&\leq\eps^{p_0}-\eps^{q_0}=(\eps^p/2)-\eps^{q_0}<\eps^p.\end{align*}
Therefore when defining their common extension $r$ we can choose $n^{r}=n^{q_0\upharpoonright_\ka\xi_\ka}=n^{q_0}$ and $\eps^{r}=\eps^{q_0\upharpoonright \xi_\ka}=\eps^{q_0}$ because then $(\star)$ holds as $\eps^{q_0}<\eps^p-((\eps^p/2)-\eps^{q_0})=(\eps^p/2)+\eps^{q_0}$.

We show that we can define $s^{r}_\ka$ as follows:
\[ s^{r}_\ka(\xi)=s^p_\ka(\xi)\cup \big(s^{q_0\upharpoonright \xi_\ka}_\ka(\xi^\ka_{m_\ka-1})\setminus(n^p\cup \{n\})\big)\;\;\text{for every}\;\xi\in A^p_\ka,\;\xi>\xi^\ka_{m_\ka-1}.\]
In particular, $n\notin s^r_\ka(\xi_\ka)$. This is a slight modification of the last case in the definition of $s^r_\ka$ in the proof of the second Claim. We have to check that $r\leq q_0\upharpoonright_\ka\xi_\ka$ and $r\leq p$ still hold. The first one holds trivially. To show that $r\leq p$, we have to check that $\varphi([n^p,n^r)\setminus s^r_\ka(\xi))\leq\eps^p-\eps^r$ for every $\xi\in A^p_\ka\setminus \xi_\ka$. Applying $(\sharp)$ and $\varphi(\{n\})<\eps^p/2$ we obtain that
\begin{align*}
\varphi\big([n^p,n^r)\setminus s^r_\ka(\xi)\big)& = \varphi\big(\big([n^p,n^{q_0\upharpoonright_\ka\xi_\ka})\setminus s^{q_0\upharpoonright \xi_\ka}_\ka(\xi^\ka_{m_\ka-1})\big)\cup \{n\}\big)\\
 & \leq \varphi\big([n^p,n^{q_0\upharpoonright_\ka\xi_\ka})\setminus s^{q_0\upharpoonright \xi_\ka}_\ka(\xi^\ka_{m_\ka-1})\big)+\varphi(\{n\})\\
 &< ((\eps^p/2)-\eps^{q_0})+(\eps^p/2) =\eps^p-\eps^r.\end{align*}
Now $r\leq p$ is as required because $n\notin s^{r}_\ka(\xi_\ka)$ and $A_n$ is predense below $r$, therefore $r\vd n\in \dot{X}\setminus \dot{T}_\ka(\xi_\ka)$.
\end{proof}

At this moment it is unclear whether there are towers of height not from $R$ in the above model (even if $V\models\mrm{CH}$).

\begin{que}\label{sptower}
For a tall analytic P-ideal $\mc{I}$ let us denote $\mf{S}_\mrm{tr}(\mc{I})=\{\ka:\ka$ is regular and $\mc{I}$ contains a cotower of height $\ka\}\subseteq [\cov^*(\mc{I}),\non^*(\mc{I})]$ the {\em tower spectrum} of $\mc{I}$. Does $\mf{S}_\mrm{tr}(\mc{I})$ have any kind of closedness properties (e.g. pcf type closedness)? Can the forcing notion defined in Theorem \ref{bigspec} be used to distinguish tower spectrums of different analytic P-ideals?
\end{que}

\section{Destroying meager towers}\label{killing}

It is easy to see that every $\subseteq^*$-decreasing sequence in $[\om]^\om$ is meager, and so the following common terminology might sound a bit strange: We say that a tower $\mc{T}$ is {\em meager} if the filter $\mrm{fr}(\mc{T})$ generated by $\mc{T}$ is meager. For example, if $\mc{T}$ is a tower in a Borel filter, then it is meager because of the following classical result.
\begin{Thm}{\em (Sierpi\'{n}ski and Talagrand, see \cite[Thm. 4.1.1-2]{BaJu})}\label{talag}
Let $\mc{F}$ be a filter on $\om$. Then the following are equivalent: (i) $\mc{F}$ has the Baire property, (ii) $\mc{F}$ is meager, (iii) there is a partition $(P_n)_{n\in\om}$ of $\om$ into finite sets such that $\forall$ $F\in\mc{F}$ $\forall^\infty$ $n\in\om$ $F\cap P_n\ne\0$.
\end{Thm}

In general, as well as concerning our investigations, it is natural to ask what we can say about the existence of meager and non-meager towers. Surprisingly, all three possibilities (see below) are consistent. The following result, especially the remark after it, plays an important role when studying meagerness of towers. We will need the {\em unbounding} and {\em dominating numbers}:
\begin{align*}
\mf{b} & =\min\big\{|U|:U\subseteq\om^\om\;\text{is}\;\leq^*\text{-unbounded}\big\}\\
\mf{d} &=\min\big\{|D|:D\subseteq\om^\om\;\text{is}\;\leq^*\text{-cofinal}\big\}
\end{align*}
It is easy to see  that $\om<\mrm{cf}(\mf{t})=\mf{t}\leq \mrm{cf}(\mf{b})=\mf{b}\leq\mrm{cf}(\mf{d})\leq\mf{d}\leq\mf{c}$.

\begin{thm} {\em (R.\ C.\ Solomon \cite{Solomon} and P.\ Simon [unpublished], see \cite[Thm. 9.10]{Bl})} \label{filgen}
If a filter is generated by less than $\mf{b}$ many sets, then it is meager but there is a non-meager filter generated by $\mf{b}$ sets.
\end{thm}

It is clear from the proof of Theorem \ref{filgen} that in the case of $\mf{b}<\mf{d}$, the example of $\mf{b}$ sets which generate a non-meager filter, can be chosen as a tower.

Now let us recall some known results. It is consistent that all towers are meager: In the Hechler model all towers are of height $\om_1<\mf{b}=\om_2$ (see \cite{baumdord}) and hence (applying Theorem \ref{filgen}) they generate meager filters.

If $\mf{b}<\mf{d}$ (e.g. in the Cohen model), then there is a non-meager tower (see the remark after Theorem \ref{filgen}). One can show that $\mf{b}=\om_1$ also implies that there is a non-meager tower: If $\mf{b}<\mf{d}$ then we are done. If $\mf{d}=\om_1$, then applying the fact (see \cite[Thm. 2.10]{Bl}) that $\mf{d}$ is the minimum of cardinalities of dominating families of interval partitions, we can easily diagonalize against all interval partitions and construct a non-meager tower of height $\om_1$.

And finally, an easy transfinite construction shows that $\mf{t}=\mf{c}$ also implies that there is a non-meager tower. In particular, applying Fact \ref{fact1} (a), under $\mrm{CH}$ and under $\mrm{MA}$ there are both meager and non-meager towers (here we used the fact that $\mrm{add}(\mc{N})\leq\mrm{add}^*(\mc{I})$ for every tall analytic P-ideal, see also in Section \ref{analpindiag}).

We show that the third possibility, namely, that all towers are non-meager, is also consistent.

\begin{lem}\label{megtowsigma}
Assume $(T_\al)_{\al<\ka}$ is a meager tower, as witnessed by the partition $(P_n)_{n\in\om}$ of $\om$ into finite sets (i.e. $\forall$ $\al$ $\forall^\infty$ $n$ $(T_\al\cap P_n\ne\0)$), and let $\PP$ be a $\sigma$-centered forcing notion. Then $\PP$ does not add a pseudointersection $X$ of $(T_\al)_{\al<\ka}$ such that $\forall^\infty$ $n$ $(X\cap P_n\ne\0)$.
\end{lem}
\begin{proof}
Assume $\dot{X}$ is a $\PP$-name for a subset of $\om$ such that $\vd_\PP \dot{X} \cap P_n \neq \emptyset$ for every $n$. We show that $\dot X$ is forced not to be a pseudointersection of the $(T_\al)_{\al<\ka}$. Let $\PP=\bigcup_{\ell\in\om}C_\ell$ where every $C_\ell$ is centered, and for $\ell \in \omega$ define
\[ X_\ell = \big\{ m :\forall\;p \in C_\ell\;(p\nVdash m \notin \dot X) \big\}=\big\{m:\forall\;p\in C_\ell\;\exists\;q\leq p\;(q\vd m\in\dot{X})\big\}.\]
We show that $X_\ell \cap P_n \neq
\emptyset$ for every $\ell$ and $n$. Fix $\ell$ and $n$, and assume on the contrary that $X_\ell\cap P_n=\0$. Then for every $m\in P_n$ there is a $p_m\in C_\ell$ witnessing that $m\notin X_\ell$, that is, $p_m\vd m\notin\dot{X}$. If $p$ is a common extension of these conditions ($p_m$, $m\in P_n$), then $p\vd\dot{X}\cap P_n=\0$, a contradiction.

In particular, all $X_\ell$ are infinite. Since $(T_\al)_{\al<\ka}$ has no pseudointersection (and $\mrm{cf}(\ka)>\om$), there is an $\alpha < \ka$ such that $X_\ell \nsubseteq^* T_\alpha$ for every $\ell$. We claim that $\vd \dot X \nsubseteq^* T_\alpha$. Fix an $\ell$, a $p\in C_\ell$, a $k\in\om$, and an $m\in X_\ell\setminus (T_\al\cup k)$. Then there is a $q\leq p$ such that $q\vd m\in \dot{X}$. In other words, $\{q\in\PP:q\vd \dot{X}\setminus T_\al\nsubseteq k\}$ is dense in $\PP$ for every $k$.
\end{proof}

Before the next theorem we recall that if $|\eps|\leq\mf{c}$ then the limit of an $\eps$ stage finite support iteration of $\sigma$-centered
forcing notions is $\sigma$-centered.

\begin{thm}\label{nomeagertower} It is consistent with $\mrm{MA}(\sigma$-centered$)$ that there are no meager towers.
\end{thm}
\begin{proof}
Do the standard $\omega_2$ stage finite support iteration to force $\mrm{MA}(\sigma$-centered$)$ over a model $V$ of $\mrm{CH}$. Let $V[G_{\al}]$ denote the $\al$th intermediate model.  Then in $V[G_{\om_2}]$ all towers are of height $\omega_2$. Suppose $V[G_{\om_2}]\models$``$(T_\alpha)_{\alpha < \omega_2}$ is a meager tower witnessed by the partition $(P_n)_{n \in \omega}$''.

Then using $\mathrm{CH}$ (in $V$) we can find a
$\ga<\om_2$ such that $(P_n)_{n\in\om},(T_\al)_{\al<\ga}\in V[G_\ga]$ and $V[G_\ga]\models$``$(T_\al)_{\al<\ga}$ is
a meager tower witnessed by $(P_n)_{n\in\om}$''. Why? Fix a sequence $(\dot{T}_\al)_{\al<\om_2}$ of nice $\PP_{\om_2}$-names such that $\mrm{val}_{G_{\om_2}}(\dot{T}_\al)=T_\al$. Then there is a sequence $(\ga_\al)_{\al<\om_2}$ in $\om_2$ (in $V$) such that $\dot{T}_\al$ is a $\PP_{\ga_\al}$-name and  $\forall$ $X\in [\om]^\om\cap V[G_\al]$ $X\nsubseteq^* T_{\ga_\al}$ for every $\al$ (here we used that $\PP_{\om_2}$ has the ccc and that $V[G_\al]\models\mrm{CH}$ for every $\al<\om_2$). Let $C=\{\ga<\om_2:\ga_\al<\ga$ for every $\al<\ga\}$,
a club in $\om_2$ (in $V$). Choose $\ga\in C$ of cofinality $\om_1$. Then $(T_\al)_{\al<\ga}$ has a $\PP_\ga$-name, and since all reals in $V[G_\ga]$ appeared in an earlier intermediate model, $(T_\al)_{\al< \ga}$ has no pseudointersection in $V[G_\ga]$.

Choose a $\ga'>\ga$ such that $T_\ga\in V[G_{\ga'}]$. The
quotient forcing $\PP_{\ga'}/G_\ga\in V[G_\ga]$ is a $\ga'-\ga$-stage finite support iteration of $\sigma$-centered forcing notions over $V[G_\ga]$ so it is also
$\sigma$-centered because $V[G_\ga]\models
|\ga'-\ga|\le\mf{c}=\om_1$. By Lemma \ref{megtowsigma}, $\PP_{\ga'}/G_\ga$ does not add a pseudointersection $X\in V[G_{\ga'}]$ of $(T_\al)_{\al<\ga}$ such that $\forall^\infty$ $n$ $X\cap P_n\ne\0$, contradiction
because $T_\ga\in V[G_{\ga'}]$.
\end{proof}

It could be interesting to investigate the existence of meager and non-meager towers under additional conditions, for example the following question is still open:

\begin{que}\label{tmeager}
Does $\mf{t} = \omega_1$ imply the existence of a meager tower? (It surely does not imply the existence of a non-meager tower, see above.)
\end{que}

In the rest of this section we present a more subtle way of destroying cotowers in $F_\sigma$ ideals and analytic P-ideals. Let $\mc{F}$ be a filter on $\om$.

The {\em Mathias-Prikry forcing
associated to $\mc{F}$}, $\mbb{M}(\mc{F})$ is the following
forcing notion:
\begin{itemize}
\item $(s,A)\in\mbb{M}(\mc{F})$ iff $s\in [\om]^{<\om}$,
$A\in\mc{F}$, and $\max(s)<\min(A)$,

\item $(s,A)\le (t,B)$ iff $s\supseteq t$, $A\subseteq B$, and
$s\bs t\subseteq B$.
\end{itemize}

The {\em Laver-Prikry forcing associated to $\mc{F}$}, $\mbb{L}(\mc{F})$ is the following forcing notion:

\begin{itemize}
\item $T\in\mbb{L}(\mc{F})$ iff $T\subseteq\om^{<\om}$ is a tree (that is, $T\ne\0$ is closed for taking initial segments and it has no $\subseteq$-maximal elements) and there is a  $\mrm{stem}(T)\in T$ which is comparable with all elements of $T$ and $\mrm{ext}_T(t)=\{n\in\om:t^\frown(n)\in T\}\in\mc{F}$ for every $t\in T$, $t\supseteq\mrm{stem}(T)$.

\item $T_0\le T_1$ iff $T_0\subseteq T_1$.
\end{itemize}

The following is trivial from the definitions.

\begin{Fact} If $\mc{F}$ is a filter on $\om$, then
$\mbb{M}(\mc{F})$ and $\mbb{L}(\mc{F})$ are $\sigma$-centered and both of them add pseudointersections to $\mc{F}$.
\end{Fact}

We will need the following variant of Lemma \ref{megtowsigma}:

\begin{Lem}\label{sigmacentered} Let $\PP$ be
a $\sigma$-centered forcing notion.
\begin{itemize}
\item[(a)] Assume $\mc{I}=\mathrm{Fin}(\varphi)$ is an $F_\sigma$ ideal
and $(A_\al)_{\al<\ka}$ is a
$\subseteq^*$-descending sequence in $\mc{I}^*$ which has no pseudointersection in
$\mc{I}^*$.  Then in $V^\PP$ this
sequence still has no pseudointersection in $\mc{I}^*$.
\item[(b)] (see also \cite[Lem. 7.4]{tristan})
Assume $\mc{I}=\mathrm{Exh}(\varphi)$ is an analytic P-ideal
and $(A_\al)_{\al<\ka}$ is a
$\subseteq^*$-descending sequence in $\mc{I}^*$ which has no pseudointersection
$X$ such that $\|\om\bs X\|_\varphi<\|\om\|_\varphi$. Then in $V^\PP$ this
sequence still has no pseudointersection $X$ satisfying $\|\om\bs
X\|_\varphi<\|\om\|_\varphi$.
\end{itemize}
\end{Lem}
\begin{proof}
Let $\PP=\bigcup_{\ell\in\om}C_\ell$ where all $C_\ell$ are
centered.

(a): We can assume that $\cf(\ka)>\om$ because the property ``$(A_n)_{n\in\om}$ is a $\subseteq^*$-decreasing sequence in $\mc{I}^*$ without pseudointersection in $\mc{I}^*$'' is $\Ubf{\Pi}^1_1$ hence absolute between $V$ and $V^\PP$. Now assume that $\dot{X}$ is a $\PP$-name for a subset of $\om$ such
that $\vd_\PP\varphi(\om\bs\dot{X})<\infty$. W.l.o.g.
we may assume that there is a $K\in\om$ such that $\vd_\PP \varphi(\om\bs\dot{X})\leq K$. Let $X_\ell=\{m\in\om:\forall$ $p\in C_\ell$ $p\nVdash m\notin \dot{X}\}$.

We claim that $\varphi(\om\bs X_\ell)\leq K$ for each $\ell$. Assume
on the contrary that $\varphi(\om\bs X_\ell)>K$. Then there is an $n$ such
that $\varphi\big((\om\bs X_\ell)\cap n\big)>K$. For each $m\in (\om\bs
X_\ell)\cap n$ there is a $p_m\in C_\ell$ with $p_m\vd
m\notin\dot{X}$. If $p$ is a common extension of these conditions,
then $p\vd \varphi(\om\bs\dot{X})\geq
\varphi((\om\bs \dot{X})\cap n)\geq
\varphi((\om\bs X_\ell)\cap n)>K$,
a contradiction.

From this point we can proceed as in the proof of Lemma \ref{megtowsigma}.

\smallskip
(b):
Note that $\cf(\ka)>\om$ because $\mc{I}$ is a
P-ideal and each element $X$ of $\mc{I}^*$ has the property
$\|\om\bs X\|_\varphi=0<\|\om\|_\varphi$.

Assume $\dot{X}$ is a $\PP$-name for a subset of $\om$ such
that $\vd_\PP\|\om\bs\dot{X}\|_\varphi<\|\om\|_\varphi$. W.l.o.g.
we may assume that there are $\eps>0$ and $N\in\om$ such that
$\vd_\PP \varphi((\om\bs\dot{X})\bs
N)\le \|\om\|_\varphi-\eps$. Let $X_\ell=\{m\in\om:\forall$ $p\in C_\ell$ $p\nVdash m\notin \dot{X}\}$.

We claim that $\|\om\bs X_\ell\|_\varphi<\|\om\|_\varphi$ for each $\ell$. Assume
on the contrary that they are equal. Then there is an $n>N$ such
that $\varphi\big((\om\bs X_\ell)\cap
[N,n)\big)>\|\om\|_\varphi-\eps$. For each $m\in (\om\bs
X_\ell)\cap [N,n)$ there is a $p_m\in C_\ell$ with $p_m\vd
m\notin\dot{X}$. If $p$ is a common extension of these conditions,
then $p\vd \varphi\big((\om\bs\dot{X})\bs N)\big)\fel
\varphi\big((\om\bs \dot{X})\cap [N,n)\big)\fel
\varphi\big((\om\bs X_\ell)\cap [N,n)\big)>\|\om\|_\varphi-\eps$,
a contradiction.

Again, from this point the proof is the same as in the case of Lemma \ref{megtowsigma}.
\end{proof}

\begin{cor}\label{firstcor}
Assume $\mc{I}$ is an $F_\sigma$ ideal or an analytic P-ideal, and $(A_\al)_{\al<\ka}$
is a tower in $\mc{I}^*$. Let $\PP$ be a $\sigma$-centered forcing
notion. Then $\PP$ forces that this sequence has no
pseudointersection in $\mc{I}^*$.
\end{cor}

The following reformulation will be useful when calculating values of cardinal invariants in certain forcing extensions (in Section \ref{analpindiag}).
\begin{cor}\label{preservingadd} Let $\PP$ be a $\sigma$-centered forcing notion.
\begin{itemize}
\item[(a)] If $\mc{I}$ is a tall $F_\sigma$ P-ideal and $\add^*(\mc{I})=\ka$, then $\vd_\PP\add^*(\mc{I})\leq\ka$.
\item[(b)] If $\mc{I}$ is a tall analytic P-ideal and there is a tower in $\mc{I}^*$ of height $\lam$ (e.g. if $\add^*(\mc{I})=\cov^*(\mc{I})=\lam$, see Fact \ref{fact1} (a)), then $\vd_\PP\add^*(\mc{I})\leq\lam$.
\end{itemize}
\end{cor}
\begin{proof}
(a): If $\add^*(\mc{I})=\ka$ then there is a $\subseteq^*$-decreasing sequence $(A_\al)_{\al<\ka}$ in $\mc{I}^*$ which has no pseudointersection in $\mc{I}^*$ and hence we can apply Lemma \ref{sigmacentered} (a).

(b) follows from Corollary \ref{firstcor}.
\end{proof}

\begin{Thm}\label{first-thm} Assume $\mathrm{CH}$ and that $\mc{I}$ is an $F_\sigma$ ideal or an analytic P-ideal.
Then the $\om_2$-stage finite support iteration of
$\mbb{M}(\mc{I}^*)$ or $\mbb{L}(\mc{I}^*)$ forces that there is no tower in $\mc{I}^*$.
\end{Thm}

\begin{proof}
Let us denote $(\PP_\al)_{\al\le\om_2}$ the iteration, $G$ a
$\PP_{\om_2}$-generic filter over $V$, and let
$G_\al=G\cap\PP_\al$ (a $\PP_\al$-generic filter over $V$).

Assume that $V[G]\models$``$(A_\al)_{\al<\om_1}$ is a
$\subseteq^*$-descending sequence in $\mc{I}^*$''. Then there is a
$\ga<\om_2$ such that this sequence is contained in $V[G_\ga]$. In
$V[G_{\ga+1}]$ there is a pseudointersection of this sequence,
i.e. it is not a tower in $\mc{I}^*$.

Assume on the contrary that $V[G]\models$``$(A_\al)_{\al<\om_2}$
is a tower in $\mc{I}^*$''. By the same argument we used in the proof of Theorem \ref{nomeagertower}, we can find a
$\ga<\om_2$ such that $V[G_\ga]\models$``$(A_\al)_{\al<\ga}$ is
a tower in $\mc{I}^*$''. If $A_\ga\in V[G_{\ga'}]$ for some $\ga'>\ga$, then the quotient forcing $\PP_{\ga'}/G_\ga\in V[G_\ga]$ is the $\ga'-\ga$-stage finite
support iteration of $\mbb{M}(\mc{I}^*)$ over $V[G_\ga]$ so it is
$\sigma$-centered, and hence, by Lemma \ref{sigmacentered}, there is
no pseudointersection $X\in \mc{I}^* \cap V[G_{\ga'}]$ of $(A_\al)_{\al<\ga}$, a contradiction because $A_\ga\in\mc{I}^*\cap V[G_{\ga'}]$.
\end{proof}

Now it is natural to ask the following:

\begin{que}\label{meagertowerinmat}
Let $\mc{I}$ be a tall $F_\sigma$ ideal or a tall analytic P-ideal.
Is there a meager tower in the iterated $\mbb{M}(\mc{I}^*)$ model?  By using Theorem \ref{first-thm} and a book-keeping argument, one can easily construct an iteration of forcing notions of the form $\mbb{M}(\mc{I}^*)$ such that in the extension $F_\sigma$ ideals and analytic P-ideals contain no towers. Is there a meager tower in this model? (We believe that the answer should be yes.)
\end{que}

As a stronger and perhaps more interesting variant of the first part of this question, one may ask whether it is possible that there are no cotowers in a tall analytic P-ideal $\mc{J}$ but there is a cotower in another tall analytic P-ideal $\mc{I}$, especially if $\mc{J}\subseteq\mc{I}$. For example, notice that up to a bijection between $\om$ and $2^{<\om}$, $\mc{I}_{1/n}\subseteq\mrm{tr}(\mc{N})\subseteq\mc{Z}$\footnotemark[1] (and hence $\mc{I}_{1/n}\leq_{\mrm{KB}}\mrm{tr}(\mc{N})\leq_{\mrm{KB}}\mc{Z}$). In particular, we can ask whether the following are consistent: (a) $\mc{Z}$ contains a cotower but $\mrm{tr}(\mc{N})$ does not. (b) $\mrm{tr}(\mc{N})$ contains a cotower but $\mc{I}_{1/n}$ does not. We know that (a) is consistent (see Example \ref{cotinZnotintr}) but (b) seems to be more difficult. We have some partial results concerning the general case when $\mc{J}$ and $\mc{I}$ are tall analytic P-ideals and $\mc{J}$ is ``fairly smaller'' than $\mc{I}$ (see Theorem \ref{reallygood}).

\footnotetext[1]{Define $\mc{I}_\mrm{tr}=\{A\subseteq 2^{<\om}:\sum\{2^{-|s|}:s\in A\}<\infty\}$ and $\mc{Z}_\mrm{tr}=\{A\subseteq 2^{<\om}:|A\cap 2^n|\cdot 2^{-n}\to 0\}$. Then $\mc{I}_\mrm{tr}$ is isomorphic to $\mc{I}_{1/n}$, $\mc{Z}_\mrm{tr}$ is isomorphic to $\mc{Z}$, and $\mc{I}_\mrm{tr}\subseteq\mrm{tr}(\mc{N})\subseteq\mc{Z}_\mrm{tr}$ .}

\section{Towers in ultrafilters}\label{tuf}

First of all, we show that if $\mc{F}$ is ``small'' then there are towers in $\mc{F}^+$.

\begin{fact}
Assume $\mc{F}$ is a meager filter. Then there is a tower in $\mc{F}^+$.
\end{fact}
\begin{proof}
Applying Theorem \ref{talag}, fix a partition $(P_n)_{n\in\om}$ which witnesses that $\mc{F}$ is meager, and let $(T_\al)_{\al<\ka}$ be a (classical) tower. For each $\al$ let $T'_\al=\bigcup\{P_n:n\in T_\al\}$. Clearly $T'_\al\in\mc{F}^+$ for each $\al$. We claim that $(T'_\al)_{\al<\ka}$ is a tower. It is clearly $\subseteq^*$-descending. Let $X\subseteq\om$ be infinite. Then the set $Y=\{n\in\om:X\cap P_n\ne\0\}$ is also infinite, and hence there is an $\al$ such that $Y\nsubseteq^*T_\al$. For all $n\in Y\setminus T_\al$ the set $P_n\cap (X\setminus T'_\al)=P_n\cap X$ is not empty, in particular $X\setminus T'_\al$ is infinite.
\end{proof}

What can we say about non-meager filters, in particular, about ultrafilters (when $\mc{F}^+=\mc{F}$)? First we show that basic cardinal characteristics do not contradict the existence of towers in ultrafilters. We define $\mrm{non}^*$ for filters as well (in the literature, especially in the case of ultrafilters, it is also called the {\em $\pi$-character} of the filter):
\[ \mrm{non}^*(\mc{F})=\mrm{non}^*(\mc{F}^*)=\min\big\{|\mc{X}|:\mc{X}\subseteq [\om]^\om\;\text{and}\;\forall\;F\in \mc{F}\;\exists\;X\in\mc{X}\;X\subseteq^* F\big\}.\]

\begin{fact}\label{nongeqb}
$\mrm{non}^*(\mc{F})\geq \mf{b}$ for every non-meager filter $\mc{F}$ (in particular, for every ultrafilter).
\end{fact}
\begin{proof}
Assume that $\mc{X}\subseteq [\om]^\om$ witnesses $\mrm{non}^*(\mc{F})<\mf{b}$.
We can find a partition $(Q_m)_{m\in\om}$ of $\om$ into finite sets such that $\forall$ $X\in\mc{X}$ $\forall^\infty$ $m$ $X\cap Q_m\ne\0$. Why? If $P=(P_n)_{n\in\om}$ and $Q=(Q_m)_{m\in\om}$ are partitions of $\om$ into finite sets then write $P\sqsubseteq^* Q$ iff $\forall^\infty$ $m$ $\exists$ $n$ $P_n\subseteq Q_m$. We know (see \cite{Bl}) that $\mf{b}=\{|\mc{P}|:\mc{P}$ is a family of partitions of $\om$ into finite sets and $\mc{P}$ is $\sqsubseteq^*$-unbounded$\}$. To every $X\in\mc{X}$ fix a partition $P_X=(P^X_n)_{n\in\om}$ of $\om$ into finite sets such that $X\cap P^X_n\ne\0$ for every $n$. Then there is a partition $Q=(Q_m)_{m\in\om}$ such that $P_X\sqsubseteq^* Q$ for every $X\in\mc{X}$, i.e. $\forall$ $X\in\mc{X}$ $\forall^\infty$ $m$ $\exists$ $n$ $P^X_n\subseteq Q_m$, in particular, $\forall$ $X\in\mc{X}$ $\forall^\infty$ $m$ $X\cap Q_m\ne\0$. Then $Q$ witnesses that $\mc{F}$ is meager, a contradiction.
\end{proof}

If $\mc{F}$ is an ultrafilter, then we know a bit more, namely, it is straightforward to see that $\mrm{non}^*(\mc{F})\geq \mf{r}$ where $\mf{r}$ is the {\em reaping number} and it is well-known that $\mf{r}\geq \mf{b}$ (see \cite{Bl}).

We show that the existence of an ultrafilter which does not contain towers is independent of $\mrm{ZFC}$. Actually, this is more or less known, we will prove a bit more.

We know (see \cite{ujref}) that consistently every ultrafilter contains towers. Here we present an alternative proof of this result. We will need the axiom $\mrm{NCF}$, {\em Near Coherence of Filters}: The Kat\v{e}tov-Blass order on ultrafilters in downward directed, that is, for every two ultrafilters $\mc{U}_0,\mc{U}_1$ there is an ultrafilter $\mc{V}$ such that $\mc{V}\leq_\mrm{KB}\mc{U}_0$ and $\mc{V}\leq_\mrm{KB}\mc{U}_1$. For the consistency of $\mrm{NCF}$ and its applications see \cite{ncf1}, \cite{ncf2}, and \cite{ncf3}. For example, we know that $\mrm{NCF}$ implies that $\mf{b}<\mf{d}$.

Notice that the first part of Theorem \ref{filgen} follows from Fact \ref{nongeqb} because (by definition) $\mrm{non}^*(\mc{F})\leq \cof^*(\mc{F})$. As we already mentioned (after Theorem \ref{filgen}), in the case of $\mf{b}<\mf{d}$ (e.g. under $\mrm{NCF}$) the example of $\mf{b}$ sets which generate a non-meager filter, can be chosen as a tower.

\begin{thm}
$\mrm{NCF}$ implies that every ultrafilter contains a tower.
\end{thm}
\begin{proof}
Let $\mc{U}_0$ be an arbitrary (non-principal) ultrafilter on $\om$. Fix a tower $(T_\al)_{\al<\mf{b}}$ such that the generated filter $\mc{F}$ is not meager, and let $\mc{U}_1$ be an extension of $\mc{F}$ to an ultrafilter. Then there is an ultrafilter $\mc{V}\leq_\mrm{KB}\mc{U}_0,\mc{U}_1$. We show that $\mc{V}$ and $\mc{U}_0$ also contain towers.

\smallskip
Fix a finite-to-one $f:\om\to\om$ such that $f^{-1}[X]\in \mc{U}_1$ for every $X\in\mc{V}$; actually $f^{-1}[X]\in\mc{U}_1$ iff $X\in\mc{V}$ because they are ultrafilters, hence $f[T_\al]\in\mc{V}$ for every $\al$. We show that the $\subseteq^*$-decreasing sequence $(f[T_\al])_{\al<\mf{b}}$ in $\mc{V}$ is a tower. Consider $f$ as a partition of $\om$ into nonempty finite sets: $(P_n=f^{-1}[\{n\}]:n\in\ran(f))$. Let $X\subseteq\om$ be infinite, we can assume that $X\subseteq \ran(f)$. We show that $X$ is not a pseudointersection $(f[T_\al])_{\al<\mf{b}}$.

If $f^{-1}[X]\in\mc{F}^*$, i.e. $f^{-1}[X]\cap T_\al$ is finite for an $\al$, then $|X\setminus f[T_\al]|=\om$, and we are done. Assume that $f^{-1}[X]\in\mc{F}^+$. Applying that $\mc{F}$ is not meager, we know that $\mc{F}\upharpoonright f^{-1}[X]$ is not meager either. Since $(P_n:n\in X)$ is a partition of $f^{-1}[X]$, there is an $\al$ such that $\exists^\infty$ $n\in X$ $T_\al\cap P_n=\0$. Clearly, $|X\setminus f[T_\al]|=\om$, and hence $X$ is not a pseudointersection $(f[T_\al])_{\al<\mf{b}}$.

\smallskip
Now fix a finite-to-one $g:\om\to\om$ such that $g^{-1}[X]\in\mc{U}_0$ iff $X\in\mc{V}$. It is easy to see that if $(E_\al)_{\al<\ka}$ is a tower in $\mc{V}$ then $(g^{-1}[E_\al])_{\al<\ka}$ is a tower in $\mc{U}_0$.
\end{proof}

To show that consistently there are ultrafilters which contain no towers, we will use an old but surprising result of set-theoretic topology. Let $X$ be a Hausdorff topological space. The {\em $\pi$-weight} of $X$  is defined as follows:
\[ \pi w(X)=\min\big\{|\mc{B}|:\mc{B}\;\text{is a family of open sets}\;\text{and}\;\forall\;U\subseteq X\;\text{open}\;\exists\;V\in\mc{B}\;V\subseteq U\big\}.\]
Let $\ka>\om$ be a cardinal. A set $Y\subseteq X$ is a {\em $P_\ka$-set} (in $X$) if each intersection of $<\ka$ many neighborhoods of $Y$ is again a neighborhood of $Y$.

We assume that the reader is familiar with the basics of $\be\om$ and $\om^*=\be\om\setminus\om$. For example, we know that $\pi w(\om^*)=\mf{c}$, and that every tower $(T_\al)_{\al<\ka}$ generates a nowhere dense closed $P_{\mrm{cf}(\ka)}$-set $\{\mc{U}\in\om^*:\forall$ $\al<\ka$ $T_\al\in\mc{U}\}$ in $\om^*$.

\begin{thm} {\em (see \cite{kmm})} If $X$ is compact and $\pi w(X)\leq\ka(>\om)$, then closed nowhere dense $P_\ka$-sets do not cover $X$.
\end{thm}

Applying this result to $\om^*$, we obtain the following:

\begin{cor}
Assume $\mrm{CH}$. Then there is an ultrafilter which contains no towers.
\end{cor}

\begin{rem}
Notice if $\mrm{NCF}$ holds and $\mf{c}$ is regular then all towers of height $\mf{c}$ are meager: Suppose not, then by the above argument every ultrafilter contains a tower of length $\mf{c}$, and therefore $\om^*$ is covered by $P_\mf{c}$ sets, contradiction. (We shall point out that it is unclear whether a tower of height $\mf{c}$ can exist under NCF.)
\end{rem}

Under stronger conditions, namely assuming $\diamondsuit_{\om_1}$, even more interesting counterexamples can be constructed:

\begin{thm}\label{seven} {\em (see \cite[Thm. 30]{seven})}
Assume $\diamondsuit_{\om_1}$. Then there exists an ultrafilter $\mc{U}$ on $\om$ such that for every $\mc{X}\subseteq \mc{U}$ either there is a $\mc{Y}\in [\mc{X}]^\om$ without a pseudointersection in $\mc{U}$, or  $\mc{X}$ has a pseudointersection. In particular, $\mc{U}$ does not contain towers.
\end{thm}

The above ultrafilter in fact is very far from being a P-point (see below) --- every P-subfilter of this ultrafilter is
meager. In fact, under $\mrm{CH}$ every P-point is generated by a tower. The following theorem shows that under $\neg\mrm{CH}$, being a non P-point is not essential: it is consistent that there is even a selective ultrafilter which does not contain towers. Recall that an ultrafilter $\mc{U}$ on $\om$ is {\em selective} if one of the following equivalent properties holds:
\begin{itemize}
\item[(S1)] For every partition $\om=\bigcup\{A_n:n\in\om\}$ into sets $A_n\notin\mc{U}$, there is an $X\in\mc{U}$ such that $|A_n\cap X|\leq 1$ for every $n$.
\item[(S2)] For every $f:\om\to\om$ there exists a $Y\in \mc{U}$ such that $f\upharpoonright Y$ is one-to-one or constant.
\item[(S3)] For every $n,k\in\om\setminus 2$ and coloring $c:[\om]^n\to k$ there is a $c$-homogeneous $H\in\mc{U}$ (that is, $|c[[H]^n]|=1$).
\item[(S4)] For every analytic $\mc{A}\subseteq [\om]^\om$ there is an $H\in\mc{U}$ such that $[H]^\om\subseteq\mc{A}$ or $[H]^\om\cap\mc{A}=\0$.
\end{itemize}
If $\mc{U}$ is selective, then it is a {\em P-point}, that is, one of the following equivalent properties holds:
\begin{itemize}
\item[(P1)] Every countable subset of $\mc{U}$ has a pseudointersection in $\mc{U}$.
\item[(P2)] For every $f:\om\to\om$ there exists an $Y\in \mc{U}$ such that $f\upharpoonright Y$ is finite-to-one or constant.
\item[(P3)] For every $k\in\om\setminus 2$ and coloring $c:[\om]^2\to k$ there is an $H\in\mc{U}$ and an $f:\om\to\om$ such that $\{\{x,y\}\in [H]^2: f(x)\leq y\}$ is $c$-homogeneous.
\end{itemize}
For more details on peculiar ultrafilters see e.g. \cite{blassuf}.

\smallskip
If $\ka>\om$ is regular and $S\subseteq\ka$ is stationary, then  $\diamondsuit(S)$ stands for the classical diamond principle on $S$, that is,
it holds iff there is sequence $(A_\al)_{\al\in S}$ such that $\{\al\in S:X\cap\al=A_\al\}$ is stationary in $\ka$ for every $X\subseteq\ka$. In particular, $\diamondsuit_\ka$ stands for $\diamondsuit(\ka)$. $\diamondsuit(S)$ clearly implies that $2^{<\ka}=\ka$. Conversely, if $2^{<\ka}=\ka$ and $S\subseteq\ka$ is stationary, then $(2^{<\ka},\supseteq)$ is $<\ka$-closed, $\ka^+$-cc, and  forces $\diamondsuit(S)$.

For regular $\ka>\om_1$ let $S^\ka_{{\geq\om_1}}=\{\al<\ka:\mrm{cf}(\al)>\om\}$. If $\lam\geq \om_1$ and $S\subseteq \{\al<\lam^+:\mrm{cf}(\al)\ne\mrm{cf}(\lam)\}$ is stationary in $\lam^+$, then the principle $\diamondsuit(S)$ is equivalent with $2^\lam=\lam^+$ (see \cite{shdiam}). In particular, if $\ka=\lam^+=2^\lam>\om_2$ then $\diamondsuit(S^\ka_{\geq\om_1})$ holds because $\{\al<\ka:\mrm{cf}(\al)>\om,\mrm{cf}(\al)\ne\mrm{cf}(\lam)\}\subseteq S^\ka_{\geq\om_1}$ is stationary.

In the next theorem we will need $\diamondsuit(S^\mf{c}_{\geq\om_1})$ together with $\mf{t}=\mf{c}$. If we start with any model $V$ satisfying $2^{<\mf{c}}=\mf{c}=\mf{t}>\om_1$ then $V^{(2^{<\mf{c}},\supseteq)}\models\diamondsuit(S^\mf{c}_{\geq\om_1})+\mf{t}=\mf{c}$ because $(2^{<\mf{c}},\supseteq)$ does not add new sequences of ground model elements of length $<\mf{c}$.

\begin{thm}
Assume $\diamondsuit(S^\mf{c}_{{\geq\om_1}})$ holds and that $\mathfrak t={\mathfrak c}$. Then there is a selective ultrafilter which does not contain a tower. Moreover it does not contain a tower of length $<{\mathfrak c}$ in any forcing extensions which preserve regularity of $\mf{c}$.
\end{thm}
\begin{proof}
First of all, notice that $\diamondsuit(S^\mf{c}_{{\geq\om_1}})$ is equivalent with the following statement: There is a sequence $(A_\al=(A^\al_\be)_{\be<\al}:\al\in  S^\mf{c}_{\geq\om_1})$ such that for every sequence $(T_\al)_{\al<\mf{c}}$ in $\mc{P}(\om)$ the set
\[ \big\{\al\in S^\mf{c}_{\geq\om_1}:(T_\be)_{\be<\al}=A_\al\big\}\;\text{is stationary in}\;\mf{c}.\]
Now let $(X_\al)_{\al<\mf{c}}$ be an enumeration of $[\om]^\om$ and $(P_\al=(P^\al_n)_{n\in\om}:\al<\mf{c})$ be an enumeration of all partitions of $\om$ into infinitely many nonempty sets such that each partition appears cofinally often.

Before describing the construction we will
need a piece of notation: If  $\mathcal{F}$ is a filter and $\mathcal D$ is a family of $\subseteq^*$-descending sequences
of subsets of $\omega$ of length $<\mf{c}$ and of uncountable cofinality (if $D\in\mc{D}$ then we will write $D=(D_\al)_{\al<\mrm{\ell}(D)}$) we say that $\mc{D}$ {\em uncovers} $\mc{F}$ and write $\mrm{UC}({\mathcal{F}},{\mathcal D})$ if either $\mc{D}=\0$ or the following holds
\[ \forall\;F\in \mathcal{F}\;\,\forall\;\mathcal{D}'\in [\mc{D}]^{<\omega}\setminus\{\0\}\;\,\exists\; x\in\prod\big(\mrm{\ell}(D):D\in\mc{D}'\big)\;\,\big|F\setminus\bigcup\big\{D_{x(D)}:D\in\mc{D}'\big\}\big|=\om.\]
We shall recursively build an increasing  sequence of filters  $(\mathcal{F}_\al)_{\al<\mf{c}}$ and
families of descending sequences $(\mc{D}_\al)_{\al<\mf{c}}$ such that the following
conditions are satisfied for every $\al<\mf{c}$:
\begin{enumerate}
  \item  $\chi(\mathcal{F}_\alpha),|\mathcal{D}_\alpha|\leq|\alpha|+\om$
     for each $\alpha<{\mathfrak c}$ where $\chi(\mc{F})=\min\{|\mc{B}|:\mc{B}$ generates $\mc{F}\}$;
  \item  $\mrm{UC}(\mathcal{F}_\alpha,\mathcal{D}_\alpha)$ for each $\alpha<\mf{c}$;
  \item either $X_\alpha\in \mathcal{F}_{\alpha+1}$ or
     $\omega\setminus X_\alpha\in \mathcal{F}_{\alpha+1}$;
  \item if $P^\al_n\in \mathcal{F}^*_\alpha$ for every $n$, then $|P^\al_n\cap Y|\leq 1$ for some $Y\in \mathcal{F}_{\alpha+1}$.
  \item if  $\mrm{cf}(\alpha)>\om$ and $\mrm{UC}(\mathcal{F}_\alpha,\mathcal{D}_\alpha\cup\{A_\alpha\})$ then
     $A_\alpha\in \mathcal{D}_{\alpha+1}$.
\end{enumerate}

We let $\mc{F}_0=\mrm{Fin}^*$ and $\mathcal{D}_0=\emptyset$. At limit steps we take unions and clearly all conditions will be satisfied (actually, we need to check only (1) and (2)).

We now show that
we can handle the successor steps. So assume $\mathcal{F}_\alpha$ and $\mathcal{D}_\alpha$ have
been constructed. We first construct $\mathcal{D}_{\alpha+1}$ in the obvious way to
satisfy (5): If $\mrm{cf}(\al)>\om$ and $\mrm{UC}(\mathcal{F}_\alpha,\mathcal{D}_\alpha\cup\{A_\alpha\})$ holds, then let $\mathcal{D}_{\alpha+1}=\mathcal{D}_\alpha\cup\{A_\alpha\}$. Otherwise let
$\mathcal{D}_{\alpha+1}=\mathcal{D}_\alpha$. We now extend $\mathcal{F}_\alpha$ in two steps to
guarantee (3) and (4) while preserving (2) ((1) will hold trivially).

\smallskip
\underline{First step}: We extend $\mc{F}_\al$ to $\mc{F}'_\al$ to guarantee (3) while preserving (2). If either $X_\alpha\in \mathcal{F}_\alpha$ or
$\omega\setminus X_\alpha\in \mathcal{F}_\alpha$, then there is nothing to do in the first step, so let $\mathcal{F}'_\alpha=\mathcal{F}_\alpha$. If neither $X_\al$ nor $\om\setminus X_\al$ belong to $\mc{F}_\al$ and the condition
$\mrm{UC}(\mrm{fr}(\mathcal{F}_\alpha\cup\{X_\alpha\}),\mathcal{D}_{\alpha+1})$ holds, let
$\mathcal{F}'_\alpha$ be $\mrm{fr}(\mathcal{F}_\alpha\cup\{X_\alpha\})$. Otherwise let $\mathcal{F}'_\alpha$
be $\mrm{fr}(\mathcal{F}_\alpha\cup\{\omega\setminus X_\alpha\})$.

\begin{claim}
$\mrm{UC}(\mathcal{F}'_\alpha,\mathcal{D}_{\alpha+1})$ holds.
\end{claim}
\begin{proof}[Proof of Claim.]
Assume on the contrary that we can choose  $F_0\in \mathcal{F}_\alpha$
and $\mathcal{C}_0\in [\mathcal{D}_{\alpha+1}]^{<\omega}$ so that
the set $F_0\cap (\om\setminus X_\alpha)$ and $\mathcal{C}_0$ witness the failure of $\mrm{UC}(\mathcal{F}'_\alpha,\mathcal{D}_{\alpha+1})$. By assumption, we can also choose $F_1\in \mathcal{F}_\alpha$
and $\mathcal{C}_1\in [\mathcal{D}_{\alpha+1}]^{<\omega}$ so that
the set $F_1\cap X_\alpha$ and $\mathcal{C}_1$ witness the failure of $\mrm{UC}(\mrm{fr}(\mathcal{F}_\alpha\cup\{X_\al\}),\mathcal{D}_{\alpha+1})$.
Since $\mrm{UC}(\mathcal{F}_\alpha,\mathcal{D}_{\alpha+1})$, we can choose a function $x\in\prod(\mrm{\ell}(D):D\in\mathcal{C}_0\cup\mathcal{C}_1)$ such that
\[
  \big|F_0\cap F_1\setminus\bigcup\big\{D_{x(D)}:D\in\mc{C}_0\cup\mc{C}_1\big\}\big|=
  \omega
\]
However, we know that
\[
  \big|F_0\cap(\om\setminus X_\alpha)\setminus\bigcup\big\{ D_{x(D)}:D\in\mc{C}_0\big\}\big|<\omega\;\;\text{and}\;\;
  \big|F_1\cap X_\al\setminus\bigcup\big\{D_{x(D)}:D\in\mc{C}_1\big\}\big|
  <\omega,
\]
and of course, $F_0\cap F_1\setminus\bigcup\{D_{x(D)}:D\in\mc{C}_0\cup\mc{C}_1\}$ is covered by the union of these sets, a contradiction.
\end{proof}

\underline{Second step}: In this step we extend  $\mc{F}'_\al$ to $\mc{F}_{\al+1}$ to guarantee (4) while preserving (2). If $P^\al_n\in (\mc{F}'_\al)^+$ for some $n$, then let $\mc{F}_{\al+1}=\mc{F}'_\al$. Assume now that $P^\al_n\in(\mc{F}'_\al)^*$ for every $n$, consider the
partial order
\[
   \PP= \big(\big\{a\in[\omega]^{<\omega}:\forall\;n\in\omega\;|a\cap P^\alpha_n|\leq 1\big\},
  \supseteq\big),
\]
and fix a base $\mc{B}$ of $\mc{F}'_\al$ of size $\leq|\al|+\om$. We can assume that every $F\in\mc{F}'_\al$ contains an element of $\mc{B}$. For each $B\in\mathcal{B}$, $\mathcal{D}\in[\mathcal{D}_{\alpha+1}]^{<\omega}$, and $n\in\om$ choose
$x_{B,\mathcal{D},n}\in\prod(\mrm{\ell}(D):D\in\mc{D})$ witnessing
$\mrm{UC}(\mathcal{F}'_\alpha,\mathcal{D}_{\alpha+1})$ for the set $B\cap \bigcap\{\omega\setminus P^\alpha_i:i<n\}\in\mc{F}'_\al$ and $\mathcal{D}$, that is,
\[ \big|B\cap \bigcap\big\{\omega\setminus P^\alpha_i:i<n\big\}\setminus\bigcup\big\{D_{x_{B,\mc{D},n}(D)}:D\in\mc{D}\big\}\big|=\om.\]
Define $y_{B,\mathcal D}\in\prod(\mrm{\ell}(D):D\in\mc{D})$ as $y_{B,\mc{D}}(D)=\sup\{x_{B,\mathcal{D},n}(D):n<\om\}$
(note that this can safely be done since each sequence in $\mathcal{D}_{\alpha+1}$ has
uncountable cofinality) and consider the set
\[
  E^k_{B,\mathcal{D}} = \big\{a\in\PP:a\cap B\setminus
               \bigcup\big\{D_{y_{B,\mathcal{D}}(D)}:D\in\mc{D}\big\}
               \nsubseteq k
             \big\}.
\]
This set is  dense in $\PP$ for every $B$, $\mc{D}$, and $k\in\om$: Applying that $x_{B,\mathcal D,n}(D)\leq y_{B,\mathcal D}(D)$ holds for every $n\in\om$ and $D\in\mc{D}$, it is easy to see that
\[ B\cap\bigcap\big\{\omega\setminus P^\alpha_i:i<n\big\}\setminus \bigcup\big\{D_{y_{B,\mathcal D}(D)}:D\in\mathcal D\big\}\]
is infinite for each $n$, denote this set by $K_n$. Now if
$a\in \PP$, fix $n$ such that $a\cap P_m^\alpha=\emptyset$ for each $m\geq n$, and pick an element $j\in K_n\setminus k$.
Then $a\cup\{j\}\in \PP$ since
$j\not\in P_i^\alpha$ for any $i<n$, and $a\cup\{j\}\in E^k_{B,\mc{D}}$ follows as well.

Since $\mf{c}=\mf{t} \leq \mrm{add}(\mc{M})\leq \mrm{cov}(\mc{M})$, Martin's Axiom for countable posets holds (see [10]). In particular, we can find a filter $G$ on $\PP$ which meets every $E_{B,\mc{D}}^k$
(by (1) there are $<\mf{c}$ many of them).

Let $\mathcal{F}_{\alpha+1}=\mrm{fr}(\mathcal{F}'_\alpha\cup\{\bigcup G\})$. It is straightforward to see that $\mc{F}_{\al+1}$ is a filter (because $\bigcup G$ must be $\mc{F}'_\al$-positive) and that $\mrm{UC}(\mathcal{F}_{\alpha+1},\mathcal{D}_{\alpha+1})$ holds: If $F\in\mc{F}'_\al$ and $\mc{D}\in [\mc{D}_{\al+1}]^{<\om}$, then there is a $B\in\mc{B}$ contained in $F$ and $y_{B,\mc{D}}$ witnesses that $\mrm{UC}(\mathcal{F}_{\alpha+1},\mathcal{D}_{\alpha+1})$ holds for $B\cap (\bigcup G)$ and $\mc{D}$, and hence for $F\cap (\bigcup G)$ and $\mc{D}$ as well.

\smallskip
This finishes the recursive construction. Finally let
$\mathcal{U}=\bigcup\{\mathcal{F}_\alpha:\al<\mf{c}\}$. It is straightforward to check that  $\mc{U}$ is an ultrafilter
(guaranteed by (3)) which is selective (guaranteed by (4)).
We next show that it does not contain a tower.
Let $(T_\al)_{\al<\mathfrak{c}}$ be a tower (there are no shorter
towers as $\mathfrak t=\mathfrak c$).

\begin{claim}
For each  $\alpha<\mathfrak c$ there is $\be<\mf{c}$ such that
$\mrm{UC}(\mathcal{F}_\alpha,\mathcal{D}_\alpha\cup\{( T_\ga)_{\ga<\be}\})$, and hence also $\mrm{UC}(\mathcal{F}_\alpha,\mathcal{D}_\alpha\cup\{( T_\ga)_{\ga<\be'}\})$ for every $\be'\geq\be$, holds.
\end{claim}
\begin{proof}[Proof of Claim.]
Fix  $\alpha<\mathfrak{c}$ and a base $\mc{B}$ for
$\mathcal{F}_\alpha$ of size $\leq|\alpha|+\om$. Since $\mrm{UC}(\mathcal{F}_\alpha,\mathcal{D}_\alpha)$
for each $B\in\mathcal{B}$ and each
$\mathcal{D}\in[\mathcal{D}_\alpha]^{<\omega}$ we can fix an $x\in\prod(\ell(D):D\in\mc{D})$ such that
$F_{B,\mathcal{D}}=B\setminus\bigcup\{D_{x(D)}:D\in\mc{D}\}$
is infinite. Since $(T_\al)_{\al<\mathfrak{c}}$ does not have a pseudointersection, for each pair $(B,\mc{D})$ there is a $\gamma(B,\mc{D})<\mathfrak{c}$ such that
$|F_{B,\mathcal{D}}\setminus T_{\gamma(B,\mathcal{D})}|=\omega$. It is easy
to check that if $\be>\gamma(B,\mathcal{D})$ for every $B\in\mathcal{B}$ and $\mathcal{D}\in[\mathcal{D}_\alpha]^{<\omega}$, then it works as required.
\end{proof}

Applying the above Claim, it is easy to see that the set $C= \{\alpha:\forall$ $\be<\al$ $\mrm{UC}(\mathcal{F}_\be,\mathcal{D}_\be
  \cup\{(T_\gamma)_{\gamma<\alpha}\})\}$
is a club in $\mf{c}$ and $\mrm{UC}(\mc{F}_\al,\mc{D}_\al\cup\{(T_\ga)_{\ga<\al})\})$ holds for every $\al\in C$. If $\al\in C$ and
$A_\al=(T_\ga)_{\ga<\alpha}$, then $(T_\ga)_{\ga<\al}\in\mc{D}_{\al+1}$ because of (5). This implies that $T_\alpha\notin \mathcal{U}$ because otherwise $T_\al\in\mc{F}_{\al'}$ for some $\al'>\al$ and the pair $T_\al,\{(T_\ga)_{\ga<\al}\}$ would witness that $\mrm{UC}(\mc{F}_{\al'},\mc{D}_{\al'})$ fails.

\smallskip
Finally we show that  $\mathcal{U}$ does not contain a tower of length $<\mathfrak{c}$ in
any forcing extension (actually, in any extension) where $\mathfrak{c}$ remains regular. To this end we will show that
there is family $\{X_\alpha:\alpha<\mathfrak{c}\}$ such that for each $F\in\mathcal{U}$
the set $\{\alpha<\mf{c}:|X_\alpha\setminus F|=\omega\}$ is of size $<\mathfrak{c}$. The result then follows immediately.  To construct the family, it is sufficient to let
\(X_\alpha\) be any pseudointersection of $\mathcal{F}_\alpha$ which exists as $\chi(\mathcal{F}_\alpha)\leq|\alpha|+\omega<\mathfrak{c}=\mathfrak{p}$.
\end{proof}

Next we show that $\mf{t}=\mf{c}$ is not necessary to construct a selective ultrafilter which does not contain towers.

\begin{thm}
Assume $\mrm{GCH}$ and let $\om<\lam<\ka$ be regular cardinals. Then in a cardinal preserving forcing extension $\mf{t}=\lam$, $\mf{c}=\ka$, and there is a selective ultrafilter which does not contain a tower.
\end{thm}
\begin{proof}
Start with a model $V$ of $\mrm{GCH}$, let $\mbb{B}=(\lam^{<\lambda},\supseteq)$ defined in $V$, and let $\PP\in V$ be a ccc forcing such that $\Vdash_\PP \mathfrak t=\mathfrak c=2^{<\mf{c}}=\ka$. Applying \cite[Lem. 15.19]{Jech}, we know that $\mbb{B}^V$ is $<\lam$-distributive in $V^{\PP}$, that is, in $V^\PP$ intersection of less than $\lam$ many dense open subsets of $\mbb{B}^V$ is dense, or equivalently, $V^\PP\models$``$\mbb{B}^V$ does not add new sequences of elements of $V^\PP$ of length $<\lam$''.

Next force $\diamondsuit(S^\ka_{\geq\om_1})+\mf{c}=\mf{t}(=\ka)$ with a $<\ka$-closed and $\ka^+$-cc forcing $\QQ\in V^\PP$ (e.g. $V^\PP\models\dot{\mbb{Q}}=(2^{<\mf{c}},\supseteq)$). $\mbb{B}^V$ still remains $<\lam$-distributive in $V^{\PP\ast\QQ}$:
$\lam < \ka$ and hence $\QQ$ adds neither new subsets of $\mbb{B}^V$ nor new sequences of subsets of $\mbb{B}^V$ of length $<\lam$. In $V^{\PP\ast\QQ}$ we can apply the previous theorem and fix a selective $\mathcal{U}$ which does not contain a tower and does not contain towers of length $<\mf{c}$ in any forcing extensions of $V^{\PP\ast\QQ}$ which preserve regularity of $\mf{c}$. In the further
extension by $\mbb{B}^V$, that is, in the model $V^{(\PP\ast\QQ)\times\mbb{B}}$ the following holds:

(a) $\mf{c}=\ka=\mrm{cf}(\ka)$ and $\mathcal{U}$ is a selective ultrafilter: $\mbb{B}^V$ is $<\lam$-distributive and $\lam^+$-cc in $V^{\PP\ast\QQ}$ hence preserves cofinalities and does not add new reals.

(b) $\mc{U}$ does not contain towers: It cannot contain towers of length $<\ka$ because of (a) and properties of $\mc{U}$. To show that $\mc{U}$ cannot contain towers of length $\ka$, (working in $V^{\PP\ast\QQ}$) assume that $(\dot{T}_\al)_{\al<\ka}$ is a sequence of $\mbb{B}^V$-names for subsets of $\om$ and $b\vd$``$(\dot{T}_\al)_{\al<\ka}$ is a $\subseteq^*$-descending sequence in $\mc{U}$'' for some $b\in \mbb{B}^V$. For every $\al<\ka$ we can fix a $b_\al\leq b$ and an $A_\al\subseteq\om$ such that $b_\al\vd\dot{T}_\al=A_\al$ ($\mbb{B}^V$ does not add new reals). Since $\ka=\mrm{cf}(\ka)>\lam=|\mbb{B}|$, we have that $b_\al=b'\leq b$ for cofinally many $\al$, and therefore ($b'$ forces that) a cofinal subsequence of $(\dot{T}_\al)_{\al<\ka}$ is in $V^{\PP\ast\QQ}$ and so it has a pseudointersection (because $\mc{U}$ does not contain towers in $V^{\PP\ast\QQ}$).

(c) $\mathfrak{t}=\lambda$: Notice that $V^{\PP\ast\QQ}\models$``$\mbb{B}^V$ can be embedded into $([\omega]^\omega,\subseteq^*)$''. Applying that $\mf{t}=\ka>\lam$ in $V^{\PP\ast\QQ}$, by recursion on the length of $b\in\mbb{B}^V=\lam^{<\lam}\cap V$ we can define $e:\mbb{B}^V\to [\om]^\om$ as follows: Let $e(\0)=\om$. If $e(b)$ is defined then fix an almost disjoint family $\{A_\al:\al<\lam\}$ on $e(b)$ and define $e(b^\frown (\al))= A_\al$. If $\dom(b)=\ga<\lam$ is limit and $e(b\upharpoonright \xi)$ is defined for every $\xi<\ga$ then let $e(b)$ be a pseudointersection of this sequence. The $\mbb{B}^V$-generic $g\in\lam^\lam$ defines the $\subseteq^*$-decreasing sequence $(e(g\upharpoonright\ga))_{\ga<\lam}$ in $[\om]^\om$ which has to be a tower because of genericity and the fact that $\mbb{B}^V$ does not add new reals. Applying the $<\lam$-distributivity of $\mbb{B}^V$ there are no towers of length $<\lam$ in the final model.
\end{proof}

\section{$\mc{I}$-Luzin sets, $\mc{I}$-inaccessibility, cardinal invariants, and Kat\v{e}tov-Blass reducibility}\label{luzinsec}
We need a natural generalization of Sierpi\'nski and Luzin sets to
any ideal.
\begin{definition}
Let $I$ be an ideal on a set $X$. An uncountable(!) set $S\subseteq X$ is {\em $I$-Luzin} if
$|S\cap A|\le\om$ for each $A\in I$.
\end{definition}

The classical Sierpi\'nski/Luzin sets are exactly the
$\mc{N}$-/$\mc{M}$-Luzin sets where $\mc{N}$ is the $\sigma$-ideal
of null subsets of $2^\om$ (with respect to the usual product measure) and $\mc{M}$ is the $\sigma$-ideal of meager subsets of $2^\om$. It is easy to see that if there is an $I$-Luzin subset $H$ of $X$, then $\mathrm{non}(I)=\om_1$ and $|H|\leq\mrm{cov}(I)$; and if $\mathrm{cov}(I)=\mathrm{cof}(I)=\om_1$ then there is an $I$-Luzin set.

The following definition is motivated by \cite{Her}.

\begin{definition}
Let $I$ be an ideal on a set $X$. We say that $X$ is {\em
$I$-accessible} if there is an $\subseteq$-increasing sequence $(A_\al)_{\al<\ka}$ in $I$ which covers $X$, otherwise we say that $X$ is {\em $I$-inaccessible}.
\end{definition}

For example, if $\mathrm{add}(I)=\mathrm{cov}(I)$ or $\mathrm{non}(I)=|X|$, then $X$ is $I$-accessible but if
$\mathrm{non}(I)<\mathrm{cov}(I)$ then $X$ is $I$-inaccessible. Clearly, a $\subseteq^*$-increasing sequence $(A_\al)_{\al<\ka}$ in $\mc{I}$ is a cotower in $\mc{I}$ iff
$[\om]^\om=\bigcup_{\al<\ka}\widehat{A}_\al$ (recall that $\widehat{A}=\{X\in[\om]^\om:|X\cap A|=\om\}$ and that $\widehat{\mc{I}}$ is the ideal generated by $\{\widehat{A}:A\in\mc{I}\}$). In other
words, if there exists a cotower in $\mc{I}$, then $[\om]^\om$ is $\widehat{\mc{I}}$-accessible.

\smallskip
We summarize our observations (see also Fact \ref{fact1}) in the following diagram where $I=\widehat{\mc{I}}$ for some tall ideal $\mc{I}$ on $\om$:
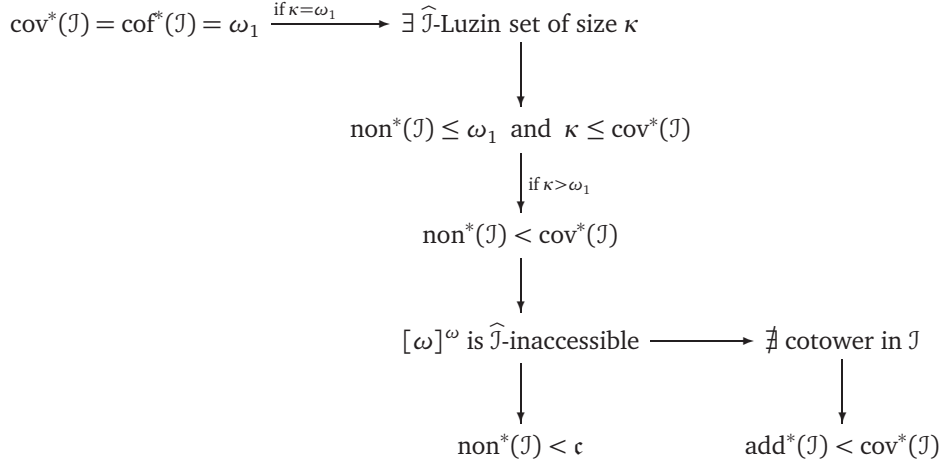
\begin{diagram}
\mrm{cov}^*(\mc{I})=\mrm{cof}^*(\mc{I})=\om_1 & \rTo^{\text{if}\;\ka=\om_1} & \exists\;\widehat{\mc{I}}\text{-Luzin set of size}\;\ka    &  & \\
&& \dTo                                        &                             & \\
&&    \mathrm{non}^*(\mc{I})\le\om_1\;\;\text{and}\;\;\ka\leq\mathrm{cov}^*(\mc{I})   &       &     \\
&& \dTo_{\text{if}\;\ka>\om_1}                                        &  & \\
&&    \mathrm{non}^*(\mc{I})<\mathrm{cov}^*(\mc{I})   &       &     \\
&& \dTo                                        &  & \\
&& [\om]^\om\;\text{is}\;\widehat{\mc{I}}\text{-inaccessible}                        & \rTo & \nexists\;\text{cotower in}\;\mc{I}\\
& & \dTo                                    && \dTo \\
& &  \mathrm{non}^*(\mc{I})<\mathfrak{c} & & \mrm{add}^*(\mc{I})<\mrm{cov}^*(\mc{I})
\end{diagram}

In the rest of the paper, we will discuss the possible reverse implications in this diagram. From now on, if $\mc{I}$ is an ideal on $\om$, we simplify our notations and talk about $\mc{I}$-Luzin sets (instead of $\widehat{\mc{I}}$-Luzin sets) and about $\mc{I}$-(in)accessibility (instead of $\widehat{\mc{I}}$-(in)accessibility).

The Kat\v{e}tov-Blass preorder plays a very important role when studying the properties in the diagram above: We already know that if $\mc{I}\leq_{\mrm{KB}}\mc{J}$ and $\mc{I}$ contains a cotower then $\mc{J}$ also contains one. Furthermore, we have the following ((a) has been known for a long time):

\begin{fact}\label{KBcons}
Assume that $\mc{I}\leq_{\mrm{KB}}\mc{J}$. Then the following hold:
\begin{itemize}
\item[(a)] $\mrm{non}^*(\mc{I})\leq\mrm{non}^*(\mc{J})$ and $\mrm{cov}^*(\mc{I})\geq\mrm{cov}^*(\mc{J})$.
\item[(b)] If there is an $\mc{J}$-Luzin set of size $\ka$, then there is an $\mc{I}$-Luzin set of size $\ka$ as well.
\item[(c)] If $[\om]^\om$ is $\mc{J}$-inaccessible then it is also $\mc{I}$-inaccessible.
\end{itemize}
\end{fact}
\begin{proof}
Let $f:\om\to\om$ be finite-to-one such that $f^{-1}[A]\in\mc{J}$ for every $A\in\mc{I}$.

(a): It is easy to see that if $\mc{X}\subseteq [\om]^\om$ and there is no $B\in\mc{J}$ having infinite intersection with every $X\in\mc{X}$, then there is no $A\in\mc{I}$ such that $|A\cap f[X]|=\om$ for every $X\in\mc{X}$ (notice that $f[X]\in [\om]^\om$). Similarly, if $\mc{D}\subseteq\mc{I}$ is a star-covering family, that is, for every $X\in [\om]^\om$ there is a $D\in \mc{D}$ such that $|X\cap D|=\om$, then $\{f^{-1}[D]:D\in\mc{D}\}\subseteq\mc{J}$ is also a star-covering family.

\smallskip
(b): When we are talking about $\mc{I}$-Luzin families, we always assume that $\mc{I}$ is tall (otherwise, $\widehat{\mc{I}}$ is not defined). Let $\mc{X}\subseteq [\om]^\om$ be a $\mc{J}$-Luzin family of size $\ka\geq\om_1$. First of all, notice that $\{ f[X]=Y:X\in \mc{X}\}$ is countable for every $Y\in [\om]^\om$ because otherwise if $A\in\mc{I}\cap [Y]^\om$ then $f^{-1}[A]$ would witness that $\mc{X}$ is not a $\mc{J}$-Luzin family. Therefore the family $\mc{X}'=\{f[X]:X\in\mc{X}\}\subseteq [\om]^\om$ is of size $\ka$. It is trivial to show that $\mc{X}'$ is an $\mc{I}$-Luzin family.

\smallskip
(c): Assume that $\mc{B}_\al\subseteq \widehat{A}_\al$, $A_\al\in\mc{I}$, and that the sequence $(\mc{B}_\al)_{\al<\ka}$ is a $\subseteq$-increasing cover of $[\om]^\om$. Let $\mc{C}_\al=\{Y\in [\om]^\om:f[Y]\in\mc{B}_\al\}$. Then $(\mc{C}_\al)_{\al<\ka}$ is also an increasing cover of $[\om]^\om$, and $|f[Y]\cap A_\al|=\om$ hence $|Y\cap f^{-1}[A_\al]|=\om$ for every $\al$. In particular, $\mc{C}_\al\subseteq \widehat{f^{-1}[A_\al]}\in\widehat{\mc{J}}$, and so $[\om]^\om$ is $\mc{J}$-accessible.
\end{proof}

\begin{rem}\label{usefulrem}
Notice that the sketched proof of $\cov^*(\mc{I})\leq\non(\mc{M})$ for tall analytic P-ideals (at the beginning of Section \ref{analp}) shows a bit more, namely, if $\mc{I}$ is a tall analytic P-ideal, then $\vd_{\mbb{C}}[\om]^\om\cap V\in\widehat{\mc{I}}$. It is easy to see that the same holds for every analytic or coanalytic (we need definability) ideal $\mc{I}$ satisfying $\mc{ED}_\mrm{fin}\leq_\mrm{KB}\mc{I}$:
If $f$ witnesses this reduction, $P_n=\{(n,m):m\leq n\}$, and
$X=\{n:P_n\cap \ran(f)\ne\0\}$, then $(f^{-1}[P_n])_{n\in X}$ is a partition of $\om$ into finite sets. If $c\in\prod(f^{-1}[P_n]:n\in X)$ is a Cohen real over our model $V$ and $(n,k_n)=f(c(n))$ for every $n\in X$, then $\{(n,k_n):n\in X\}\in\mc{ED}_\mrm{fin}$ and hence  $\ran(c)\subseteq f^{-1}\{(n,k_n):n\in X\}\in\mc{I}$. Because of a trivial density argument, $|S\cap \ran(c)|=\om$ for every $S\in [\om]^\om\cap V$.

In particular, we have the following observations: Let $\ka$ be regular, $\mc{I}$ be as above, and assume that $\PP$ is the limit of a nontrivial $\ka$-stage finite support iteration of ccc forcing notions.
\begin{itemize}
\item[(1)] If $\ka>\om_1$, then $V^\PP\models$``$\mrm{cov}^*(\mc{I})\leq\ka\leq \mathrm{non}^*(\mc{I})$, $[\om]^\om$ is $\mc{I}$-accessible, and there are no $\mc{I}$-Luzin sets''.
\item[(1')] If $\ka>\om_1$ and $\mc{I}$ is destroyed (i.e., its tallness is destroyed) at cofinally many stages (e.g. by iterating $\mbb{M}(\mc{I}^*)$ or $\mbb{L}(\mc{I}^*)$), then additionally to (1) we have $V^\PP\models\mrm{cov}^*(\mc{I})=\ka= \mathrm{non}^*(\mc{I})$.
\item[(2)] If $\ka=\om_1$ then $V^\PP\models$``$\mrm{cov}^*(\mc{I})\leq\om_1\leq \mathrm{non}^*(\mc{I})$ and $[\om]^\om$ is $\mc{I}$-accessible''.
\item[(2')] If $\ka=\om_1$ and $\mc{I}$ is destroyed at cofinally many stages, then additionally to (2) we have $V^\PP\models$``$\mrm{cov}^*(\mc{I})=\om_1=\mathrm{non}^*(\mc{I})$  and there are $\mc{I}$-Luzin sets''.
\end{itemize}
\end{rem}

In the next diagram, we summarize all Kat\v{e}tov-Blass reductions (known at this moment) between our examples. We know that, beyond the three question marks, there are no other Kat\v{e}tov-Blass-reductions (in almost all cases, not even Kat\v{e}tov-reductions) between these ideals:

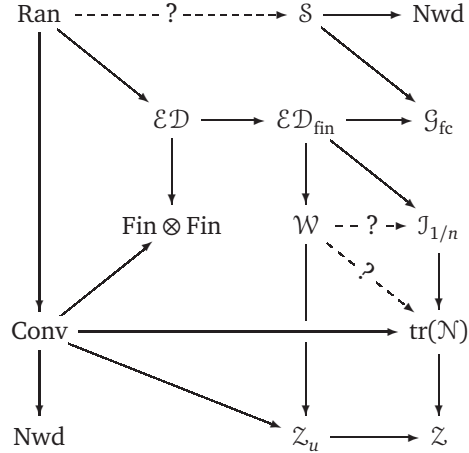
\begin{diagram}
\mrm{Ran} & & \rDashto~{\text{{\large ?}}} & & \mc{S} & \rTo & \mrm{Nwd} \\
          & \rdTo  & & & & \rdTo & \\
 & & \mc{ED} & \rTo & \mc{ED}_\mrm{fin} & \rTo & \mc{G}_\mrm{fc} \\
 \dTo    & & \dTo & & \dTo & \rdTo &  & \\
 &  & \mrm{Fin}\otimes\mrm{Fin} & & \mc{W} & \rDashto~{\text{{\large ?}}} & \mc{I}_{1/n}\\
 & \ruTo &  & & \vLine & \rdDashto~{\text{{\large ?}}} & \dTo \\
\mrm{Conv} & \rTo  & &  & \HonV &  & \mrm{tr}(\mc{N})\\
 \dTo & \rdTo(4,2) & & & \dTo &  & \dTo \\
\mrm{Nwd} & &  & & \mc{Z}_u & \rTo & \mc{Z}
\end{diagram}

M. Hru\v{s}\'ak, J. Brendle, and J. Fla\v{s}kov\'a already presented (see \cite{hrusakkatetov} and \cite{janajorg}) a slightly smaller diagram and proved that there are no other reductions between the ideals included (with the exception of a possible reduction of $\mrm{Ran}$ to $\mc{S}$). We added $\mc{W}$,  $\mc{G}_\mrm{fc}$, $\mc{Z}_u$, and $\mrm{tr}(\mc{N})$ to their diagram. We already know that $\mc{I}_{1/n}\leq_\mrm{KB}\mrm{tr}(\mc{N})\leq_\mrm{KB}\mc{Z}$. First of all, we show that $\mc{Z}\nleq_\mrm{K}\mrm{tr}(\mc{N})\nleq_\mrm{K}\mc{I}_{1/n}$ (and more):
\begin{fact}$ $
\begin{itemize}
\item[(a)] $\mrm{tr}(\mc{N})\nleq_{\mrm{K}}\mc{I}_{1/n}\upharpoonright Y$ for every $Y\in\mc{I}^+_{1/n}$. \item[(b)] $\mc{Z}\nleq_{\mrm{K}}\mrm{tr}(\mc{N})\upharpoonright Y$ for every $Y\in\mrm{tr}(\mc{N})^+$.
\end{itemize}
\end{fact}

\begin{proof}
(a): Let $f:Y\to 2^{<\om}$ be arbitrary. We show that $f$ is not a Kat\v{e}tov-reduction. We can assume that $f^{-1}[\{s\}]\in\mc{I}_{1/n}$ for every $s\in 2^{<\om}$ otherwise we are done. We know that $\sum\{1/(n+1):n\in Y\}=\infty$ hence there is an $n_0$ such that $\sum\{1/(n+1):f(n)\in 2^{<n_0}\}\geq 1$. We can also find an $s_0\in 2^{n_0}$ such that $\sum\{1/(n+1):f(n)\supseteq s_0\}=\infty$. We can continue this procedure, that is, we can construct a strictly increasing sequence $s_{-1}:=\0\subseteq s_0\subseteq s_1\dots $ in $2^{<\om}$ such that $\sum\{1/(n+1):s_k\subseteq f(n)$ and $|f(n)|<|s_{k+1}|\}\geq 1$ for every $k\geq -1$. Let
\[ A=\big\{t\in 2^{<\om}:\exists\;k\geq -1\;\big(s_k\subseteq t\;\text{and}\;|t|<|s_{k+1}|\big)\big\}.\]
Then $[A]=\{x\}\in\mc{N}$ where $x=s_0\cup s_1\cup\dots$ but $f^{-1}[A]\notin\mc{I}_{1/n}$.

(b): This is a direct consequence of the characterization of {\em forcing indestructibility} of ideals (see \cite[Thm. 1.6]{fq}) and the fact that $\mc{Z}$ is random-indestructible (see \cite{elekes} or \cite[Thm. 3.12]{cardinvanalp}).
\end{proof}

Let us present the rest of the proofs concerning the four new ideals ($\mc{W},\mc{G}_\mrm{fc},\mc{Z}_u$, and $\mrm{tr}(\mc{N})$) we extended the diagram with. Of course, to check that these reductions and non-reducibilities presented below do not leave any more questions beyond the marked ones, the reader has to use the original diagram (see \cite{hrusakkatetov} and \cite{janajorg})  and transitivity of $\leq_\mrm{K(B)}$.

\medskip
Position of $\mc{G}_\mrm{fc}$: We know that $\mc{S}\leq_\mrm{KB} \mc{G}_\mrm{fc}$ and $\mc{ED}_\mrm{fin}\leq_\mrm{KB}\mc{G}_\mrm{fc}$ (see \cite{Meza}). To find the exact position of $\mc{G}_\mrm{fc}$ in the diagram, it is enough to show that  $\mrm{Conv}\nleq_\mrm{K}\mc{G}_\mrm{fc}$, $\mc{I}_{1/n}\nleq_\mrm{K}\mc{G}_\mrm{fc}$, and $\mc{W}\nleq_\mrm{KB}\mc{G}_\mrm{fc}$. We will need the following:

\begin{thm}\label{aboveconv} {\em (see \cite[Thm. 2.4.3]{Meza})}. Let $\mc{I}$ be an ideal on $\om$. Then $\mrm{Conv}\leq_\mrm{K}\mc{I}$ iff there is a countable $\mc{X}\subseteq [\om]^\om$ such that $\forall$ $Y\in\mc{I}^+$ $\exists$ $X\in\mc{X}$ $|Y\cap X|=|Y\setminus X|=\om$.
\end{thm}

We say that an ideal $\mc{I}$ is a {\em P$^+$-ideal}, if every $\subseteq$-decreasing sequence $(Y_n)_{n\in\om}$ of $\mc{I}$-positive sets has an $\mc{I}$-positive pseudointersection. It is well-known that $F_\sigma$ ideals are P$^+$: Let $\mc{I}=\mrm{Fin}(\varphi)$ for some lsc submeasure $\varphi$.  Applying lower-semicontinuity of $\varphi$ (and that $\varphi(Y_n)=\infty$ for every $n$), we can pick finite sets $F_n\subseteq Y_n$ such that $\varphi(F_n)\geq n$ for every $n$. It follows that $\bigcup\{F_n:n\in\om\}$ is an $\mc{I}$-positive pseudointersection of $(Y_n)_{n\in\om}$.

\begin{cor}\label{FsigmaConv}
There is no P$^+$-ideal (in particular, $F_\sigma$ ideal) Kat\v{e}tov-above $\mrm{Conv}$.
\end{cor}
\begin{proof}
We will apply Theorem \ref{aboveconv}. Let $\mc{I}$ be an P$^+$-ideal and $\{X_n:n\in\om\}\subseteq [\om]^\om$ a countable family. We can fix an $\eps\in 2^\om$ such that $Y_n=X_0^{\eps(0)}\cap X_1^{\eps(1)}\cap\dots\cap X_n^{\eps(n)}\in\mc{I}^+$ for every $n$ (where of course, $X^0=X$ and $X^1=\om\setminus X$). Then $Y_0\supseteq Y_1\supseteq Y_2\supseteq\dots$ is a decreasing sequence in $\mc{I}^+$, hence it has a pseudointersection $Y\in\mc{I}^+$, and it follows that $|Y\cap X^{1-\eps(n)}|<\om$ for every $n$.
\end{proof}

In particular, $\mrm{Conv}\nleq_\mrm{K} \mc{G}_\mrm{fc}$ and $\mrm{Conv}\nleq_\mrm{K}\mc{W}$.

\begin{fact} $\mc{I}_{1/n}\nleq_\mrm{K}\mc{G}_\mrm{fc}$.
\end{fact}
\begin{proof}
Fix a function $f:[\om]^2\to\om$ and assume that $f^{-1}[\{n\}]\in\mc{G}_\mrm{fc}$ for every $n\in\om$. By recursion, we will define a sequence $(F_n)_{n\in\om}$ of finite subsets of $\om$ such that $A=\bigcup\{F_n:n\in\om\}\in\mc{I}_{1/n}$ but $f^{-1}[A]\in\mc{G}_\mrm{fc}^+$. Let $F_0=\0$, $F_1=\{1\}$, and assume that we are done with $F_n$. Let $k_n>\max(F_n)$ be large enough such that $\binom{n+1}{2}/(k_n+1)\leq 2^{-n-1}$. We know that $f^{-1}[k_n]=f^{-1}[\{k:k<k_n\}]\in\mc{G}_\mrm{fc}$, i.e. there is a finite partition $\om=\bigcup\{P^n_m:m< \ell_n\}$ ($\ell_n\in\om$) such that $[P^n_m]^2\cap f^{-1}[k_n]=\0$ for every $m<\ell_n$. Then there is an $m_n$ such that $|P^n_{m_n}|=\om$. Let $G\in [P^n_{m_n}]^{n+1}$ be arbitrary and define $F_{n+1}=f[[G]^2]$. Then $A=\bigcup\{F_n:n\in\om\}\in\mc{I}_{1/n}$ because
\[ \sum_{k\in A}\frac{1}{k+1}=\sum_{n\in\om}\sum_{k\in F_n}\frac{1}{k+1}\leq\frac{1}{2}+\sum_{n\geq 1}\frac{|F_{n+1}|}{k_n+1}\leq\frac{1}{2}+\sum_{n\geq 1}
\frac{\binom{n+1}{2}}{k_n+1}\leq\frac{1}{2}+\sum_{n\geq 1}2^{-n-1}=1,\]
and $f^{-1}[A]\in\mc{G}_\mrm{fc}^+$ because it contains (the set of edges of) arbitrary large finite complete graphs.
\end{proof}

The only case when we cannot show strong non-reducibility in our diagram is the following result (that is, we can prove only non-Kat\v{e}tov-Blass-reducibility unlike in all other cases when we have non-Kat\v{e}tov-reducibility).

\begin{thm} $\mc{W}\nleq_\mrm{KB}\mc{G}_\mrm{fc}$.
\end{thm}
\begin{proof}
Let $f:[\om]^2\to\om$ be finite-to-one. First of all, notice that we can fix an infinite $A\subseteq\om\setminus \{0\}$ such that letting $I_n=\{f(\{n',n\}):n'<n\}\in [\om]^{<\om}$, we have $2\cdot\max(I_m)<\min(I_n)$ for every $m<n$ from $A$. Now we will construct sequences $(A_k)_{k\in\om}$ and $(a_k)_{k\in\om}$ such that
\begin{itemize}
\item[(a)] $A=A_0\supseteq A_1\supseteq A_2\supseteq\dots$ are infinite;
\item[(b)] $a_k\in A_k$ and $a_k<\min(A_{k+1})$ for every $k$;
\item[(c)] for every $k$ and $n\in A_k$, the set $I_{n,k}=\{f(\{a_\ell,n\}):\ell<k\}$ does not contain arithmetic progressions of length $>2$ ($a_\ell< n$ for every $\ell<k$ because of (b), in particular $I_{n,k}\subseteq I_n$).
\end{itemize}
Suppose we have such sequences. Then let $B=\{a_k:k\in\om\}\in [A]^\om$ and $E=[B]^2$. Clearly, $E\notin\mc{G}_\mrm{fc}$. On the other hand, we claim that $X=f[E]\in\mc{W}$.

First of all, notice that $X=\bigcup\{I_{a_k,k}:k\in \om\}$. If $P\subseteq X$ is an arithmetic progression, then because of the ``big gaps'' between elements of $(I_n)_{n\in A}$ and hence of $(I_{a_k,k})_{k\in\om}$, $P$ can meet at most two $I_{a_k,k}$'s. Moreover, if $\ell<k$ and $P$ meets both $I_{a_\ell,\ell}$ and $I_{a_k,k}$, then $|P\cap I_{a_\ell,\ell}|=1$ (because of the long gaps again). At the same time, $|P\cap I_{a_k,k}|\leq 2$ for every $k$ because $I_{a_k,k}$ does not contain arithmetic progressions of length $>2$ (of course here we used that $\max(I_{a_\ell,\ell})<\min(I_{a_k,k})$ whenever $\ell<k$ and hence $P\cap I_{a_k,k}$ is an interval in $P$). Therefore $|P|\leq 3$.

\smallskip
Construction of the sequences $(A_k)_{k\in\om}$ and $(a_k)_{k\in\om}$: $A_0=
A$ is given. Also $a_0,A_1,a_1$ and $A_2$ can be constructed for trivial reasons. Assume that we already constructed $a_{k-1}$ and $A_k$. If we cannot find appropriate $a_k$ and $A_{k+1}$, then
\[ \forall\;m\in A_k\;\forall^\infty\;n\in A_k\setminus(m+1)\;\,\mrm{ap}\big(\big\{f(\{a,n\}):a\in \{a_0,a_1,\dots,a_{k-1},m\}\big\}\big)\geq 3\]
where $\mrm{ap}(\cdot)$ stands for the same operation as in Example \ref{cotW}. Then we can find an infinite $A'\subseteq A_k$ such that $\mrm{ap}(\{f(\{a,n\}):a\in \{a_0,a_1,\dots,a_{k-1},m\}\})\geq 3$ for every $m<n$ from $A'$.

Clearly, if $a<b$ are natural numbers, then there are at most three $c\in\om\setminus\{a,b\}$ such that $\{a,b,c\}$ is an arithmetic progression (namely, $c=2a-b,2b-a,(b-a)/2$). Hence for every $n\in A'$ we can fix a $J_n\subseteq I_n$ such that (i) $|J_n|\leq 3\binom{k}{2}$ and (ii) $f(\{m,n\})\in J_n$ for every $m<n$ from $A'$. Let $E'=[A']^2\notin\mc{G}_\mrm{fc}$. Then $X'=f[E']\in\mc{W}$ (a contradiction, and we are done). Why? We show that $\mrm{ap}(X')\leq 3\binom{k}{2}+1$. Notice that $X'\subseteq\bigcup\{J_n:n\in A'\}$, and just like above (for $X$) an arithmetic progression $P\subseteq X'$
can meet at most two $J_n$'s and if it meets both $J_m$ and $J_n$ where $m<n$ then $|P\cap J_m|=1$ and of course $|P\cap J_n|\leq 3\binom{k}{2}$ for every $n$.
\end{proof}

Position of $\mc{W}$:
It is easy to see that $\mc{ED}_\mrm{fin}\leq_\mrm{KB}\mc{W}$ (i.e. that $\mc{W}$ is not a weak Q-ideal). We also know that (i) $\mc{W}\nleq_\mrm{KB}\mc{ED}_\mrm{fin}$ because $\mc{W}\nleq_\mrm{KB}\mc{G}_\mrm{fc}$, (ii) $\mc{W}\leq_\mrm{KB}\mc{Z}_u$  because $\mc{W}\subseteq\mc{Z}_u$, and (iii) $\mrm{Conv}\nleq_\mrm{K}\mc{W}$ because $\mc{W}$ is $F_\sigma$. To show that there are no other reductions from and to $\mc{W}$, the last non-reducibility we have to check is $\mc{I}_{1/n}\nleq_\mrm{K}\mc{W}$ and this is easy.

\medskip
Position of $\mrm{tr}(\mc{N})$: To show that there are no additional reductions from or to $\mrm{tr}(\mc{N})$, we  prove that $\mrm{Conv}\leq_\mrm{KB}\mrm{tr}(\mc{N})$, $\mrm{tr}(\mc{N})\nleq_\mrm{K}\mc{Z}_u$, and  $\mc{Z}_u\nleq_\mrm{K}\mrm{tr}(\mc{N})$. The last one follows from \cite[Thm. 2.10, Thm. 4.2]{covprop} and  \cite[Thm. 1.6]{fq}.

\begin{fact}\label{convtrn}  $\mrm{Conv}\leq_\mrm{KB}\mrm{tr}(\mc{N})$.
\end{fact}
\begin{proof}
Let $r_\0=2^{-1}$ and for every $t\in 2^{<\om}\setminus\{\0\}$ define $r_t=2^{-1}+\sum\{(-1)^{t(k)+1}2^{-k-2}:k<|t|\}$.
For example, if $t$ ranges over $2^{\leq 2}$ then $r_t$ ranges over $\{1/8,2/8,\dots,7/8\}$.
The map $t\mapsto r_t$ is one-to-one, $X=\{r_t:t\in 2^{<\om}\}\subseteq \mbb{Q}\cap [0,1]$, and $r_t< r_s$ iff $t< s$ in the ``lexicographic'' order on $2^{<\om}$, that is,
\[ \big(t\subseteq s\;\text{and}\;s(|t|)=1\big)\;\text{or}\;\big(s\subseteq t\;\text{and}\;t(|s|)=0\big)\;\text{or}\] \[\big(\big(\exists\;n\;t(n)\ne s(n)\big)\;\text{and}\; s(\min\{n:t(n)\ne s(n)\})=1\big).\]
Clearly, $X\in\mrm{Conv}^+$, therefore $\mrm{Conv}\leq_\mrm{KB}\mrm{Conv}\upharpoonright X$ (actually, these ideals are isomorphic up to a bijection between the underlying sets). We show that $\mrm{Conv}\upharpoonright X\leq_\mrm{KB}\mrm{tr}(\mc{N})$.

Notice that the bijection $t\mapsto r_t$ allows us to define $\mrm{Conv}\upharpoonright X$ on $2^{<\om}$: An infinite set $A\subseteq 2^{<\om}$ is a generator of $\mrm{Conv}_\mrm{tr}$ iff $\exists$ $x\in 2^\om$ $\forall$ $n\in\om$ $\forall^\infty$ $t\in A$ $(x\upharpoonright n\subseteq t)$. It is easy to see that $A$ is a generator of $\mrm{Conv}_\mrm{tr}$ iff the sequence $(r_t)_{t\in A}$ is convergent in $[0,1]$, therefore $\mrm{Conv}_\mrm{tr}=\{B\subseteq 2^{<\om}:\{r_t:t\in B\}\in\mrm{Conv}\}$. To finish the proof, notice that $\mrm{Conv}_\mrm{tr}\subseteq\mrm{tr}(\mc{N})$ because if $B\in\mrm{Conv}_\mrm{tr}$ then $[B]$ is finite. (Remark: $\mrm{tr}([2^\om]^{<\om})=\{B\subseteq 2^{<\om}:[B]$ is finite$\}\ne\mrm{Conv}_\mrm{tr}$.)
\end{proof}

\begin{fact}
$\mc{I}_{1/n}\nleq_\mrm{K}\mc{Z}_u$ (in particular, $\mrm{tr}(\mc{N})\nleq_\mrm{K}\mc{Z}_u$).
\end{fact}
\begin{proof}
Let $f:\om\to \om$ and assume that $f^{-1}[\{n\}]\in\mc{Z}_u$ for every $n$. We will define a sequence $(E_n)_{n\in\om}$ of successive intervals on $\om$ such that $|E_n|=n$ hence $E=\bigcup\{E_n:n\in\om\}\notin\mc{Z}_u$ but $f[E]\in\mc{I}_{1/n}$. Let $E_0=\0$, $E_1=\{0\}$, and assume that we already defined $E_n$ such that $f(k)\geq 2^n$ for every $k\in E_n$ (if $n>1$).

Let $m=\max\{\max\{f(k):k\in E_n\},2^{n+1}\}+1$. Then $A=f^{-1}[m]\in\mc{Z}_u$. In particular, if $M>2(n+1)$ is large enough such that $S_{M}(A)<M/(2(n+1))$, and $E'_{n+1}$ is an interval of length $M$  after $E_n$, then it must contain a subinterval $E_{n+1}$ of size $n+1$ which is disjoint from $A$. Then $E=\bigcup\{E_n:n\in\om\}\notin\mc{Z}_u$, and of course $f[E]\in\mc{I}_{1/n}$ because $\sum\{1/(\ell+1):\ell\in f[E_n]\}\leq n/2^n$.
\end{proof}

Position of $\mc{Z}_u$: To find the position of $\mc{Z}_u$, we have to show that $\mrm{Conv}\leq_\mrm{KB}\mc{Z}_u$.

\begin{fact}
$\mrm{Conv}\leq_\mrm{KB}\mc{Z}_u$.
\end{fact}
\begin{proof}
We will show that $\mrm{Conv}_\mrm{tr}\leq_\mrm{KB} \mc{Z}_u$. Let $f:\om\to 2^{<\om}$ be the binary expansion function starting with the least significant bit, that is, $f(0)=\0$ and if $n>0$ then $f(n)$ is the unique elements of $2^{<\om}$ satisfying $f(n)(|f(n)|-1)=1$ and $n=\sum\{ 2^k \cdot f(n)(k):k<|f(n)|\}$. $f$ is clearly one-to-one.

Now, if $S\subseteq 2^{<\om}$ converges to $x\in 2^\om$, then for each $k\in\om$ all but finitely many members of $S$ extend $x\upharpoonright k$. However, if $n\ne m$ and both $f(n)$ and $f(m)$ extend $x\upharpoonright k$ then $|n-m|\geq 2^k$.  It easily follows that $f^{-1}[S]$ has uniform density $0$.
\end{proof}

\section{Some ideals with $\mrm{non}^*=\om<\mrm{cov}^*$}\label{countnon}

We know that (see e.g. \cite{Meza}, and \cite{covnwd} for the results on $\mrm{Nwd}$)
\begin{itemize}
\item $\mrm{non}^*(\mc{ED})=\om$, $\mrm{cov}^*(\mc{ED})=\mrm{non}(\mc{M})$, and $\mrm{cof}^*(\mc{ED})=\mf{c}$;
\item $\mrm{non}^*(\mrm{Ran})=\om$, $\mrm{cov}^*(\mrm{Ran})=\mrm{cof}^*(\mrm{Ran})=\mf{c}$;
\item $\mrm{non}^*(\mc{S})=\om$, $\mrm{cov}^*(\mc{S})=\mrm{non}(\mc{N})$, and $\mrm{cof}^*(\mc{S})=\mf{c}$;
\item $\mrm{non}^*(\mrm{Nwd})=\om$, $\mrm{cov}^*(\mrm{Nwd})=\mrm{cov}(\mc{M})$, and $\mrm{cof}^*(\mrm{Nwd})=\mrm{cof}(\mc{M})$;
\item $\mrm{non}^*(\mrm{Conv})=\om$, $\mrm{cov}^*(\mrm{Conv})=\mrm{cof}^*(\mrm{Conv})=\mf{c}$;
\item $\mrm{non}^*(\mrm{Fin}\otimes\mrm{Fin})=\om$, $\mrm{cov}^*(\mrm{Fin}\otimes\mrm{Fin})=\mf{b}$, and $\mrm{cof}^*(\mrm{Fin}\otimes\mrm{Fin})=\mf{d}$.
\end{itemize}

In the case of these ideals (the upper left triangle in the diagram of KB-reducibility), the diagram summarizing the connections between cardinal invariants, existence of Luzin type families, and existence of cotowers in the ideal, becomes much simpler because $\mrm{non}^*(\mc{I})=\om<\om_1\leq\mrm{cov}^*(\mc{I})$ if $\mc{I}=\mc{ED},\mrm{Ran},\mc{S},\mrm{Nwd},\mrm{Conv},\mrm{Fin}\otimes\mrm{Fin}$.
The remaining related questions are the following: Let $\mc{I}$ be one these ideals.
\begin{itemize}
\item[(Q1)] Does there exist an $\mc{I}$-Luzin set in $\mrm{ZFC}$? If no, is it possible that $\mrm{cov}^*(\mc{I})=\mf{c}$ and there are no $\mc{I}$-Luzin sets?
\item[(Q2)] Can we force large $\mc{I}$-Luzin sets?
\item[(Q3)] Is it consistent with $\mrm{ZFC}$ that there is an $\mc{I}$-Luzin set and $\mrm{cov}^*(\mc{I})<\mrm{cof}^*(\mc{I})=\om_2$?
\item[(Q4)] Is it consistent with $\mrm{ZFC}$ that there is an $\mc{I}$-Luzin set but there are no $\mc{I}$-Luzin sets of size $\mrm{cov}^*(\mc{I})=\om_2=\mf{c}$?
\end{itemize}

The next two examples answer all these questions in the case of $\mc{I}=\mrm{Ran}$ and $\mc{I}=\mrm{Conv}$.

\begin{exa}
There is a $\mrm{Conv}$-Luzin set (and hence by Fact \ref{KBcons} (b), also a $\mrm{Ran}$-Luzin set) of size $\mf{c}$. For every $x\in [0,1]\setminus \mbb{Q}$ fix a sequence $(r^x_n)_{n\in\om}$ of rational numbers tending to $x$. Then the family $\{\{r^x_n:n\in\om\}:x\in [0,1]\setminus\mbb{Q}\}$ is a $\mrm{Conv}$-Luzin set.
\end{exa}

For the rest, that is, for $\mc{I}=\mc{ED},\mc{S},\mrm{Nwd},\mrm{Fin}\otimes\mrm{Fin}$, it is consistent that there are no $\mc{I}$-Luzin sets (even together with $\mrm{cov}^*(\mc{I})=\mf{c}$) because of the following  easy results.

\begin{Prop}\label{nosluzin}
Assume $\mrm{MA}_{\om_1}$ (or just $\mrm{cov}(\mc{N})>\om_1$). Then there are no $\mc{S}$-Luzin sets.
\end{Prop}
\begin{proof}
Recall that $\Omega=\{A\subseteq 2^\om:A$ is clopen and $\lam(A)=1/2\}$ and the ideal $\mc{S}$ on $\Omega$ is generated by $\{I_x:x\in 2^\om\}$ where $I_x=\{A\in\Omega:x\in A\}$.

If $X=\{C_n:n\in \om\}\in [\Omega]^\om$ then we can consider $\widetilde{X}=\limsup X=\bigcap_{n\in\om}\bigcup_{k\geq n}C_k$, clearly $\lam(\widetilde{X})\geq 1/2$. Now assume that $\mc{X}=\{X_\al:\al<\om_1\}\subseteq [\Omega]^\om$, and for every $\al$ fix a compact set $C_\al\subseteq \widetilde{X}_\al$ of positive measure. Then $\{C_\al:\al<\om_1\}\subseteq\PP=\{B\in\mrm{Borel}(2^\om):\lam(B)>0\}$. $\PP$ equipped with $\subseteq$ is a ccc forcing notion (it is the random forcing). Applying $\mrm{MA}_{\om_1}$, we know that there is a $Z\in [\om_1]^{\om_1}$ such that $\mc{C}=\{C_\al:\al\in Z\}$ is centered in $\PP$, in particular, it is centered in the usual sense as well, and because of compactness, there is an $x\in\bigcap\mc{C}$. Clearly, $|I_x\cap X_\al|=\om$ for every $\al\in Z$.
\end{proof}

Applying Fact \ref{KBcons} (b) and that $\mc{S}\leq_\mrm{KB}\mrm{Nwd}$, it is also consistent that there are no $\mrm{Nwd}$-Luzin families.

\begin{Prop}\label{noedluzin}
Assume $\mrm{MA}_{\om_1}$ (or just $\mrm{cov}(\mc{M})>\om_1$). Then there are no $\mc{ED}$-Luzin sets.
\end{Prop}
\begin{proof}
Let $\mc{X}=\{X_\al:\al<\om_1\}\subseteq [\om\times\om]^\om$. By shrinking the family and its elements as well, we can assume that every $X_\al$ is an infinite partial function $\om\to\om$ (if it is impossible, then $\mc{X}$ cannot be an $\mc{ED}$-Luzin set anyway). Let $\PP=(\om^{<\om},\supseteq)$ be the Cohen-forcing. It is clear that with $\om_1$ many dense sets we can ensure that the generic function has infinite intersection with all $X_\al$.
\end{proof}

Applying Fact \ref{KBcons} (b) and that $\mc{ED}\leq_\mrm{KB}\mrm{Fin}\otimes\mrm{Fin}$, this last result also shows that consistently there are no $\mrm{Fin}\otimes\mrm{Fin}$-Luzin families. The next result says a bit more about $\mrm{Fin}\otimes\mrm{Fin}$-Luzin families.

\begin{Prop}
There is a $\mathrm{Fin}\otimes\mathrm{Fin}$-Luzin set if and only if $\mf{d}=\om_1$.
\end{Prop}
\begin{proof}
Clearly the columns of $\om\times\om$ and sets of the form
$A_f=\{(n,m):m\le f(n)\}$ (where $f\in\om^\om$) generate $\mathrm{Fin}\otimes\mathrm{Fin}$.

If $\mf{d}=\om_1$ then there is a scale $(f_\al)_{\al<\om_1}$, that is, a $\le^*$-well ordered dominating family in $\om^\om$, and $\{f_\al:\al<\om_1\}$ forms a $\mathrm{Fin}\otimes\mathrm{Fin}$-Luzin set because $|A_f\cap f_\al|<\om$ for all but countably many $\al$.

Conversely, assume that $\mc{X}\subseteq [\om\times\om]^\om$ is a $\mathrm{Fin}\otimes\mathrm{Fin}$-Luzin set.
By shrinking the family and its elements as well, we can assume that all elements of $\mc{X}$ are infinite partial functions. Next, we can assume that all elements of $\mc{X}$ are strictly increasing because if $\om_1$ many do not contain strictly increasing partial functions, then they are bounded and $\om_1$ many have the same bound therefore $\mc{X}$ cannot be $\mrm{Fin}\otimes\mrm{Fin}$-Luzin. We can extend these functions to total functions by $\tilde{g}(k)=g(\min\{\mathrm{dom}(g)\setminus k\})$ for $g\in\mc{X}$, and let $\mc{D}=\{\tilde{g}:g\in \mc{X}\}$. We claim that $\mc{D}$ is dominating and $|\mc{D}|=\om_1$. If $f\in\om^\om$, then the set $\mc{D}_f=\{\tilde{g}\in\mc{D}:|\tilde{g}\cap A_f|=\om\}$ is countable, and for all $\tilde{h}\in\mc{D}\setminus\mc{D}_f$ we have $f\le^* \tilde{h}$. It follows that all uncountable subfamilies of $\mc{D}$ are dominating, and therefore $\mf{d}=\om_1$. In particular, there is a scale $(f_\al)_{\al<\om_1}$ and $\mc{D}$ is the union of the countable sets $\{\tilde{g}\in\mc{D}:\tilde{g}\le^* f_\al\}\subseteq \mc{D}_{f_\al}$ hence it has cardinality $\om_1$.
\end{proof}

What can we say about (Q2)? We have to deal with $\mc{ED}$, $\mc{S}$, and $\mrm{Nwd}$.

If $I$ is an ideal on $2^\om$, then one can define $\mrm{tr}(I)$ as $\mrm{tr}(\mc{N})$ was defined. It is trivial to see that if there is an $I$-Luzin set $\{x_\al:\al<\ka\}\subseteq 2^\om$ of size $\ka$, then there is a $\mrm{tr}(I)$-Luzin set of size $\ka$: Simply $\{\{x_\al\upharpoonright n:n\in\om\}:\al<\ka\}$ is a $\mrm{tr}(I)$-Luzin set. In particular, if there is a Sierpi\'nski/Luzin set of size $\ka$, then there is a $\mrm{tr}(\mc{N})/\mrm{Nwd}$-Luzin set of size $\ka$. \label{largetrluzinlabel} Here, of course, we used the easy fact that $\mrm{tr}(\mc{M})$ is the tree version $\mrm{Nwd}_\mrm{tr}$ of $\mrm{Nwd}$ defined in the same manner as $\mrm{Conv}_\mrm{tr}$ in Fact \ref{convtrn}.  (We do not need it but it is easy to see that $\mrm{Nwd}$ is isomorphic to $\mrm{Nwd}_\mrm{tr}$.)

Consequently, it is consistent that there are $\mrm{Nwd}$-Luzin (hence $\mc{S}$-Luzin) sets of size $\mf{c}>\om_1$ (e.g. in the Cohen model), and that there are $\mrm{tr}(\mc{N})$-Luzin (hence $\mc{ED}$-Luzin) sets of size $\mf{c}>\om_1$ (e.g. in the random model).

\medskip
In the rest of this section we focus on (Q3) and (Q4). We have to deal with $\mc{ED}$, $\mc{S}$, and $\mrm{Nwd}$.

\begin{model}\label{q3many} ((Q3) for $\mc{ED}$, $\mc{S}$, and $\mrm{Nwd}$.)
It is consistent that there are  $\mc{I}$-Luzin sets and $\mrm{cov}^*(\mc{I})=\om_1<\mrm{cof}^*(\mc{I})=\mf{c}=\om_2$ for $\mc{I}=\mc{ED}$, $\mc{S}$, and $\mrm{Nwd}$ (in the same model).

Start with a model $V$ of $\mf{b}=\mf{c}=\om_2$ and consider the $\om_1$ stage finite support iteration $(\PP_\al)_{\al\leq\om_1}$ based on $(\QQ_\al)_{\al<\om_1}$ such that $\QQ_\al$ takes the values $\mbb{B}$ (the random forcing), $\mbb{M}(\mc{ED}^*)$, and $\mbb{M}(\mc{S}^*)$ alternately.

$\PP_{\om_1}$ adds $\mc{I}$-Luzin sets for all three ideals: For $\mc{I}=\mc{ED}$ and $\mc{I}=\mc{S}$, the sequence of $\mbb{M}(\mc{I}^*)$-generic sets forms an $\mc{I}$-Luzin set. The Cohen reals appearing at limit stages form a (classical) Luzin set and hence we also obtain $\mrm{Nwd}$-Luzin sets.

$V^{\PP_{\om_1}}\models\mrm{cov}^*(\mc{ED})=\mrm{non}(\mc{M})=\om_1$ because of the Cohen reals added at limit stages. $V^{\PP_{\om_1}}\models\mrm{cov}^*(\mc{S})=\mrm{non}(\mc{N})=\om_1$ because of the random reals added at cofinally many stages, in particular $V^{\PP_{\om_1}}\models\mrm{cov}^*(\mrm{Nwd})=\mrm{cov}(\mc{M})=\om_1$ because $\mrm{cov}(\mc{M})\leq\mrm{non}(\mc{N})$.

$\mrm{cof}^*(\mc{ED})=\mrm{cof}^*(\mc{S})=\mf{c}=\om_2$ in $\mrm{ZFC}$. To show that $V^{\PP_{\om_1}}\models\mrm{cof}^*(\mrm{Nwd})=\mrm{cof}(\mc{M})=\om_2$, we know that $V\models \mf{b}=\mf{d}=\om_2$ hence there is a scale $(f_\al)_{\al<\om_2}$ in $V$ (that is, a $\leq^*$-well-ordered dominating sequence in $\om^\om$), and it is enough to show that this sequence remains unbounded in $V^{\PP_{\om_1}}$ because then $V^{\PP_{\om_1}}\models\mf{d}=\om_2\leq\mrm{cof}(\mc{M})$.
$\mbb{B}$ is $\om^\om$-bounding hence it preserves unbounded families, and we know (see \cite{jorgmob}, or \cite{HuMi} for a more general characterisation) that $\mbb{M}(\mc{I}^*)$ preserves well-ordered unbounded families if $\mc{I}$ is $F_\sigma$. At limit stages we can apply \cite[Lem. 6.5.7]{BaJu}.
\end{model}

\begin{model}\label{modelq4ed} ((Q4) for $\mc{ED}$.)
It consistent that there are  $\mc{ED}$-Luzin sets but there are no $\mc{ED}$-Luzin sets of size $\mrm{cov}^*(\mc{ED})=\om_2=\mf{c}$.

Let $V\models\mrm{cov}(\mc{N})=\mrm{cov}(\mc{M})=\om_2=\mf{c}$ and add $\om_1$ many random reals simultaneously, this forcing notion will be denoted by $\mbb{B}(\om_1)$. It adds a Sierpi\'nski set of size $\om_1$ hence a $\mrm{tr}(\mc{N})$-Luzin set too, and applying Fact \ref{KBcons} (b), an $\mc{ED}$-Luzin set as well.

We know that $\mbb{B}$ preserves $\mrm{cov}(\mc{M})=\om_2$ (see \cite[Thm. 3.2.43]{BaJu}) hence by Proposition \ref{noedluzin}, there are no large $\mc{ED}$-Luzin sets in the final extension.

To show that $V^{\mbb{B}(\om_1)}\models\mrm{cov}^*(\mc{ED})=\mrm{non}(\mc{M})=\om_2$, we prove that $\mbb{B}(\om_1)$ preserves $\mrm{cov}(\mc{N})=\om_2$. Then we are done because $\mrm{cov}(\mc{N})\leq\mrm{non}(\mc{M})$.

This is a standard Fubini-type argument: Let $(\dot{A}_\al)_{\al<\om_1}$ be a sequence of $\mbb{B}(\om_1)$-names for Borel null subsets of $2^\om$. We show that there is a $y\in 2^\om$ such that $\vd_{\mbb{B}(\om_1)} y\notin\bigcup\{\dot{A}_\al:\al<\om_1\}$. $\dot{A}_\al$ is coded by a countable set hence we can assume that it is a (nice) $\mbb{B}(\be_\al)\simeq\mbb{B}$-name for some $\be_\al<\om_1$. Applying \cite[Lem. 3.1.6]{BaJu}, there are Borel sets $B_\al\subseteq 2^\om\times 2^\om$ such that $\vd_{\mbb{B}(\be_\al)}\dot{A}_\al=(B_\al)_{\dot{r}_{\be_\al}}:=\{y:(\dot{r}_{\be_\al},y)\in B_\al\}$ where $\dot{r}_{\be_\al}$ stands for the canonical $\mbb{B}(\be_\al)$-name of the generic (random) real (here of course, we identify $\mbb{B}(\be_\al)$ and $\mbb{B}$ along a bijection between $\be_\al$ and $\om$ fixed in $V$). This implies that co-null many vertical sections of $B_\al$ are null, and hence co-null many horizontal sections of it are also null, let $C_\al=\{y\in 2^\om:(B_\al)^y:=\{x:(x,y)\in B_\al\}$ is not a null set$\}\in\mc{N}$. Applying that $\mrm{cov}(\mc{N})=\om_2$ in $V$, we know that there is a $y\in 2^\om\setminus \bigcup\{C_\al:\al<\om_1\}$. Now if $\al<\om_1$ then $(B_\al)^y$ is a null set hence $\vd_{\mbb{B}(\be_\al)}\dot{r}_{\be_\al}\notin (B_\al)^y$ i.e. $\vd_{\mbb{B}(\be_\al)}y\notin (B_\al)_{\dot{r}_{\be_\al}}=\dot{A}_\al$, and hence $\vd_{\mbb{B}(\om_1)} y\notin \bigcup\{\dot{A}_\al:\al<\om_1\}$ because $\mbb{B}(\be_\al)$ is a complete subforcing of $\mbb{B}(\om_1)$.

Remark: One may notice that starting with $V\models\mrm{add}(\mc{M})=\om_2=\mf{c}$ would simplify the proof because then $\mf{b}=\mrm{cov}(\mc{M})=\om_2$ in $V$ and the random forcing preserves $\mf{b}=\om_2\leq\mrm{non}(\mc{M})$. However, this particular model described above will be useful later as well.
\end{model}

(Q4) for $\mc{S}$ is open (in the last section we will recall all remaining open questions including this one as well).

In the last model of this section, we will need the following interesting fact (notice that if $\mrm{non}(\mc{M})>\om_1$ then there are no Luzin sets).

\begin{thm}\label{bnwdluzin}$ $
\begin{itemize}
\item[(a)] If $\mf{b}\geq\om_2$ then there are no $\mrm{Nwd}$-Luzin sets.
\item[(b)] If there is an $\mrm{Nwd}$-Luzin set of size $\kappa > \omega_1$, then there is a Luzin set of size $\kappa$.
\end{itemize}
\end{thm}
\begin{proof}
(a): Instead of $\mrm{Nwd}$,  we will work with $\mrm{Nwd}_\mrm{tr}=\{A\subseteq 2^{<\om}:\forall$ $s\in 2^{<\om}$ $\exists$ $t\supseteq s$ $A\cap t^{\uparrow}=\0\}$ where $t^{\uparrow}=\{t'\in 2^{<\om}:t'\supseteq t\}$.  Let $\mc{X} \subseteq [2^{<\om}]^\omega$, $|\mc{X}|=\om_1$. Since Luzinity is preserved
if we replace members of $\mc{X}$ by their (infinite) subsets, applying K\"onig's Lemma and compactness of $2^\om$, we may assume that every $X \in \mc{X}$ is
\begin{itemize}
\item[(1)] either a subset of a branch $b_X \subseteq 2^{<\om}$ with ``limit'' $y_X= \bigcup b_X \in 2^\om$,
\item[(2)] or an antichain ``converging'' to a $y_X \in 2^\om$, i.e. $\forall$ $n$ $\forall^\infty$ $s \in X$ $(y_X \upharpoonright n \subseteq s)$.
\end{itemize}
Since $\mrm{non}(\mc{M}) \geq\mf{b}\geq \omega_2$, we can assume that there is a nowhere dense closed set $C\subseteq 2^\om$ such that
$\{y_X:X\in\mc{X}\}\subseteq C$. We know that $C=[T]$ for a tree $T\subseteq 2^{<\om}$, in particular $T \in \mrm{Nwd}_\mrm{tr}$. If we are in case (1), then $X \subseteq b_X\subseteq T$ and hence
$|X \cap T| = \omega$ is immediate. In case (2), when
$X$ is an antichain converging to $y_X$, we define $f_X \in \om^\om$ as follows:
\[ f_X (n) = \min\big\{ |s| : s \in X\;\text{and}\; y_X \upharpoonright n \subseteq s \big\}. \]
By $\mf{b}\geq \omega_2$, there is a strictly increasing $g \in\om^\om$ such that $f_X \leq^* g$ for all such $X \in \mc{X}$.
Now define $S \subseteq 2^{<\om}$ as follows:
\[ S = \big\{ s\in 2^{<\om}: |s|\leq \max\{g(|t|):t\in T,t\subseteq s\}\big\}. \]
Clearly, $T\subseteq S$ and $S$ is also a tree. But also $[S] = [T]$ because if $x\in 2^\om\setminus [T]$, $x\upharpoonright n\in T$, and $x\upharpoonright (n+1)\notin T$, then $x\upharpoonright (g(n)+1)\notin S$. Therefore $S \in\mrm{Nwd}_\mrm{tr}$ as well. Now if $X \in \mc{X}$ is an antichain, and $f_X(n)\leq g(n)$ for every $n \geq N$, then for every such $n$ the witness $s_n\in X$ in the definition $f_X(n)$ belongs to $S$ as well. In particular, $|X \cap S| = \omega$ for every antichain $X\in\mc{X}$ and hence
$|X \cap S| =\omega$ for every $X \in \mc{X}$, therefore $\mc{X}$ is not $\mrm{Nwd}_\mrm{tr}$-Luzin.

\smallskip
(b): The proof is similar to the proof of (a). Note that the assumption implies $\mrm{cov}(\mc{M})=\cov^* (\mrm{Nwd}) \geq\ka > \omega_1$.  Let $\mc{X}\subseteq [2^{<\om}]^\om$ be an $\mrm{Nwd}_\mrm{tr}$-Luzin family of size $\ka$. As above, we can assume that every $X\in\mc{X}$ satisfies (1) or (2) (without decreasing $|\mc{X}|$). Define $\mc{X}_1=\{X\in\mc{X}:X$ satisfies $(1)\}$, $\mc{X}_2=\mc{X}\setminus\mc{X}_1$, and similarly let $Y_1 = \{ y_X : X\in \mc{X}_1 \}$ and $Y_2=\{y_X:X\in\mc{X}_2\}$ (notice that $Y_1$ and $Y_2$ are not necessarily disjoint). We distinguish two cases:

Case I: $|\mc{X}_1|=\ka$. If $y\in Y_1$ then the set $\{X\in\mc{X}_1:y_X=y\}$ is countable because otherwise $\{y\upharpoonright n:n\in\om\}\in\mrm{Nwd}_\mrm{tr}$ witnesses that $\mc{X}_1$ is not $\mrm{Nwd}_\mrm{tr}$-Luzin. In particular, $|Y_1|=\ka$. We show that $Y_1$ is Luzin. Suppose not and let $A\subseteq 2^{\om}$ be nowhere dense closed set such that $|Y_1 \cap A | \geq \omega_1$. Then $A=[T]$ for some nowhere dense tree $T\subseteq 2^{<\om}$. Then $T\in\mrm{Nwd}_\mrm{tr}$ and $|T\cap X|=\om$ for every $X\in \mc{X}_1$ with $y_X\in A$ (that is, for uncountable many $X\in\mc{X}_1$), contradiction.

Case II: If $|\mc{X}_1|<\ka$ then $|\mc{X}_2|=\ka$. First we show that $|Y_2|=\ka$. Suppose not, then there is a $y\in Y_2$ and an $\mc{X}'\in [\mc{X}_2]^{\om_1}$ such that $y=y_X$ for every $X\in\mc{X}'$. For $X\in \mc{X}'$ fix a function $g_X:b=\{y\upharpoonright n:n\in\om\}\to X$ such that $s\subseteq g_X(s)$ for every $s\in b$. Then the set
\[ \big\{f\in \big(2^{<\om}\big)^b:\forall^\infty\;s\in b\;\big(f(s)\ne g_X(s)\big)\big\}\]
is meager in $(2^{<\om})^b$ and hence (applying that $\mrm{cov}(\mc{M})>\om_1$) there is an $f:b\to 2^{<\om}$ such that  $\forall$ $X\in\mc{X}'$ $\exists^\infty$ $s$ $(f(s)=g_X(s))$. We can assume that $f(s)\supseteq s$ for every $s$. Then $\mrm{ran}(f)\in\mrm{Nwd}_\mrm{tr}$ and $|X\cap \mrm{ran}(f)|=\om$ for every $X\in\mc{X}'$, contradiction.

Finally, we show that $Y_2$ is Luzin.  Suppose not and let $T\in\mrm{Nwd}_\mrm{tr}$ be a nowhere dense tree such that $|[T] \cap Y_2 | \geq \omega_1$. Fix an $\mc{X}''\subseteq \mc{X}_2$ of size $\om_1$ such that $y_X\in [T]$ for every $X\in\mc{X}''$. For $X\in \mc{X}''$ define $g_X$ as above. Then for every $X\in\mc{X}''$ the set
\[ \big\{f\in \big(2^{<\om}\big)^T :\forall^\infty\;s\in\dom(g_X)\;\big(f(s)\ne g_X(s)\big)\big\}\]
is meager in $(2^{<\om})^T$, and hence there is an $f:T\to 2^{<\om}$ such that $\forall$ $X\in\mc{X}''$ $\exists^\infty$ $s\in\dom(g_X)$ $(f(s)=g_X(s))$. We can assume that $s\subseteq f(s)$ for every $s\in T$, in particular if $S=T\cup\mrm{ran}(f)$ then $[S]=[T]$ so $S\in\mrm{Nwd}_\mrm{tr}$, and of course, $|S\cap X|=\om$ for every $X\in\mc{X}''$, contradiction.
\end{proof}

This argument in part (b) above does not work for $\ka=\om_1$, this case is still open:
\begin{que}\label{nwdluzinque}
Does the existence of an $\mrm{Nwd}$-Luzin set imply the existence of a Luzin set?
\end{que}

\begin{model} ((Q4) for $\mrm{Nwd}$.)
It is consistent that there are  ($\mrm{Nwd}$-)Luzin sets but there are no $\mrm{Nwd}$-Luzin sets of size $\mrm{cov}^*(\mrm{Nwd})=\om_2=\mf{c}$.

Assume $\mrm{add}(\mc{M}) = \omega_2 = \mf{c}$ in the ground model $V$ and add $\omega_1$ many Cohen reals, this forcing notion will be denoted by $\mbb{C}(\om_1)$. Then in $V^{\mbb{C}(\om_1)}$ there is a Luzin set so also an $\mrm{Nwd}$-Luzin set of size $\omega_1$. By the same Fubini type argument we used in Model \ref{modelq4ed} it is easy to show that  $\mrm{cov}^*(\mrm{Nwd})=\mrm{cov}(\mc{M}) =\omega_2$ is preserved.

We still have to prove that there are no large $\mrm{Nwd}$-Luzin sets. Another Fubini type argument shows that $\mrm{add}(\mc{M}) = \omega_2$ holds in all intermediate extensions (see \cite[Thm. 3.3.16]{BaJu}). This argument only works for the intermediate
extensions, of course, $\mrm{add}(\mc{M}) = \mrm{non}(\mc{M}) = \omega_1$ in $V^{\mbb{C}(\om_1)}$. A fortiori, $\mrm{non}(\mc{M}) = \mf{b} = \omega_2$ in all intermediate extensions, hence there are no $\mrm{Nwd}$-Luzin sets in these models (by Theorem \ref{bnwdluzin} (a)), therefore there are no large $\mrm{Nwd}$-Luzin sets in $V^{\mbb{C}(\om_1)}$.
\end{model}

\section{Around $\mc{ED}_\mrm{fin}$}\label{edfin}

In this section we briefly study one ideal, $\mc{ED}_\mrm{fin}$ from the ``next layer'' of ideals in the diagram of KB-reducibility. We know (see e.g. \cite{Meza}) that  $\mrm{add}^*(\mc{ED}_\mrm{fin})=\om$, $\mrm{cof}^*(\mc{ED}_\mrm{fin})=\mf{c}$, $\mrm{cov}(\mc{M})=\min\{\mf{d},\mrm{non}^*(\mc{ED}_\mrm{fin})\}$, $\mrm{non}(\mc{M})=\max\{\mrm{cov}^*(\mc{ED}_\mrm{fin}),\mf{b}\}$, $\mf{s}\leq \mrm{cov}^*(\mc{ED}_\mrm{fin})$ and $\mrm{non}^*(\mc{ED}_\mrm{fin})\leq\mf{r}$ where
    \begin{align*}
    \mf{s} & =\min\big\{|\mc{H}|:\mc{H}\subseteq [\om]^\om\;\text{and}\;\forall\;X\in [\om]^\om\;\exists\;H\in\mc{H}\;|X\cap H|=|X\setminus H|=\om\big\}\\
    \mf{r} & = \min\big\{|\mc{O}|:\mc{O}\subseteq [\om]^\om\;\text{and}\;\forall\;H\in [\om]^\om\;\exists\;O\in\mc{O}\;\big(|O\cap H|<\om\;\text{or}\;|O\setminus H|<\om\big)\big\}.
    \end{align*}

We also know that $\mc{ED}_\mrm{fin}$ does not contain cotowers. The remaining reasonable questions concerning $\mc{ED}_\mrm{fin}$ are the following: (Q2)-(Q4) from the previous section and

\begin{itemize}
\item[(Q1')] Does there exist an $\mc{ED}_\mrm{fin}$-Luzin set in $\mrm{ZFC}$? If no, is it possible that $\mrm{non}^*(\mc{ED}_\mrm{fin})=\om_1$, $\mrm{cov}^*(\mc{ED}_\mrm{fin})=\mf{c}$, and there are no $\mc{ED}_\mrm{fin}$-Luzin sets?
\item[(Q5)] Is it consistent that $\om_1<\mrm{non}^*(\mc{ED}_\mrm{fin})<\mrm{cov}^*(\mc{ED}_\mrm{fin})$?
\item[(Q6)] Is it consistent that $[\Delta]^\om$ is $\mc{ED}_\mrm{fin}$-inaccessible but $\mrm{cov}^*(\mc{ED}_\mrm{fin})=\om_1$?
\item[(Q7)] Is it consistent that $\mrm{non}^*(\mc{ED}_\mrm{fin})=\om_1<\mf{c}$ but $[\Delta]^\om$ is $\mc{ED}_\mrm{fin}$-accessible?
\end{itemize}

The second part of (Q1') is open. Notice that Proposition \ref{noedluzin} and Fact \ref{KBcons} (b) imply that consistently there are no $\mc{ED}_\mrm{fin}$-Luzin families, and under $\mrm{MA}_{\om_1}+\mf{c}=\om_2$, we have $\om_2=\mf{s}\leq\mrm{cov}^*(\mc{ED}_\mrm{fin})=\om_2$ hence (Q1) is answered easily. However, under $\mrm{MA}_{\om_1}$ we have $\mrm{non}^*(\mc{ED}_\mrm{fin})\geq \mrm{cov}(\mc{M})>\om_1$.

Answering (Q2), we know that consistently there are large $\mc{ED}_\mrm{fin}$-Luzin sets because of large Sierpi\'nski sets.

Answering (Q3) and (Q7), take Model \ref{q3many}, more precisely, at cofinally many stages force with $\mbb{M}(\mc{ED}_\mrm{fin}^*)$ as well, and apply Remark \ref{usefulrem} (2) and (2').

Answering (Q4), consider Model \ref{modelq4ed}: Applying Fact \ref{KBcons} (a) and (b) in the extension we know that $\mrm{cov}^*(\mc{ED}_\mrm{fin})\geq\mrm{cov}^*(\mrm{tr}(\mc{N}))\geq\mrm{cov}(\mc{N})=\om_2(=\mf{c})$; there is an $\mc{ED}_\mrm{fin}$-Luzin set because there is a $\mrm{tr}(\mc{N})$-Luzin set; and finally, there are no large $\mc{ED}_\mrm{fin}$-Luzin sets in this model because there are no large $\mc{ED}$-Luzin sets either.

(Q5) holds in any model of $\om_1<\mrm{cov}(\mc{M})=\mrm{non}(\mc{N})<\mrm{cov}(\mc{N})=\mrm{non}(\mc{M})$ (for such a model see \cite{shattered}) because $\cov(\mc{M})\leq\non^*(\mc{ED}_\mrm{fin})\leq \non^*(\mrm{tr}(\mc{N}))\leq\non(\mc{N})$ and $\cov(\mc{N})\leq\cov^*(\mrm{tr}(\mc{N}))\leq\cov^*(\mc{ED}_\mrm{fin})\leq\non(\mc{M})$.

\section{Analytic P-ideals}\label{analpindiag}

The ``last layer'' of ideals in the KB-reducibility diagram consists of three tall analytic P-ideals: $\mc{I}_{1/n}$, $\mrm{tr}(\mc{N})$, and $\mc{Z}$. We know that the following (in)equalities hold (where every second line, the ``prime'' version, is just the dual (in)equality):
\begin{itemize}
\item[(a)] $\mrm{add}^*(\mc{I}_{1/n})=\mrm{add}^*(\mrm{tr}(\mc{N}))=\mrm{add}^*(\mc{Z})=\mrm{add}(\mc{N})$;
\item[(a')] $\mrm{cof}^*(\mc{I}_{1/n})=\mrm{cof}^*(\mrm{tr}(\mc{N}))=\mrm{cof}^*(\mc{Z})=\mrm{cof}(\mc{N})$;
\item[(b)] $\mrm{cov}(\mc{M})\leq\mrm{non}^*(\mc{I}_{1/n})\leq\mrm{non}^*(\mrm{tr}(\mc{N}))\leq\mrm{non}^*(\mc{Z})$;
\item[(b')] $\mrm{non}(\mc{M})\geq\mrm{cov}^*(\mc{I}_{1/n})\geq\mrm{cov}^*(\mrm{tr}(\mc{N}))\geq\mrm{cov}^*(\mc{Z})$;
\item[(c)] $\mrm{non}^*(\mrm{tr}(\mc{N}))\leq\mrm{non}(\mc{N})$;
\item[(c')] $\mrm{cov}^*(\mrm{tr}(\mc{N}))\geq\mrm{cov}(\mc{N})$;
\item[(d)] $\min\{\mrm{cov}(\mc{N}),\mf{b}\}\leq\mrm{cov}^*(\mc{Z})\leq\max\{\mf{b},\mrm{non}(\mc{N})\}$;
\item[(d')] $\min\{\mrm{cov}(\mc{N}),\mf{d}\}\leq\mrm{non}^*(\mc{Z})\leq\max\{\mf{d},\mrm{non}(\mc{N})\}$;
\item[(e)] $\mrm{cov}^*(\mc{Z})\leq \mf{d}$;
\item[(e')]  $\mf{b}\leq\mrm{non}^*(\mc{Z})$.
\end{itemize}

(a), (a'): $\mrm{add}^*(\mrm{tr}(\mc{N}))=\mrm{add}(\mc{N})$ and $\mrm{cof}^*(\mrm{tr}(\mc{N}))=\mrm{cof}(\mc{N})$ are trivial.
See \cite[524H]{Fr51} and \cite[Thm. 4.2.4, Cor. 4.2.5]{mythesis} for $\mrm{add}^*(\mc{I}_{1/n})=\mrm{add}(\mc{N})$ and $\mrm{cof}^*(\mc{I}_{1/n})=\mrm{cof}(\mc{N})$. And see \cite[526G]{Fr51} for $\mrm{add}^*(\mc{Z})=\mrm{add}(\mc{N})$ and $\mrm{cof}^*(\mc{Z})=\mrm{cof}(\mc{N})$.

\smallskip
(b), (b'): See \cite[Thm. 3.7]{cardinvanalp} for the first inequality. The second and third inequalities are consequences of Fact \ref{KBcons} (a).

\smallskip
(c), (c'): Trivial.

\smallskip
(d), (d'): See \cite[Thm. 3.10, 3.12]{cardinvanalp}.

\smallskip
(e), (e'): See \cite{dilipshelah}.

\smallskip
Let us summarize (almost) all the above inequalities in a diagram. Two inequalities, $\mrm{cov}^*(\mc{Z})\leq\max\{\mf{b},\mrm{non}(\mc{N})\}$ and its dual $\min\{\mrm{cov}(\mc{N}),\mf{d}\}\leq\mrm{non}^*(\mc{Z})$ are missing, $a,a',b,b'$ next to dotted arrows indicate min-max results , and of course $\mc{I}=\mc{I}_{1/n},\mrm{tr}(\mc{N})$, or $\mc{Z}$.

{\footnotesize
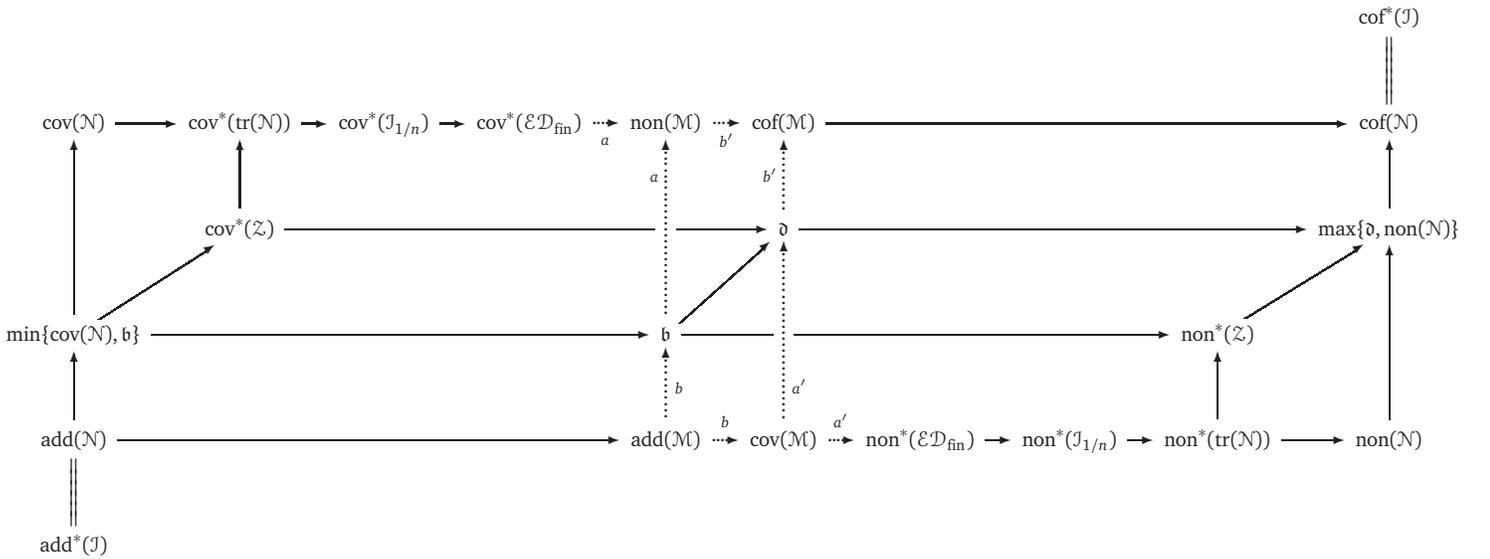
\begin{diagram}
&&&&&&&&&&&&&&&&&& \mrm{cof}^*(\mc{I})\\
&&&&&&&&&&&&&&&&&& \uEqv\\
\mrm{cov}(\mc{N}) & \rTo & \mrm{cov}^*(\mrm{tr}(\mc{N})) & \rTo & \mrm{cov}^*(\mc{I}_{1/n}) & \rTo &\cov^*(\mc{ED}_\mrm{fin}) & \rDotsto_a & \mrm{non}(\mc{M}) & \rDotsto_{b'} & \mrm{cof}(\mc{M}) & & & & \rTo & & & & \mrm{cof}(\mc{N})\\
 & & \uTo & & & & & & \uDotsto^a & & \uDotsto^{b'} & & & & & & &  &  \uTo   \\
 & & \mrm{cov}^*(\mc{Z}) & & &\hLine & & & \VonH &  \rTo& \mf{d} & & & & \rTo & & & & \max\{\mf{d},\mrm{non}(\mc{N})\} \\
\uTo & \ruTo& & & & && &  & \ruTo & & & & && & & \ruTo& \\
\min\{\mrm{cov}(\mc{N}),\mf{b}\} &  &  & &\rTo & &  & & \mf{b} & \hLine & \VonH & && \rTo && & \mrm{non}^*(\mc{Z}) &  & \uTo \\
\uTo & & & & & && & \uDotsto_b & &  \uDotsto_{a'} & & && & & \uTo & & \\
\mrm{add}(\mc{N}) & & & &\rTo && & & \mrm{add}(\mc{M}) & \rDotsto^b & \mrm{cov}(\mc{M}) & \rDotsto^{a'} & \non^*(\mc{ED}_\mrm{fin}) &\rTo & \mrm{non}^*(\mc{I}_{1/n}) & \rTo & \mrm{non}^*(\mrm{tr}(\mc{N})) & \rTo & \mrm{non}(\mc{N})\\
\uEqv &&&&&&&&&&&&&&&&\\
\mrm{add}^*(\mc{I}) &&&&&&&&&&&&&&&&
\end{diagram}}

\smallskip
Concerning these ideals, beyond (Q1')-(Q2)-(Q6), the following questions are interesting:

\begin{itemize}
\item[(Q7')] Is it consistent that $[\om]^\om$ is $\mc{I}$-accessible, yet $\mrm{non}^*(\mc{I})<\mf{c}$ and there are no cotowers in $\mc{I}$ (in particular, $\mrm{add}^*(\mc{I})<\mrm{cov}^*(\mc{I})$)?
\item[(Q8)] Is it consistent that there are no cotowers in $\mc{I}$ but $\mrm{non}^*(\mc{I})=\mf{c}$?
\item[(Q9)] Is it consistent that $\mrm{non}^*(\mc{I})<\mf{c}$ but $\mrm{add}^*(\mc{I})=\mrm{cov}^*(\mc{I})$?
\item[(Q10)] Is it consistent that there is a cotower in $\mc{I}$, $\mrm{non}^*(\mc{I})=\mf{c}$, and $\mrm{add}^*(\mc{I})<\mrm{cov}^*(\mc{I})$?
\item[(Q11)] Is it consistent that there is a cotower in $\mc{I}$, $\mrm{non}^*(\mc{I})<\mf{c}$, and $\mrm{add}^*(\mc{I})<\mrm{cov}^*(\mc{I})$?
\end{itemize}

Proposition \ref{noedluzin} and Fact \ref{KBcons} (b) imply the consistency that there are no $\mc{I}$-Luzin sets (even with $\mrm{cov}^*(\mc{I})=\mf{c}$ answering (Q1)) but in this model $\mrm{non}^*(\mc{I})=\mf{c}$ for every tall analytic P-ideal $\mc{I}$. Alternatively, applying Remark \ref{usefulrem} (1), we can take any nontrivial $\om_2$ stage finite support iteration over a model of $\mrm{CH}$, then in the extension $\mrm{non}^*(\mc{I})=\om_2$ and there are no $\mc{I}$-Luzin sets.

(Q1'), (Q2), (Q4), and (Q5) for $\mc{I}=\mc{Z}$ seem to be very difficult because in such models  $\mathrm{non}^*(\mc{Z})<\mathrm{cov}^*(\mc{Z})$ would hold, and the consistency of this strict inequality is still an open problem. (Q1') is open for $\mc{I}_{1/n}$ and $\mrm{tr}(\mc{N})$ too.

(Q2) for $\mrm{tr}(\mc{N})$ (and hence for $\mc{I}_{1/n}$) has already been answered: If there is a large Sierpi\'nski set, then there is a large $\mrm{tr}(\mc{N})$-Luzin set as well.

(Q9) is easy, consider the Sacks model. A bit more interesting model answers (Q3) and (Q9) at the same time for an arbitrary tall analytic P-ideal $\mc{I}$: Consider the appropriate version of the (dual) Model \ref{q3many}, that is, let $V\models\mrm{add}(\mc{N})=\om_2$, force with the $\om_1$ stage finite support iteration of $\mbb{M}(\mc{I}^*)$. Applying Remark \ref{usefulrem} (2'), in the extension there is an $\mc{I}$-Luzin set and $\mrm{non}^*(\mc{I})=\mrm{cov}^*(\mc{I})=\om_1<\mf{c}$. The next lemma implies that $\mrm{cof}^*(\mc{I})=\mf{c}$ also holds in this extension.

\begin{lem}
If $\mrm{add}^*(\mc{I})=\ka$ and $\PP$ is $\sigma$-centered, then $V^\PP\models\mrm{cof}^*(\mc{I})\geq\ka$.
\end{lem}
\begin{proof}
We will prove a bit more, namely that
\[ V^\PP\models\min\big\{|\mc{A}|:\mc{A}\subseteq\mc{I}:\forall\;B\in\mc{I}\cap V\;\exists\;A\in\mc{A}\;B\subseteq A\big\}\geq\ka.\]
Let $\mc{I}=\mrm{Exh}(\varphi)$ for some lsc submeasure $\varphi$, we can assume that $\|\om\|_\varphi=1$. Let $\PP=\bigcup_{n\in\om}C_n$ where $C_n$ is centered for every $n$, and assume on the contrary that there is a family $\{\dot{A}_\al:\al<\lam\}$ of $\PP$-names for elements of $\mc{I}$ for some $\lam<\ka$ such that $V^\PP\models\forall$ $B\in\mc{I}\cap V$ $\exists$ $\al<\lam$ $B\subseteq \dot{A}_\al$.

For every $X\in\mc{I}$ we can find a $p_X\in \PP$, an $\al_X<\lam$, and a $k_X\in\om$ such that $p_X\vd$``$X\subseteq \dot{A}_{\al_X}$ and $\varphi(\dot{A}_{\al_X}\setminus k_X)\leq 1/2$''. Furthermore, for every $n,k\in\om$ and $\al<\lam$ define $\mc{X}_{n,\al,k}=\{X\in\mc{I}:p_X\in C_n,\al_X=\al,$ and $k_X=k\}$.

We claim that $\varphi(\bigcup\mc{X}_{n,\al,k}\setminus k)\leq 1/2$. Otherwise there is a $K>k$ such that $\varphi(\bigcup\mc{X}_{n,\al,k}\cap [k,K))>1/2$ and there are finitely many $X_0,X_1,\dots,X_{m-1}\in\mc{X}_{n,\al,k}$ such that $\bigcup\mc{X}_{n,\al,k}\cap [k,K)=\bigcup_{i<m}X_i\cap [k,K)$. Now if $p$ is a common extension of $p_{X_0},p_{X_1},\dots,p_{X_{m-1}}$, then
$p\vd 1/2<\varphi(\bigcup_{i<m}X_i\cap [k,K))\leq \varphi(\dot{A}_\al\setminus k)$, a contradiction.

In particular, $\om\setminus\bigcup\mc{X}_{n,\al,k}$ is infinite for every $n,\al$, and $k$, hence it has an infinite subset $B_{n,\al,k}\in\mc{I}$. If $X\in\mc{I}$ such that $B_{n,\al,k}\subseteq^* X$ for every $\al,n,k$, then $X$ cannot belong to any $\mc{X}_{n,\al,k}$, a contradiction.
\end{proof}

Answering (Q4) for $\mc{I}_{1/n}$ and $\mrm{tr}(\mc{N})$, consider the same model and explanation as in Model \ref{modelq4ed}.

Answering (Q5) for $\mc{I}_{1/n}$ and $\mrm{tr}(\mc{N})$, take any model of $\om_1<\mrm{cov}(\mc{M})=\mrm{non}(\mc{N})<\mrm{cov}(\mc{N})=\mrm{non}(\mc{M})$. For such a model see \cite{shattered}.

The answer to (Q6) is NO because, as we pointed out earlier, if $\mc{I}$ is a P-ideal, then $\mrm{cov}^*(\mc{I})=\om_1$ iff there exists of a cotower in $\mc{I}$ is height $\om_1$, in particular, $[\om]^\om$ is $\mc{I}$-accessible. This easy observation motivates the following weak version of (Q6):

\begin{itemize}
\item[(Q6')] Is it consistent that $[\om]^\om$ is $\mc{I}$-inaccessible but $\mrm{non}^*(\mc{I})\geq\mrm{cov}^*(\mc{I})$?
\end{itemize}

This is open and seems to be quite difficult.

Before dealing with (Q7') let us show that answering (Q7) is not difficult, simply consider the appropriate version of Model \ref{q3many}. (Q7') is more problematic, at this moment we can answer it for $F_\sigma$ P-ideals only:

\begin{model} ((Q7') for tall $F_\sigma$ P-ideals.)
Let $\I$ be a tall $F_\sigma$ P-ideal.
It is consistent that $\omoms$ is $\I$-accessible, there are no towers in $\I^*$, and $\add^* (\I) = \omega_1 < \non^* (\I) = \cov^* (\I) = \omega_2 < \cc = \omega_3$.

Start with a model $V$ of $\cc = \omega_3$ and consider the $\om_2$ stage finite support iteration of $\MM (\I^*)$ over $V$, call this iteration $(\PP_\alpha)_{\alpha \leq \omega_2}$.
Applying Remark \ref{usefulrem} (1') we know that $[\om]^\om$ is $\mc{I}$-accessible and $\mrm{cov}^*(\mc{I})=\mrm{non}^*(\mc{I})=\om_2$ in $V^{\PP_{\om_2}}$.  Hence it suffices to show that there are no towers in $\I^*$ (then $\add^* (\I) = \omega_1$ follows automatically).

We know that in $V^{\PP_{\om_2}}$ the only possible height of a tower is $\omega_2$ (because the cofinality of its length is between $\mrm{cov}^*(\mc{I})$ and $\mrm{non}^*(\mc{I})$). Assume $(\dot{A}_\alpha)_{\alpha < \omega_2}$ is a $\PP_{\omega_2}$-name for a $\subseteq^*$-decreasing sequence in $\mc{I}^*$ (where $\dot{A}_\al$ is a nice $\PP_{\om_2}$-name for every $\al$). Let $C=\{\al<\om_2:\dot{A}_\ga$ is a $\PP_\al$-name for every $\ga<\al\}$.
Since for every $\al$
\[  \vd_{\omega_2}``(\dot A_\ga)_{\ga < \al}\;\text{has a pseudointersection in}\; \I^*", \]
by Lemma \ref{sigmacentered} and the $\sigma$-centeredness of the remainder forcing $\PP_{\om_2}/\dot{G}_\al\in V^{\PP_\al}$,
we see that for $\al\in C$ there is a nice $\PP_\al$-name $\dot{B}_\al$ such that
\[  \vd_\al ``(\dot A_\ga)_{\ga<\al}\;\text{has a pseudointersection}\; \dot B_\al \in \I^*". \]
Let $S=\{\al\in C:\mrm{cf}(\al)=\om_1\}$ (a stationary subset of $\om_2$) and for every $\al\in S$ pick a $\ga_\al<\al$ such that $\dot B_\al$ is a $\PP_{\gamma_\al}$-name. By pressing-down (Fodor's Lemma),
there are $\gamma < \omega_2$ and a stationary set $S' \subseteq S$ such that $\gamma_\al = \gamma$ for all $\al \in S'$.
Since $\vd_\gamma \dot B_\al \in \I^*$
for $\al \in S'$, there is a $\PP_{\gamma + 1}$-name $\dot B$ (the name for the generic) such that $\vd_{\gamma + 1}\dot B \subseteq^* \dot B_\al$ for all $\al \in S'$. Then $\PP_{\omega_2}$ forces $\dot B$ to be a pseudointersection of $(\dot A_\alpha)_{\alpha < \omega_2}$.
\end{model}

Answering (Q8), simply consider the $\om_2$ stage finite support iteration of $\mbb{M}(\mc{I}^*)$ over a model of $\mrm{CH}$ (where $\mc{I}$ is any tall analytic P-ideal), then in the extension there are no cotowers in $\mc{I}$ (see Theorem \ref{first-thm}) and $\mrm{non}^*(\mc{I})=\om_2=\mf{c}$ (see Remark \ref{usefulrem} (1)).

We already dealt with (Q9).

\begin{model} ((Q10) for tall analytic P-ideals.) It is consistent that there is a cotower in $\mc{I}$, $\mrm{non}^*(\mc{I})=\mf{c}$, and $\mrm{add}^*(\mc{I})<\mrm{cov}^*(\mc{I})$.

Start with a model $V\models$``there is a tower in $\mc{I}^*$ of height $\mf{c}=2^{\om_1}=\om_2$'' (e.g. apply Theorem \ref{locmodel} or Theorem \ref{bigspec}, or simply start with a model of $\mrm{MA}_{\om_1}+\mf{c}=\om_2$), and construct an $\om_2$ stage finite support iteration of forcing notions of the form $\mbb{M}(\mc{F})$ where $\mc{F}\subseteq\mc{I}^*$ is a filter such that (a) $\chi(\mc{F})=\om_1$ for every $\mc{F}$ along the iteration, and (b) if $\PP_{\om_2}$ is the limit of this iteration (denoted by $(\PP_\al)_{\al\leq\om_2}$) and $\mc{F}$ is a subfilter of $\mc{I}^*$ generated by $\om_1$ elements in $V^{\PP_{\om_2}}$, then we forced with $\mbb{M}(\mc{F})$.

The tower from $V$ survives this iteration because if it was destroyed, then a $\PP_\al$ destroyed it for some $\al<\om_2$ but $\PP_\al$ has a dense subforcing of size $\om_1$, and hence it can not destroy a tower of height $\om_2$. Why? Let $(T_\al)_{\al<\om_2}$ be a $\subseteq^*$-decreasing sequence, $|\mbb{Q}|=\om_1$, and $\dot{X}$ be a $\mbb{Q}$-name for an infinite subset of $\om$ such that $\vd_\mbb{Q}\dot{X}\subseteq^* T_\al$ for every $\al<\om_2$. Then for every $\al$ we can pick a $q_\al\in\mbb{Q}$ and a $k_\al\in\om$ such that $q_\al\vd\dot{X}\setminus k_\al\subseteq T_\al$. Then there are a pair $(q,k)$ and an $S\in [\om_2]^{\om_2}$ such that $q_\al=q$ and $k_\al=k$ for every $\al\in S$, and hence $q\vd\dot{X}\setminus k\subseteq\bigcap\{T_\al:\al\in S\}$, in particular, $q\vd$``$\bigcap\{T_\al:\al\in S\}$ is infinite'', but then $\bigcap\{T_\al:\al\in S\}$ is infinite in $V$ as well, therefore this sequence is not a tower in $V$.

$V^{\PP_{\om_2}}\models\cov^*(\mc{I})=\om_2=\mf{c}$ is trivial, and because of the tower in $\mc{I}^*$, we know that $\cov^*(\mc{I})\leq\non^*(\mc{I})$.

To show that $\add^*(\mc{I})=\om_1$ in $V^{\PP_{\om_2}}$, notice that $V^{\PP_{\om_1}}\models\non(\mc{M})=\om_1$ (because of the Cohen reals added at cofinally many stages) and hence $V^{\PP_{\om_1}}\models\add^*(\mc{I})=\cov^*(\mc{I})=\om_1$ (because $\cov^*(\mc{I})\leq\non(\mc{M})$, see at the beginning of Section \ref{analp}). The quotient forcing $\PP_{\om_2}/\dot{G}_{\om_1}\in V^{\PP_{\om_1}}$ is $\sigma$-centered (because $\mf{c}=\om_2$) and hence we can apply Corollary \ref{preservingadd} (b).
\end{model}

The last model ((Q11) for tall analytic P-ideals) is probably the most interesting one. We will need to do some preparations. Fix an lsc submeasure $\varphi$ such that $\mc{I}=\mrm{Exh}(\varphi)$. If $\mc{T}=(T_\xi)_{\xi<\al}$ is a $\subseteq^*$-decreasing sequence in $\mc{I}^*$, then define the forcing notion $\mbb{Q}(\mc{T})$ as follows: $q\in\mbb{Q}(\mc{T})$ iff $q=(s^q,\eps^q,A^q)$ where $s^q\in 2^{<\om}$, $\eps^q$ is a positive rational, $A^q\in [\al]^{<\om}$, and \[ \varphi\Big(\om\setminus\Big( |s^q|\cup\bigcap\big\{T_\xi:\xi\in A^q\big\}\Big)\Big)<\eps^q;\] $q_1=(s^{q_1},\eps^{q_1},A^{q_1})\leq q_0=(s^{q_0},\eps^{q_0},A^{q_0})$ iff $s^{q_1}\supseteq s^{q_0}$, $\eps^{q_1}\leq\eps^{q_0}$, $A^{q_1}\supseteq A^{q_0}$, and the following two conditions hold: \begin{align*} \big\{k\in [|s^{q_0}|,|s^{q_1}|):s^{q_1}(k)=1\big\} & \subseteq\bigcap\big\{T_\xi:\xi\in A^{q_0}\big\},\\
\varphi\big(\big\{k\in[|s^{q_0}|,|s^{q_1}|):s^{q_1}(k)=0\big\}\big) & \leq\eps^{q_0}-\eps^{q_1}.
\end{align*}

It is straightforward to show that $\mbb{Q}(\mc{T})$ adds a $\dot{T}_\al\in V^{\mbb{Q}(\mc{T})}$ such that
\begin{itemize}
\item[(i)] $V^{\mbb{Q}(\mc{T})}\models\dot{T}_\al\in\mc{I}^*$,
\item[(ii)] $V^{\mbb{Q}(\mc{T})}\models$``$\dot{T}_\al$ is a pseudointersection of $(T_\xi)_{\xi<\al}$'',
\item[(iii)] $V^{\mbb{Q}(\mc{T})}\models\forall$ $X\in [\om]^\om\cap V$ $(|X\setminus \dot{T}_\al|=\om)$.
\end{itemize}
In particular, applying Corollary \ref{firstcor}, if $\mc{T}$ is a tower then $\mbb{Q}(\mc{T})$ is not $\sigma$-centered. We will also need the fact that $\mbb{Q}(\mc{T})$ is absolute in the following sense: If $V$ is a transitive model containing $\varphi\upharpoonright\mrm{Fin}$ and $\mc{T}$, then $\mbb{Q}(\mc{T})^V=\mbb{Q}(\mc{T})$.

Now build an iteration $(\PP_\al)_{\al\leq\om_2}$ based on $(\dot{\mbb{Q}}_\al)_{\al<\om_2}$ by recursion on $\al$ such that $\PP_\al$ adds $\dot{\mc{T}}_\al=(\dot{T}_\xi)_{\xi<\al}$ (let $T_0=\om$) and $\vd_{\PP_\al}\QQ_\al=\mbb{Q}(\dot{\mc{T}}_\al)$. $\PP_{\om_2}$ adds a tower in $\mc{I}^*$ of height $\om_2$. Notice that essentially $\PP_{\om_2}$ is just a new presentation of the forcing notion we defined in Theorem \ref{bigspec} if $R=\{\om_2\}$. Why do we not use that forcing notion then? We could do so but then we would have to deal with some technical difficulties in Model \ref{lastone}, namely, we would have to prove that the quotient forcings at  successor stages satisfy property (iii) regardless of the model. More precisely, working with the iterated construction defined above, if $G_\al$ is a $(V,\PP_\al)$-generic filter, $\mbb{Z}\in V[G_\al]$ is a forcing notion, $G^\al_{\al+1}\times H$ is a $(V[G_\al],\mbb{Q}(\mc{T}_\al)\times\mbb{Z})$-generic filter, and $T_\al$ is the generic pseudointersection defined from $G^\al_{\al+1}$, then \[ V[G_\al][G^\al_{\al+1}\times H]\models \forall\;X\in [\om]^\om\cap V[G_\al][H]\;(|X\setminus T_\al|=\om)\]  because $G^\al_{\al+1}$ is $(V[G_\al][H],\mbb{Q}(\mc{T}_\al)^{V[G_\al]})$-generic, $\mbb{Q}(\mc{T}_\al)^{V[G_\al]}=\mbb{Q}(\mc{T}_\al)^{V[G_\al][H]}$,
and hence we can apply (iii).

What goes wrong if we try to work with the forcing notion defined in Theorem \ref{bigspec}? We can decompose it into a chain $\mbb{I}_\al\leq_c\mbb{I}_\be$ ($\al<\be\leq\om_2$) where $\mbb{I}_\al$ adds $\mc{T}_\al=(T_\xi)_{\xi<\al}$. Now if $\Gamma_\al$ is a $(V,\mbb{I}_\al)$-generic filter,  $\mbb{I}^\al_{\al+1}=\mbb{I}_{\al+1}/\Gamma_\al$ is the quotient forcing,  $\mbb{Z}\in V[\Gamma_\al]$ is a forcing notion, $\Gamma^\al_{\al+1}\times H$ is a $(V[G_\al],\mbb{I}^\al_{\al+1}\times\mbb{Z})$-generic filter, and $E_\al$ is the generic pseudointersection defined from $\Gamma^\al_{\al+1}$, then we would need that $V[\Gamma_\al][\Gamma^\al_{\al+1}\times H]\models \forall$ $X\in [\om]^\om\cap V[\Gamma_\al][H]$ $(|X\setminus E_\al|=\om)$ which is unclear. Of course, it very probably holds but we prefer to use the clearer iterated construction than dealing with this quotient forcing.

The following lemma is somewhat surprising since $\mbb{Q}(\mc{T})$ is not necessarily $\sigma$-centered.

\begin{lem}\label{damnsigmacent} If $\om_2\leq\mf{c}$ then
$\PP_{\om_2}$ is $\sigma$-centered.
\end{lem}
\begin{proof} The end of the argument is a modification of the proof that $\mbb{C}(\om_2)$ is $\sigma$-centered if $\om_2\leq\mf{c}$ but first we will define a dense subforcing $\PP''$ of $\PP_{\om_2}$, and then apply this classical method to $\PP''$. Conditions of $\PP_{\om_2}$ are considered to be partial functions, in particular, the support of a condition is its domain.

First define $\PP'\subseteq\PP_{\om_2}$ as follows: $p\in\PP'$ iff there are sequences $s_\al\in 2^{<\om},\eps_\al\in\{$positive rationals$\}$, and $A_\al\in [\al]^{<\om}$ ($\al\in\dom(p)$) such that
\begin{itemize}
\item[(a)] $p\upharpoonright\al\vd_\al p(\al)=(s_\al,\eps_\al,A_\al)$ for every $\al\in\dom(p)$,
\item[(b)] $|s_\al|\geq |s_\be|$ for every $\al<\be$ from $\dom(p)$,
\item[(c)] $\bigcup\{A_\al:\al\in\dom(p)\}\subseteq \dom(p)$.
\end{itemize}
An easy recursive argument shows that $\PP'$ is dense in $\PP_{\om_2}$: Fix a condition $p\in\PP_{\om_2}$. Then we can define a decreasing sequence $(p_n)$ of extensions of $p$ (where $p_0=p$) such that after the $n$th step we will not modify the last $n$ values any more (i.e. the restriction of $p_n$ to the last $n$ elements of its domain is the same as the restriction of $p_{n+1}$ to the last $n$ elements of its domain). In particular, this procedure must terminate after finitely many steps (but of course, it may take much more steps than $|\dom(p)|$ because of the new elements added to the domain when taking extensions of initial segments and also because we wish to guarantee (c)). Notice that if $p\in\PP'$ and $\al_0=\min(\dom(p))$ then $p(\al_0)$ is of the form $(\cdot,\cdot,\0)$.

We define $\PP''\subseteq\PP'$ as follows:  $p=((s_\al,\eps_\al,A_\al):\al\in\dom(p))\in\PP''$ if there is an $n^p\in\om$ such that $|s_\al|=n^p$ for every $\al\in\dom(p)$. We show that $\PP''$ is dense in $\PP'$. Let $p=((s_\al,\eps_\al,A_\al):\al\in\dom(p))\in\PP'$ be arbitrary and $n=|s_{\al_0}|$ where $\al_0=\min(\dom(p))$ (in particular, $n\geq |s_\al|$ for every $\al\in\dom(p)$ because of (b)). We will define $r=((t_\al,\de_\al,A_\al):\al\in\dom(p))\in\PP''$ such that $n^r=n$ and $r\leq p$. First of all, by a recursion on $\dom(p)$, we define the sequence $(t_\al)_{\al\in\dom(p)}$: Let $t_{\al_0}=s_{\al_0}$ and if $\al>\al_0$ then let $t_\al\upharpoonright |s_\al|=s_\al$ and on $[|s_\al|,n)$ define $t_\al(k)=1$ iff $t_\xi(k)=1$ for every $\xi\in A_\al$ (here we used (c)).

\begin{claim} $\eps_\al-\varphi(\{k\in [|s_\al|,n):t_\al(k)=0\})>0$ for every $\al\in\dom(p)$.
\end{claim}
\begin{proof}[Proof of the Claim]
Assume on the contrary that $\eps_\al-\varphi(\{k\in [|s_\al|,n):t_\al(k)=0\})\leq 0$ for some $\al\in\dom(p)$, we can assume that $\al$ is minimal, that is,  $\eps_\xi-\varphi(\{k\in [|s_\xi|,n):t_\xi(k)=0\})>0$ for every $\xi\in\dom(p)\cap\al$ (notice that $\al>\al_0$). Let \[ 0<\eps<\min\big\{\eps_\xi-\varphi\big(\big\{k\in [|s_\xi|,n):t_\xi(k)=0\big\}\big):\xi\in\dom(p)\cap\al\big\}\] and fix a sequence $(\de_\xi)_{\xi\in\dom(p)\cap \al}$ of positive rationals such that $\sum\{\de_\zeta:\zeta\in A_\xi\}<\de_\xi\leq\eps$ for every $\xi\in\dom(p)\cap\al$.

It is enough to show that $r'=((t_\xi,\de_\xi,A_\xi):\xi\in\dom(p)\cap\al)\in\PP_\al$ and $r'\leq p\upharpoonright\al$. Why? Because then $r'\vd_\al (s_\al,\eps_\al,A_\al)\in\mbb{Q}(\dot{\mc{T}}_\al)$, in particular, $r'\vd_\al\varphi(\om\setminus(|s_\al|\cup\bigcap\{\dot{T}_\xi:\xi\in A_\al\}))<\eps_\al$. By the definition of the sequence $t_\xi$, $r'$ decides $\dot{T}_\xi\cap [|s_\al|,n)$ ($\xi\in A_\al$) and \[ r'\vd_\al\big\{k\in [|s_\al|,n):t_\al(k)=0\big\}= [|s_\al|,n)\setminus\bigcap\big\{\dot{T}_\xi:\xi\in A_\al\big\},\]
in particular,
\[ r'\vd_\al\big\{k\in [|s_\al|,n):t_\al(k)=0\big\}\subseteq \om\setminus\big(|s_\al|\cup\bigcap\big\{\dot{T}_\xi:\xi\in A_\al\big\}\big),\]
and hence $r'\vd_\al\varphi(\{k\in [|s_\al|,n):t_\al(k)=0\})<\eps_\al$, contradiction.

$r'\in\PP_\al$: We know that $\vd_\xi\varphi(\om\setminus(n\cup\bigcap\{\dot{T}_\zeta:\zeta\in A_\xi\}))\leq \sum_{\zeta\in A_\xi}\varphi(\om\setminus (n\cup \dot{T}_\zeta))$ for every $\xi\in\dom(p)\cap\al$. By induction on $\xi\in\dom(p)\cap(\al+1)$, we show that $r'\upharpoonright\xi=((t_\zeta,\de_\zeta,A_\zeta):\zeta\in\dom(p)\cap \xi)\in\PP_\xi$.  Clearly, $r'\upharpoonright\al_0=\0\in\PP_{\al_0}$. Now assume that we already know that $r'\upharpoonright \xi\in\PP_\xi$ for some $\xi\in\dom(p)\cap \al$. We have to show that $r'\upharpoonright\xi\vd_\xi \varphi(\om\setminus (n\cup \bigcap\{\dot{T}_\zeta:\zeta\in A_\xi))<\de_\xi$. By the definition of $\mbb{Q}(\dot{\mc{T}}_\zeta)$, we know that  $r'\upharpoonright\zeta\vd_\zeta\varphi(\om\setminus(n\cup\dot{T}_\zeta))\leq\de_\zeta$ for every $\zeta\in A_\xi$, and by the choice of the sequence $(\de_\xi)$, we obtain that
\[r'\upharpoonright\xi\vd_\xi \varphi\big(\om\setminus\big(n\cup\bigcap\big\{\dot{T}_\zeta:\zeta\in A_\xi\big\}\big)\big)\leq \sum_{\zeta\in A_\xi}\de_\zeta<\de_\xi.\]

$r'\leq p\upharpoonright\al$:  By induction on $\xi\in\dom(p)\cap(\al+1)$, we show that $r'\upharpoonright\xi\leq p\upharpoonright \xi$. The first step is trivial because $r'\upharpoonright\al_0=p\upharpoonright\al_0=\0$. Now if we already know that $r'\upharpoonright\xi\leq p\upharpoonright\xi$, then we have to check that
$r'\upharpoonright \xi\vd_\xi\{k\in [|s_\xi|,n):t_\xi(k)=1\}\subseteq\bigcap\{\dot{T}_\zeta:\zeta\in A_\xi\}$ and that $r'\upharpoonright\xi \vd_\xi\varphi(\{k\in [|s_\xi|,n):t_\xi(k)=0\})\leq \eps_\xi-\de_\xi$. The first one is given by the choice of sequence $t_\xi$, and the second follows from the choice of $\de_\xi$:
$\de_\xi\leq\eps<\eps_\xi-\varphi(\{k\in[|s_\xi|,n):t_\xi(k)=0\})$ and hence $\varphi(\{k\in[|s_\xi|,n):t_\xi(k)=0\})<\eps_\xi-\de_\xi$.
\end{proof}

Applying the claim, we can fix an
\[ 0<\eps<\min\big\{\eps_\al-\varphi\big([|s_\al|,n):t_\al(k)=0\big\}\big):\al\in\dom(p)\big\}\]
and a sequence $(\de_\al)_{\al\in\dom(p)}$ such that $\sum_{\xi\in A_\al}\de_\xi<\de_\al\leq\eps$ for every $\al\in\dom(p)$. A trivial modification of the proof of the Claim shows that $r=((t_\al,\de_\al,A_\al):\al\in\dom(p))$ is a condition and $r\leq p$ (of course, $r\in\PP''$).

\smallskip
Finally, we show that $\PP''$ is $\sigma$-centered. The point is that there is a countable family $\mc{H}$ of partitions of $\om_2$ into finitely many sets such that whenever $F\in [\om_2]^{<\om}$, there is a partition $\{H_i:i<m\}\in\mc{H}$ such that $|F\cap H_i|\leq 1$ for every $i<m$. We can construct such an $\mc{H}$ by the following way: It is enough to work on $2^\om$ because $\om_2\leq\mf{c}$, and on $2^\om$ it is very easy to define such a family: Take for example $\{\{[t]:t\in 2^n\}:n\in\om\}$ where $[t]=\{x\in 2^\om:t\subseteq x\}$.

Now let $F$ be the set of those functions $f:\om_2\to 2^{<\om}$ such that there is an $\{H_i:i<m\}\in\mc{H}$ and an $n(f)\in\om$ such that $f\upharpoonright H_i$ is constant for every $i$ and $\ran(f)\subseteq 2^{n(f)}$. For every $f\in F$ define
\[ C_f=\big\{p\in\PP'':\forall\;\al\in\dom(p)\;\,s^p_\al=f(\al)\big\}\]
where of course $p=((s^p_\al,\eps^p_\al,A^p_\al):\al\in\dom(p))$. Then $\PP''=\bigcup\{C_f:f\in F\}$. We claim that $C_f$ is centered.  Let  $\{p_k=((s^k_\al,\eps^k_\al,A^k_\al):\al\in\dom(p_k)):k<\ell\}\subseteq C_f$. We will define a common extension $p$ of these conditions. The trick is that we will manipulate the second coordinates (that is, the $\eps$'s) only. Let $n=n(f)$, $\dom(p)=\bigcup\{\dom(p_k):k<\ell\}$, and for $\al\in\dom(p)$ let $A_\al=\bigcup\{A^k_\al:\al\in\dom(p_k)\}$ and $s_\al=s^k_\al$ for some $p_k$ such that $\al\in\dom(p_k)$ ($s_\al$ does not depend on $k$ because of the definition of $C_f$).  Define $\eps=\min\{\eps^k_\al:k<\ell,\al\in\dom(p_k)\}$ and fix a sequence $(\de_\al)_{\al\in\dom(p)}$ of positive rationals such that $\sum_{\xi\in A_\al}\de_\xi<\de_\al\leq\eps$. The last thing we need to check is that $p=((s_\al,\de_\al,A^p_\al):\al\in\dom(p))$ is a condition because then $p\leq p_k$ holds trivially. In other words, we have to show by recursion on $\al$ that $p\upharpoonright\al\in\PP_\al$ and
\[ p\upharpoonright\al\vd_\al\varphi\big(\om\setminus\big(n\cup\bigcap\big\{\dot{T}_\xi:\xi\in A_\al\big\}\big)\big)<\de_\al.\]
For every $\xi\in\dom(p)\cap\al$ we know that $p\upharpoonright\xi\vd_\xi\varphi(\om\setminus (n\cup\dot{T}_\xi))<\de_\xi$ and of course
\[ \vd_\al\om\setminus\big(n\cup\bigcap\big\{\dot{T}_\xi:\xi\in A_\al\big\}\big)=\bigcup\big\{\om\setminus\big(n\cup \dot{T}_\xi\big):\xi\in A_\al\big\},\] therefore
\[ p\upharpoonright\al\vd_\al\varphi\big(\om\setminus\big(n\cup\bigcap\big\{\dot{T}_\xi:\xi\in A_\al\big)\big)\leq \sum_{\xi\in A_\al}\de_\xi< \de_\al.\qedhere\]
\end{proof}

\begin{model}\label{lastone} ((Q11) for tall analytic P-ideals.) It is consistent that there is a cotower in $\mc{I}$, $\mrm{non}^*(\mc{I})<\mf{c}$, and $\mrm{add}^*(\mc{I})<\mrm{cov}^*(\mc{I})$. Notice that in such a model $\mf{c}\geq\om_3$ because there is a tower in $\mc{I}^*$ and hence $\cov^*(\mc{I})\leq\non^*(\mc{I})$ holds.

Start with a model $V$ of $\mf{c}=\om_3$ and $\cov^*(\mc{I})=\om_1$ (e.g. add $\om_3$ Cohen reals to a model of $\mrm{CH}$, then $\cov^*(\mc{I})=\non(\mc{M})=\om_1$ in the extension). First we force with $\PP_{\om_2}$ (see above), and then over $V^{\PP_{\om_2}}$ we construct another $\om_2$ stage finite support iteration $(\mbb{R}_\al)_{\al\leq\om_2}$ based on $(\dot{\mbb{Z}}_\al)_{\al<\om_2}$. Let $G_{\om_2}$ be a $(V,\PP_{\om_2})$-generic filter and $G_\al=G_{\om_2}\cap\mbb{P}_\al$. Working in $V[G_{\om_2}]$, by recursion on $\al$ we will define $\mbb{R}_\al$ and the $\mbb{R}_\al$-name $\dot{\mbb{Z}}_{\al}$ such that $\mbb{R}_\al$ has an $\mbb{P}_\al$-name $\dot{\mbb{R}}_\al$. Assume that we already defined $\mbb{R}_\al$ for some $\al<\om_2$ and let $H_\al$ be a $(V[G_{\om_2}],\mbb{R}_\al)$-generic filter. Applying that $\mbb{P}_\al\ast\dot{\mbb{R}}_\al\leq_c \mbb{P}_{\om_2}\ast\dot{\mbb{R}}_\al$ and hence $V[G_\al][H_\al]\subseteq V[G_{\om_2}][H_\al]$, let $\dot{\mbb{Z}}_\al$ be an $\mbb{R}_\al$-name of $\mbb{M}(\mc{I}^*)^{V[G_\al][H_\al]}$.

Since $\vd_{\mbb{R}_{\al}}$``$\dot{\mbb{Z}}_\al$ is $\sigma$-centered'' holds trivially, $\mbb{R}_{\om_2}$ is an $\om_2(<\mf{c}^{V[G_{\om_2}]})$ stage finite support iteration of $\sigma$-centered forcing notions, and hence it is also $\sigma$-centered. Therefore $\PP_{\om_2}\ast\dot{\mbb{R}}_{\om_2}$ is $\sigma$-centered. In particular, applying Corollary  \ref{preservingadd} (b), $V[G_{\om_2}][H_{\om_2}]\models\add^*(\mc{I})=\om_1$.

We show that $V[G_{\om_2}][H_{\om_2}]\models$``$\non^*(\mc{I})=\cov^*(\mc{I})=\om_2$ and there is a tower in $\mc{I}^*$''. One may follow the argument on the following diagram (where $\to$ stands for $\subseteq$):
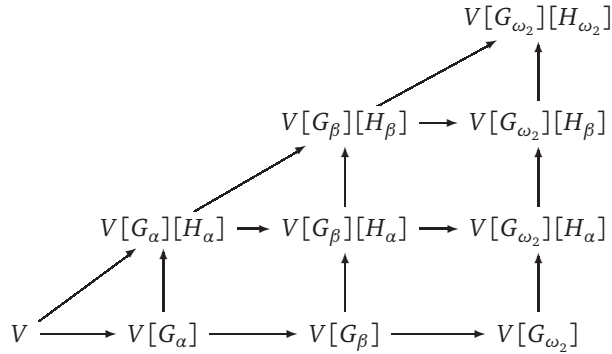
\begin{diagram}
&&&&&& V[G_{\om_2}][H_{\om_2}]\\
&&&&&\ruTo & \uTo\\
&&&& V[G_\be][H_\be] &\rTo& V[G_{\om_2}][H_\be]\\
&&&\ruTo & \uTo &&  \uTo\\
&& V[G_\al][H_\al] &\rTo & V[G_\be][H_\al] & \rTo & V[G_{\om_2}][H_\al]\\
& \ruTo & \uTo  && \uTo &&\uTo\\
V & \rTo & V[G_\al] & \rTo & V[G_\be] & \rTo & V[G_{\om_2}]
\end{diagram}

First of all, we show that the tower $(T_\al)_{\al<\om_2}$ added by $\mbb{P}_{\om_2}$ survives the second extension, in particular $V[G_{\om_2}][H_{\om_2}]\models$``$\cov^*(\mc{I})\leq\om_2\leq\non^*(\mc{I})$'' (see Fact \ref{fact1} (b)). Let $X\in [\om]^\om \cap V[G_{\om_2}][H_{\om_2}]$. Then $X\in V[G_{\om_2}][H_\al]$ for some $\al<\om_2$, i.e. $X$ has a nice $\mbb{P}_\al$-name $\dot{X}\in V[G_{\om_2}]$. There is also a $\be\in(\al,\om_2)$ such that $\dot{X}\in V[G_\be]$  because $\mbb{R}_\al\in V[G_\al]$ and $\dot{X}$ is coded by a countable sequence in $\mbb{R}_\al$ (and $\mbb{P}_{\om_2}$ satisfies the ccc), therefore $X\in V[G_\be][H_\al]$. Notice that $\mbb{P}_{\be+1}\ast\dot{\mbb{R}}_\al$ and $(\mbb{P}_\be\ast\dot{\mbb{R}}_\al)\ast\QQ_\be$ are forcing equivalent (clearly, $\QQ_\be$ can be seen as a $\mbb{P}_\be\ast\dot{\mbb{R}}_\al$-name). Why? Simply because $\dot{\mbb{R}}_\al$ and $\QQ_\be=\mbb{Q}(\dot{\mc{T}}_\be)$ are $\mbb{P}_\be$-names, $\mbb{R}_\al,\mbb{Q}(\mc{T}_\be)\in V[G_\be]$ and hence their two stage iteration does not depend on the order, it is (forced to be) forcing equivalent to their product. More precisely, if $G^\be_{\be+1}$ is the $(V[G_\be],\mbb{Q}(\mc{T}_\be))$-generic filter such that $G_{\be+1}=G_\be\ast G^\be_{\be+1}$, then $H_\al$ is $\mbb{R}_\al$-generic over $V[G_\be][G^\be_{\be+1}]$ and hence $G^\be_{\be+1}$ is $(V[G_\be][H_\al],\mbb{Q}(\mc{T}_\be))$-generic, in particular  $|Y\setminus T_\be|=\om$ for every $Y\in [\om]^\om\cap V[G_\be][H_\al]$, and so it holds for $Y=X$ too (here we used that $\mbb{Q}(\mc{T}_\be)$ is absolute).

A very similar argument shows that  $V[G_{\om_2}][H_{\om_2}]\models$``$\non^*(\mc{I})\leq\om_2\leq\cov^*(\mc{I})$''  is witnessed by the $\mbb{R}_{\om_2}$-generic sequence.
\end{model}

There is one more natural question concerning the construction of large $\mrm{tr}(\mc{N})$-Luzin families, namely, if we really need a large Sierpi\'nski set.

\begin{exa}\label{trluzinnosier}
If $V\models\mrm{non}(\mc{N})=\mf{c}=\om_2$ and we force with the $\om_1$ stage finite support iteration of $\mbb{M}(\mrm{tr}(\mc{N})^*)$, then  in the extension there is a $\mrm{tr}(\mc{N})$-Luzin set of size $\om_1$ because of Remark \ref{usefulrem} (2'), and there are no Sierpi\'nski sets because $\mrm{non}(\mc{N})=\om_2$ is preserved by $\sigma$-centered forcing notions.
\end{exa}

\section{More related questions}\label{questions}

In the previous sections we already stated some important problems: \ref{mainq} (its second part is still open, see also \ref{fsq}), \ref{fragmq}, \ref{sptower},  \ref{tmeager}, \ref{meagertowerinmat}, and \ref{nwdluzinque}. Here we collect some additional questions we found interesting.

\smallskip
We already mentioned that if we fix tall analytic P-ideals $\mc{I}$ and $\mc{J}$ such that $\mc{I}$ is ``fairly bigger'' than $\mc{J}$, then it is natural to ask whether we can force that $\mc{I}$ contains a cotower but $\mc{J}$ does not. Let us show an easy example:

\begin{exa}\label{cotinZnotintr}
In the random model (that is, in $V^{\mbb{B}(\om_2)}$ where $V$ is a model of $\mrm{GCH}$) $\mc{Z}$ contains a cotower but $\mrm{tr}(\mc{N})$ does not. Why? On the one hand, we know that in $V^{\mbb{B}(\om_2)}$ there is a $\mrm{tr}(\mc{N})$-Luzin set of size $\om_2$ (see page \pageref{largetrluzinlabel}) hence there are no cotowers in $\mrm{tr}(\mc{N})$. On the other hand, $\mc{Z}$ is random indestructible (see \cite{elekes} or \cite[Thm. 3.12]{cardinvanalp}) hence $\mrm{add}^*(\mc{Z})=\mrm{cov}^*(\mc{Z})=\om_1$ in $V^{\mbb{B}(\om_2)}$, in particular there is a cotower in $\mc{Z}$.
\end{exa}

Let us turn back to the general setting, that is, $\mc{I}$ and $\mc{J}$ are tall analytic P-ideals such that $\mc{I}$ is ``fairly bigger'' than $\mc{J}$. To find a model in which $\mc{I}$ contains a cotower but $\mc{J}$ does not, the very natural idea would be to show that $\mbb{L}(\mc{J}^*)$ (or $\mbb{M}(\mc{J}^*)$) preserves cotowers in $\mc{I}$, and then consider the $\om_2$ stage finite support iteration of $\mbb{L}(\mc{J}^*)$ over a model of $\mrm{CH}$ (and apply Theorem \ref{first-thm}). In other words, the question is whether $\cov^*(\mc{I})=\om_1$ is preserved when iterating $\mbb{M}(\mc{J}^*)$. We prove something slightly weaker and then explain why it is not enough to solve the general case. Let $\om^{\uparrow\om}$ denote the set of strictly increasing $\om\to\om$ functions.

\begin{thm}\label{reallygood}
Let $\mc{I}$ be a P-ideal, $\mc{J}$ an arbitrary ideal, and assume that $\mc{I}\nleq_{\mrm{K}}\mc{J}\upharpoonright Z$ for any $Z\in\mc{J}^+$. Then $\mbb{L}(\mc{J}^*)$ does not add a pseudointersection of $\mc{I}^*$.
More precisely, if $\vd_{\mbb{L}(\mc{J}^*)}$``$\dot{x}\in\om^{\uparrow\om}$ and $\dot{\ell}\leq \dot{x}$'' where $\dot{\ell}$ stands for the canonical $\mbb{L}(\mc{J}^*)$-generic real, then there is an $A\in\mc{I}$ such that $\forall$ $T\in\mbb{L}(\mc{J}^*)$ $\forall^\infty$ $n\in\om$ $T\nVdash \dot{x}(n)\notin A$.
\end{thm}
\begin{proof} (Sketch.)
First of all, notice that the ``more precisely'' part in the statement implies the theorem: If $\dot{x}$ is the increasing enumeration of a (rare enough) $\dot{X}\in [\om]^\om$ and $A\in\mc{I}$ is as above, then no condition $T\in\mbb{L}(\mc{J}^*)$ can force that $\dot{X}\cap A$ is finite.

Now assume that $\vd$``$\dot{x}\in \om^{\uparrow\om}$ and $\dot{\ell}\leq\dot{x}$'' and let $\dot{X}$ be an $\mbb{L}(\mc{J}^*)$-name for $\ran(\dot{x})$.

We will need a ranking argument on $\om^{<\om}$.
If $s\in\om^{<\om}$ and  $\psi=\psi(\dot{a}_0,\dots,\dot{a}_{n-1})$ is a formula of the forcing language (where $\dot{a}_i\in V^{\mbb{L}(\mc{J}^*)}$ for every $i$, of course),  then we say that

\begin{itemize}
\item $s$ {\em forces} $\psi$, $s\Vvdash\psi$, iff there is a $T\in\mbb{L}(\mc{J}^*)$ with $\mrm{stem}(T)=s$ such that $T\vd\psi$.
\item $s$ {\em favors} $\psi$ iff $s\not\Vvdash\neg\psi$, i.e. $\forall$ $T\in\mbb{L}(\mc{J}^*)$ $(\mrm{stem}(T)=s\rightarrow\exists$ $T'\leq T$ $T'\vd\psi)$.
\end{itemize}
Clearly, if $s\Vvdash\psi$ then $s$ favors $\psi$; and there are three possibilities: $s$ forces $\psi$, $s$ forces $\neg\psi$, or $s$ favors both $\psi$ and $\neg\psi$.
We define the $m$-{\em rank} on $\om^{<\om}$ for every $m\in\om$ as follows:
\begin{itemize}
\item[(i)] $\varrho_m(s)=0$ iff there is a $k\in\om$ such that $s$ favors $\dot{x}(m)=k$;
\item[(ii)] $\varrho_m(s)=\al>0$ iff $\varrho_m(s)\not<\al$ and $\{n:\varrho_m(s^\frown(n))<\al\}\in\mc{J}^+$.
\end{itemize}
It is straightforward to see that $\mrm{dom}(\varrho_m)=\om^{<\om}$ for every $m$.
We will need that $\varrho_m(s)\geq 1$ for every $s\in\om^{<\om}$ and $m\geq |s|$:
If $T_k\in\mbb{L}(\mc{J}^*)$ with $\mrm{stem}(T_k)=s$ and $\mrm{ext}_{T_k}(s)\subseteq\om\setminus (k+1)$ then $T_k\vd\dot{x}(m)\geq \dot{\ell}(m)> k$, in particular no extension $T'$ of $T_k$ can force $\dot{x}(m)=k$. In other words, $s$ cannot favor $\dot{x}(m)=k$ for any $k$, i.e. $\varrho_m(s)\ne 0$.

\smallskip
If $\varrho_m(s)=1$ then $Y_{s,m}=\{n:\varrho_m(s^\frown(n))=0\}\in\mc{J}^+$ and we can define a function $f_{s,m}:Y_{s,m}\to\om$ such that $s^\frown(n)$ favors $\dot{x}(m)=f_{s,m}(n)$, i.e.
\[ \forall\;T\in\mbb{L}(\mc{J}^*)\;\big(\mrm{stem}(T)=s^\frown(n)\rightarrow\exists\; T'\leq T\;T'\vd\dot{x}(m)=f_{s,m}(n)\big).\]
Notice that
$f_{s,m}^{-1}(\{k\})\in\mc{J}$ for every $k$: $\varrho_m(s)\ne 0$, i.e. for every $k$ there is a $T_k\in\mbb{L}(\mc{J}^*)$ with $\mrm{stem}(T_k)=s$ such that $T_k\vd\dot{x}(m)\ne k$. If $f_{s,m}^{-1}(\{k\})\in\mc{J}^+$ for some $k$, $n\in \mrm{ext}_{T_k}(s)\cap f_{s,m}^{-1}(\{k\})$, and $T=T_k\upharpoonright s^\frown(n)=\{t\in T_k:t\subseteq s^\frown(n)$ or $t\supseteq s^\frown(n)\}$, then applying that $s^\frown(n)$ favors $\dot{x}(m)=f_{s,m}(n)=k$, there is a $T'\leq T$ such that $T'\vd\dot{x}(m)=k$. A contradiction because $T'\leq T_k$.

\smallskip
Applying our assumption on non-Kat\v{e}tov-reducibility, for every $s$ and $m$ with $\varrho_m(s)=1$, there is an $A_{s,m}\in\mc{I}$ such that $f_{s,m}^{-1}(A_{s,m})\in\mc{J}^+$. We also know that $\mc{I}$ is a P-ideal, let $A\in\mc{I}$ such that $A_{s,m}\subseteq^* A$ for every such $s,m$.

\smallskip
To finish the proof, we show that if $T\in\mbb{L}(\mc{J}^*)$, $\mrm{stem}(T)=s$, and $m\geq |s|$ (hence $\varrho_m(s)\geq 1$), then there  is a $T'\leq T$ such that $T'\vd\dot{x}(m)\in A$. It is easy to see that there is a $t\in T$, $t\supseteq s$ with $\varrho_m(t)=1$. Therefore $Y_{t,m}$, $f_{t,m}$, and $A_{t,m}$ are defined. $f^{-1}_{t,m}(A_{t,m}\cap A)\in\mc{J}^+$ because $f^{-1}_{t,m}(A_{t,m})\in\mc{J}^+$ and $|A_{t,m}\setminus A|<\om$ hence $f^{-1}_{t,m}(A_{t,m}\setminus A)\in\mc{J}$.

Pick an $n\in \mrm{ext}_T(t)\cap f^{-1}_{t,m}(A_{t,m}\cap A)$, and let $k=f_{t,m}(n)\in A_{t,m}\cap A$. Then by the definition of $Y_{t,m}$ and $f_{t,m}$, we know that $t^\frown (n)$ favors $\dot{x}(m)=k$ hence we can find a $T'\leq T\upharpoonright t^\frown (n)\leq T$ which forces $\dot{x}(m)=k\in A$.
\end{proof}

Why is this theorem not enough to solve the general case? Assume $\mrm{CH}$ holds in $V$ and fix a cofinal cotower $\mc{T}$ in $\mc{I}$. Then $\mc{T}$ remains a cotower in $V^{\mbb{L}(\mc{J}^*)}$ but it is not cofinal in $\mc{I}$ anymore. Therefore if we would like to apply the previous theorem at the successor stage, then we would need that in the extension the ideal $\mrm{id}(\mc{T})$ generated by $\mc{T}$ is not Kat\v{e}tov-below $\mc{J}\upharpoonright Z$ for every $Z\in\mc{J}^+$ but this is unclear, we probably need a stronger version of the theorem, for example:

\begin{que}
Let $\mc{I}$ be a P-ideal, $\mc{J}$ be a tall analytic P-ideal, and assume that $\mc{I}\nleq_\mrm{K}\mc{J}\upharpoonright Z$ for every $Z\in\mc{J}^+$ (or any similar reasonable formalization that $\mc{I}$ is ``bigger'' than $\mc{J}$). Does it imply that $\mbb{L}(\mc{J}^*)$ preserves this property, that is, $V^{\mbb{L}(\mc{J}^*)}\models$``$\mrm{id}(\mc{I}) \nleq_\mrm{K}\mc{J}\upharpoonright Z$ for every $Z\in\mc{J}^+$''? ($\mc{J}$ is definable, it refers to the family $\subseteq\mc{P}(\om)$ given by this defintion in the extension.)
\end{que}

\begin{rem}
It seems that if $\mc{I}$ is an $F_\sigma$ ideal, then in the iterated $\mbb{M}(\mc{I}^*)$ model ($\om_2$ stage with finite support over a model of $\mrm{CH}$) we have $\mrm{cov}^*(\mc{Z})=\om_1$ and hence there is a cotower in $\mc{Z}$ (and of course, there are no cotowers in $\mc{I}$). Why? In this model $\mf{b}=\mf{s}=\om_1$ (see \cite{jorgmob} and \cite{judahshelah}), and D. Raghavan recently claimed that $\mrm{cov}^*(\mc{Z})\leq\min\{\mf{b},\mf{s}\}$ holds in $\mrm{ZFC}$. Unfortunately, Raghavan's result is still unpublished.
\end{rem}

\begin{que}
Assume that $\mc{F}$ is filter of measure zero. Does $\mc{F}^+$ contain towers? (Though, the statement ``All meager filters are null.'' is independent of $\mrm{ZFC}$, we know that there are nonmeager null filters in $\mrm{ZFC}$, see \cite{bgjs} and \cite{Tala}.)
\end{que}

Theorem \ref{seven} raises many interesting questions. Let us call an ultrafilter KMZ (Kunen-Medini-Zdomskyy) if every subset either has a pseudointersection or has a countable subset with no pseudointersection in the ultrafilter. Let us call an ultrafilter PSM (``P Subfilters are Meager'') if it does not contain non-meager P-subfiletrs. It is easy to see that KMZ implies PSM, we know (see \cite{seven}) that the existence of PSM ultrafilters is independent of $\mrm{ZFC}$.

\begin{que} Does the existence of PSM ultrafilters imply the existence of KMZ ultrafilters?
\end{que}

\begin{que}
Is $\mrm{CH}$ enough to construct a PSM (or even KMZ) ultrafilter?
\end{que}

\begin{que}
Is there any reasonable characterisation of those (Borel) ideals which are Kat\v{e}tov(-Blass) below $\mc{G}_\mrm{fc}$? Is $\mc{W}\leq_\mrm{K}\mc{G}_\mrm{fc}$?
\end{que}

\begin{que}
What can we say about the possible cuts of the diagram summarizing implications between existence of $\mc{I}$-Luzin familes, inequalities between cardinal invariants, and existence of cotowers in $\mc{I}$ in the case of $\mc{I}=\mc{W}$, $\mc{G}_{\mrm{fc}}$, and $\mc{Z}_u$?
\end{que}

In the next few questions we summarize the main open problems from sections \ref{countnon}, \ref{edfin}, and \ref{analpindiag}.
\begin{que} ((Q4) for $\mc{S}$) Is it consistent that there are $\mc{S}$-Luzin sets but there are no $\mc{S}$-Luzin sets of size $\mrm{cov}^*(\mc{S})=\om_2=\mf{c}$?
\end{que}

\begin{que} ((Q1') for $\mc{ED}_\mrm{fin}$) Is it consistent that $\mrm{non}^*(\mc{ED}_\mrm{fin})=\om_1$,  $\mrm{cov}^*(\mc{ED}_\mrm{fin})=\mf{c}=\om_2$, and there are no $\mc{ED}_\mrm{fin}$-Luzin sets?
\end{que}

\begin{que} ((Q6) for $\mc{ED}_\mrm{fin}$)
Is it consistent that $[\Delta]^\om$ is $\mc{ED}_\mrm{fin}$-inaccessible but $\mrm{cov}^*(\mc{ED}_\mrm{fin})=\om_1$?
\end{que}

\begin{que} ((Q1') for tall analytic P-ideals) Let $\mc{I}=\mc{I}_{1/n},\mrm{tr}(\mc{N})$, or $\mc{Z}$ (or any tall analytic P-ideal). Is it consistent that there are no $\mc{I}$-Luzin sets and $\mrm{non}^*(\mc{I})=\om_1<\mrm{cov}^*(\mc{I})=\mf{c}$?
\end{que}

\begin{que} ((Q6') for tall analytic P-ideals) Is it consistent that $[\om]^\om$ is $\mc{I}$-inaccessible but $\mrm{non}^*(\mc{I})\geq\mrm{cov}^*(\mc{I})$?
\end{que}

\begin{que} ((Q7') for tall analytic P-ideals) Let $\mc{I}=\mrm{tr}(\mc{N}),\mc{Z}$ or an arbitrary non $F_\sigma$ tall analytic P-ideal.
Is it consistent that $[\om]^\om$ is $\mc{I}$-accessible, yet $\mrm{non}^*(\mc{I})<\mf{c}$ and there are no cotowers in $\mc{I}$ (in particular, $\mrm{add}^*(\mc{I})<\mrm{cov}^*(\mc{I})$)?
\end{que}

\begin{que}
Let $\mc{I}\leq_\mrm{KB}\mc{J}$ be two of our examples from the diagram of KB-reducibilities. What can we say about the cuts of the other diagram concerning $\mc{I}$ and $\mc{J}$ at the same time? More precisely, for instance one can ask whether it is consistent that there is an $\mc{ED}_\mrm{fin}$-Luzin family of size $\om_2$ but $\mc{I}_{1/n}$ contains a cotower, clearly many questions of the same type can be studied.
\end{que}

We already know that the existence of a $\mrm{tr}(\mc{N})$-Luzin family does not imply the existence of a Sierpi\'nski (i.e. $\mc{N}$-Luzin) set (see Example  \ref{trluzinnosier}). However, in our example the $\mrm{tr}(\mc{N})$-Luzin set was of size $\om_1$.

\begin{que} Is it consistent that there is a $\mrm{tr}(\mc{N})$-Luzin family of size $\om_2$ but there are no (large, i.e. of size $\geq\om_2$) Sierpi\'nski sets?
\end{que}

\end{document}